\documentclass{amsart}

\usepackage{epsfig}

\usepackage{amscd,amssymb,amsmath}
\usepackage[all,web]{xy}

\def\comment#1{{\sf{[#1]}}}

\def\Z{{\mathbb Z}}
\def\Q{{\mathbb Q}}
\def\N{{\mathbb N}}
\def\R{{\mathbb R}}
\def\C{{\mathbb C}}
\def\P{{\mathbb P}}
\def\A{{\mathbb A}}

\def\bD{{\mathbb D}}
\def\bE{{\mathbb E}}

\def\H{{\mathbb H}}
\def\L{{\mathbb L}}
\def\OO{{\mathbb O}}
\def\bS{{\mathbb S}}
\def\V{{\mathbb V}}
\def\W{{\mathbb W}}

\def\cA{{\mathcal A}}
\def\cC{{\mathcal C}}
\def\cD{{\mathcal D}}
\def\E{{\mathcal E}}
\def\Fr{{\mathcal F}}
\def\cG{{\mathcal G}}
\def\cH{{\mathcal H}}

\def\K{{\mathcal K}}
\def\cL{{\mathcal L}}
\def\M{{\mathcal M}}
\def\cN{{\mathcal N}}
\def\cO{{\mathcal O}}
\def\cP{{\mathcal P}}
\def\cS{{\mathcal S}}
\def\U{{\mathcal U}}
\def\cV{{\mathcal V}}

\def\f{{\mathfrak f}}
\def\g{{\mathfrak g}}
\def\h{{\mathfrak h}}

\def\m{{\mathfrak m}}
\def\n{{\mathfrak n}}
\def\p{{\mathfrak p}}
\def\r{{\mathfrak r}}

\def\u{{\mathfrak u}}

\def\B{{\mathfrak B}}
\def\fM{{\mathfrak M}}

\def\bK{{\mathbf K}}

\def\bA{{\pmb \cA}}
\def\bcG{{\pmb \cG}}

\def\bg{{\pmb\g}}
\def\bp{{\pmb\p}}
\def\bu{\pmb{\u}}
\def\bubar{\overline{\bu}}

\def\a{\mathbf{a}}
\def\b{\mathbf{b}}
\def\e{\mathbf{e}}

\def\v{\mathbf{v}}
\def\bh{\mathbf{h}}
\def\bw{\mathbf{w}}

\def\br{{\sf r}}

\def\w{{\omega}}
\def\D{{\Delta}}
\def\G{{\Gamma}}

\def\etilde{{\tilde{\e}}}

\def\xtilde{\tilde{x}}
\def\ytilde{\tilde{y}}
\def\ztilde{\tilde{z}}
\def\gtilde{\tilde{\g}}

\def\Ftilde{{\widetilde{F}}}
\def\Hbar{{\overline{\cH}}}	%% canonical extension of \cH

\def\Vbar{{\overline{\cV}}}

\def\Xbar{{\overline{X}}}

\def\Ghat{{\widehat{\G}}}
\def\cGhat{{\widehat{\cG}}}
\def\Uhat{{\widehat{\U}}}
\def\uhat{{\widehat{\u}}}

\def\cGtilde{{\widetilde{\cG}}}

\def\Omegahat{{\widehat{\Omega}}}

\def\ghat{\widehat{\g}}
\def\uhat{\widehat{\u}}
\def\xhat{{\widehat{x}}}

\def\rhotilde{{\tilde{\rho}}}
\def\rhohat{{\hat{\rho}}}

\def\phihat{{\hat{\phi}}}

\def\taubar{\overline{\tau}}
\def\ubar{\overline{u}}
\def\wbar{\overline{\w}}

\def\Omegatilde{{\widetilde{\Omega}}}

\def\Mdual{\check{M}}
\def\Mbar{\overline{\M}}

\def\epsilondual{\check{\epsilon}}
\def\edual{\check{\e}}

\def\MHS{{\sf MHS}}

\def\HRep{{\sf HRep}}

\def\sl{\mathfrak{sl}}

\def\Sp{{\mathrm{Sp}}}
\def\SL{{\mathrm{SL}}}
\def\SO{{\mathrm{SO}}}
\def\GL{{\mathrm{GL}}}
\def\PSL{{\mathrm{PSL}}}
\def\Ga{{\mathbb{G}_a}}
\def\Gm{{\mathbb{G}_m}}

\def\un{\mathrm{un}}
\def\cts{\mathrm{cts}}

\def\fin{\mathrm{fin}}

\def\tot{\mathrm{tot}\,}
\def\eis{\mathrm{eis}}
\def\cusp{{\mathrm{cusp}}}

\def\DR{{\mathrm{DR}}}

\def\adual{{\check{\a}}}
\def\bdual{{\check{\b}}}

\def\vv{{\vec{v}}}
\def\vw{{\vec{w}}}

\def\dot{{\bullet}}
\def\bs{{\backslash}}
\def\bbs{{\bs\negthickspace \bs}}
\def\blank{{\phantom{x}}}

\def\ll{\langle\langle}
\def\rr{\rangle\rangle}

\def\Pminus{{\P^1-\{0,1,\infty\}}}

\def\Mdual{\check{M}}

\def\Rdual{\check{R}}
\def\LCS{{\mathrm{LCS}}}
\def\cone{\mathrm{cone}}
\def\Frbar{{\overline{\Fr}}}

\def\delbar{\bar{\partial}}

\def\Ql{{\Q_\ell}}

\newcommand\im{\operatorname{im}}               % image

\newcommand\ad{\operatorname{ad}}

\newcommand\Span{\operatorname{span}}

\newcommand\Hom{\operatorname{Hom}}

\newcommand\Ext{\operatorname{Ext}}
\newcommand\End{\operatorname{End}}
\newcommand\Aut{\operatorname{Aut}}
\newcommand\Der{\operatorname{Der}}
\newcommand\Out{\operatorname{Out}}
\newcommand\OutDer{\operatorname{OutDer}}
\newcommand\InnDer{\operatorname{InnDer}}
\newcommand\Gr{\operatorname{Gr}}

\newcommand\Isom{\operatorname{Isom}}
\newcommand\Res{\operatorname{Res}}
\newcommand\Sym{\operatorname{Sym}}
\newcommand\comptensor{\operatorname{\hat{\otimes}}}
\newcommand\Dec{\operatorname{Dec}}
\newcommand\Gal{\operatorname{Gal}}

\renewcommand\Im{\operatorname{Im}}

\numberwithin{equation}{section}

\newtheorem{theorem}[equation]{Theorem}
\newtheorem{lemma}[equation]{Lemma}
\newtheorem{proposition}[equation]{Proposition}
\newtheorem{corollary}[equation]{Corollary}

\theoremstyle{definition}
\newtheorem{definition}[equation]{Definition}
\newtheorem{example}[equation]{Example}

\theoremstyle{remark}
\newtheorem{remark}[equation]{Remark}

\begin{document}

\title{The Hodge-de~Rham Theory of Modular Groups}
%%[Hodge Theory of Modular Groups]

\author{Richard Hain}
\address{Department of Mathematics\\ Duke University\\
Durham, NC 27708-0320}
\email{hain@math.duke.edu}

\thanks{Supported in part by grant DMS-1005675 from the National Science
Foundation}

\date{\today}

%\subjclass{Primary xxxxx; Secondary xxxxx}

%\keywords{}

\maketitle

\tableofcontents

\section{Introduction}

The completion $\cG_\G$ of a finite index subgroup $\G$ of $\SL_2(\Z)$ with
respect to the inclusion $\rho : \G \hookrightarrow \SL_2(\Q)$ is a proalgebraic
group, defined over $\Q$, which is an extension
$$
1 \to \U_\G \to \cG_\G \to \SL_2 \to 1
$$
of $\SL_2$ by a prounipotent group $\U_\G$, and a Zariski dense homomorphism $\G
\to \cG_\G(\Q)$. The main result of \cite{hain:malcev} implies that, for each
choice of a base point of the associated orbi-curve $X_\G = \G\bbs\h$, the
coordinate ring $\cO(\cG_\G)$  has a canonical mixed Hodge structure (MHS) that
is compatible with its product, coproduct and antipode. This MHS induces one on
the Lie algebra $\g_\G$ of $\cG_\G$.\footnote{There is more structure: if $X_\G$
is defined over the number field $K$ and if $x \in X_\G$ is a $K$-rational
point, then one also has a Galois action on $\cG\otimes \Ql$. This and the
canonical MHS on $\cG_\G$ should be the Hodge and $\ell$-adic \'etale
realizations of a motivic structure on $\cG$ that depends on $x$.}

In this paper we give a detailed exposition of the construction and basic
properties of the natural MHS on (the coordinate ring and Lie algebra of)
relative completions of modular groups. Part~\ref{part:rel_compln} is an
exposition of the basic properties of relative completion. It also contains a
direct construction of the MHS on the relative completion of the fundamental
group (and, more generally, of path torsors) of a smooth affine (orbi) curve
with respect to the monodromy representation of a polarized variation of Hodge
structure (PVHS). Part~\ref{part:mod_gps} is an exploration of the MHS on
relative completions of modular groups and their associated path torsors,
especially in the case of the full modular group $\SL_2(\Z)$.

Completions of modular groups are interesting because of their relationship to
modular forms and to categories of admissible variations of MHS over modular
curves. Because the inclusion $\G \to \SL_2(\Q)$ is injective, one might expect
the prounipotent radical $\U_\G$ of $\cG_\G$ to be trivial. However, this is not
the case. Its Lie algebra $\u_\G$ is a pronilpotent Lie algebra freely
topologically generated (though not canonically) by
\begin{equation}
\label{eqn:h1}
\prod_{m\ge 0} H^1(\G,S^m H)^\ast \otimes S^m H
\end{equation}
where $H$ denotes the defining representation of $\SL_2$ and $S^m H$ its $m$th
symmetric power. Because $H^1(\G,S^m H)$ is a space of modular forms of $\G$ of
weight $m+2$ (Eichler-Shimura), there should be a close relationship between the
MHS on $\cG_\G$ and the geometry and arithmetic of elliptic curves.

To explain the connection with admissible variations of MHS, consider the
category $\MHS(X_\G,\H)$ of admissible variations of MHS $\V$ over $X_\G$ whose
weight graded quotients have the property that the monodromy representation
$$
\G \to \Aut \Gr^W_m \V
$$
factors through an action of the algebraic group $\SL_2$ for all $m$. The
monodromy representation
$$
\G \cong \pi_1(X_\G,x) \to \Aut V_x
$$
of such a variation $\V$ factors through the canonical homomorphism $\G \to
\cG_\G(\Q)$, so that one has a natural coaction
\begin{equation}
\label{eqn:hodge_coaction}
V_x \to V_x \otimes \cO(\cG_\G).
\end{equation}
In Section~\ref{sec:avmhs} we show that there is an equivalence of categories
between $\MHS(X_\G,\H)$ and the category of ``Hodge representations'' of
$\cG_\G$ --- that is the category of representations of $\G$ on a MHS $V$ that
induce a homomorphism $\cG_\G \to \Aut V$ for which the coaction
(\ref{eqn:hodge_coaction}) is a morphism of MHS. The prounipotent radical
$\U_\G$ of $\cG_\G$, and hence modular forms via (\ref{eqn:h1}), control
extensions in $\MHS(X_\G,\H)$. This is a special case of a more general result
which is proved in \cite{hain-pearlstein}.

Modular forms give simple extensions in $\MHS(X_\G,\H)$. The fundamental
representation $H$ of $\SL_2$ corresponds to the polarized variation of Hodge
structure $\H$ of weight 1 over $X_\G$ whose fiber over the point $x\in X_\G$ is
the first cohomology group of the corresponding elliptic curve. The
classification of admissible variations of MHS over $X_\G$ in the previous
paragraph and the computation (\ref{eqn:h1}) imply that there are extensions of
variations
$$
0 \to H^1(X_\G,S^m\H)^\ast \otimes S^m\H \to \bE \to \Q \to 0.
$$
When $\G$ is a congruence subgroup, this variation splits as the sum of
extensions
$$
0 \to \Mdual_f \otimes S^m\H \to \bE_f \to \Q \to 0,
$$
where $f$ is a normalized Hecke eigenform of weight $m+2$, $M_f$ is the
corresponding Hodge structure, and $\Mdual_f=M_f(m+1)$ is its dual. When $f$ is
a cusp form, $\Mdual_f \otimes S^m\H$ has weight $-1$. In the case where
$\G=\SL_2(\Z)$, we give an explicit construction of these extensions and the
corresponding normal functions in Section~\ref{sec:cusp}. When $f$ is an
Eisenstein series, $\Mdual_f=\Q(m+1)$ and the extension is of the form
$$
0 \to S^m\H(m+1) \to \bE_f \to \Q \to 0.
$$
These extensions are constructed explicitly in Section~\ref{sec:eisenstein} when
$\G=\SL_2(\Z)$. They correspond to the elliptic polylogarithms of Beilinson
and Levin \cite{beilinson-levin}.

This work also generalizes and clarifies Manin's work on ``iterated Shimura
integrals'' \cite{manin1,manin2}. The exact relationship is discussed in
Section~\ref{sec:manin}. The periods of the MHS on $\cO(\cG_\G)$ are iterated
integrals (of the type defined in \cite{hain:malcev}) of the logarithmic forms
in Zucker's mixed Hodge complex that computes the MHS on the cohomology groups
$H^1(X_\G,S^m\H)$, whose definition is recalled in Section~\ref{sec:mhc}.
Manin's iterated Shimura integrals are iterated integrals of elements of the
subcomplex of holomorphic forms in Zucker's complex. They form a Hopf algebra
whose spectrum is a quotient $\U_A$ of $\U_\G$ by the normal subgroup generated
by $F^0\U_\G$. This quotient is not motivic as it does not support a MHS for
which the quotient mapping $\U_\G \to \U_A$ is a morphism of MHS. There is a
further quotient $\U_B$ of $\U_A$ that is dual to the Hopf algebra generated by
iterated integrals of Eisenstein series. In Section~\ref{sec:relations} we show
that it is not motivic by relating it to the {\em Eisenstein quotient} of
$\U_\G$, described below.

Fix a base point $x\in X_\G$, so that $\cG_\G$ denotes the completion of
$\pi_1(X_\G,x) \cong \G$ with its natural MHS. The ``Eisenstein quotient''
$\cG_\G^\eis$ of $\cG_\G$, defined in Section~\ref{sec:eisen_quot}, is the
maximal quotient of $\cG_\G$ whose Lie algebra $\g_\G^\eis$ has a MHS whose
weight graded quotients are sums of Tate twists of the natural Hodge structure
on $S^n H_x$. Its isomorphism type does not depend on the base point $x$. As $x$
varies in $X_\G$, the coordinate rings of the Eisenstein quotients form an
admissible VMHS over $X_\G$.

Denote the Lie algebra of $\U_B$ by $\u_B$ and of the prounipotent radical
$\U_\G^\eis$ of $\cG^\eis_\G$ by $\u_\G^\eis$. Since the Hodge structure
$\Mdual_f \otimes S^n H_x$ associated to an eigencuspform $f$ is not of this
type, such Hodge structures will lie in the kernel of
$$
H_1(\u_\G) \to H_1(\u_\G^\eis),
$$
which implies that $H_1(\u_\G^\eis)$ is generated by one copy of $S^m H_x(m+1)$
for each normalized Eisenstein series of weight $m+2$. In particular, when $\G =
\SL_2(\Z)$,
$$
H_1(\u_\G^\eis) \cong \prod_{n\ge 1} S^{2n} H_x(2n+1).
$$
There is a natural projection $\U_B \to \U_\G^\eis$ from Manin's quotient of
$\U_\G$ to $\U_\G^\eis$ that induces an isomorphism
$$
H_1(\u_B) \cong H_1(\u_\G^\eis).
$$
But, as we show in Section~\ref{sec:relations}, the cuspidal generators
$\Mdual_f\otimes S^{2n}H_x$ of $\u_\G$ become non-trivial relations in
$\u_\G^\eis$. Such relations were suggested by computations in the $\ell$-adic
\'etale version with Makoto Matsumoto (cf.\ \cite{hain-matsumoto:mem}). Evidence
for them was provided by Aaron Pollack's undergraduate thesis \cite{pollack} in
which he found quadratic relations between the generators of the image of the
representation $\Gr^W_\dot\u_\G^\eis \to \Der\L(H)$ induced by the natural
action of $\U_\G^\eis$ on the unipotent fundamental group of a once punctured
elliptic curve, which we construct in Section~\ref{sec:monodromy}.\footnote{He
also found, for each cusp form, relations of all degrees $\ge 3$ that hold in a
certain quotient of the image of this representation.} The arguments in
Section~\ref{sec:relations} and the computations of Brown \cite{brown} and
Terasoma \cite{terasoma} (Thm.~\ref{thm:cup}) imply that Pollack's quadratic
relations also hold in $\u_\G^\eis$. Since $\u_B$ is free and since $\u^\eis$ is
not, Manin's quotient $\u_B$ does not support a natural MHS.

The starting point of much of this work is the theory of ``universal mixed
elliptic motives'' \cite{hain-matsumoto:mem} developed with Makoto Matsumoto.
The origin of that project was a computation in 2007 of the $\ell$-adic weighted
completion of $\pi_1(\M_{1,1/\Z[1/\ell]})$ in which we observed that cuspidal
generators of the relative completion of the geometric fundamental group of
$\M_{1,1/\Z[1/\ell]}$ appeared to become relations in the weighted completion of
its arithmetic fundamental group. Pollack's thesis \cite{pollack} added evidence
that these cuspidal generators had indeed become relations in the weighted
completion. 

Finally, we mention related work by Levin and Racinet \cite{levin-racinet},
Brown and Levin \cite{brown-levin}, and Calaque, Enriquez and Etingof
\cite{cee}, and subsequent work of Enriquez.

Although the paper contains many new results, it is expository. The intended
audience is somebody who is familiar with modern Hodge theory. Several standard
topics, such as a discussion of modular symbols, are included to fix notation
and point of view, and also to make the paper more accessible. The reader is
assumed to be familiar with the basics of mixed Hodge structures, their
construction and their variations, including the basics of computing limit mixed
Hodge structures.

\bigskip

\noindent{\em Acknowledgments:} It is a pleasure to acknowledge the
mathematicians with whom I have had fruitful discussions, which helped shape my
view of the subject of these notes. In particular, I would like to thank my long
term collaborator, Makoto Matsumoto, as well as Aaron Pollack and Francis Brown.
I am indebted to Francis Brown and Tomohide Terasoma, each of whom communicated
their computation of the cup product, Theorem~\ref{thm:cup}. I am also grateful
to Francis Brown for his interest in the project and for his numerous
constructive comments and corrections. 
\vspace{1cm}

\subsection{Notation and Conventions}

\subsubsection{Path multiplication and iterated integrals} In this paper we use
the topologist's convention (which is the opposite of the algebraist's
convention) for path multiplication. Two paths $\alpha, \beta:[0,1]\to X$ in a
topological space $X$ are composable when $\alpha(1)=\beta(0)$. The product
$\alpha\ast \beta$ of two composable paths first traverses $\alpha$ and then
$\beta$.

Denote the complex of smooth $\C$-valued forms on a smooth manifold $M$ by
$E^\dot(M)$. Iterated integrals are defined using Chen's original definition: if
$\w_1,\dots,\w_r\in E^1(M)\otimes A$ are 1-forms on a manifold $M$ that take
values in an associative $\C$-algebra $A$ and $\alpha : [0,1]\to M$ is a
piecewise smooth path, then
$$
\int_\alpha \w_1\w_2\dots \w_r
= \int_{\D^r} f_1(t_1)\dots f_r(t_r)dt_1dt_2\dots dt_r.
$$
where $f_j(t)dt = \alpha^\ast \w_j$ and $\D^r$ is the ``time ordered''
$r$-simplex
$$
\D^r = \{(t_1,\dots,t_r)\in \R^n : 0 \le t_1 \le t_2 \le \cdots \le t_r \le 1\}.
$$
An exposition of the basic properties of iterated integrals can be found
in \cite{hain:geom,hain:prospects}.

\subsubsection{Filtrations}
The lower central series (LCS) $L^\dot G$ of a group $G$ is defined by
$$
G = L^1 G \supseteq L^2 G \supseteq L^3 G \supseteq \cdots
$$
where $L^{m+1}G = [G,L^mG]$. Its associated graded $\Gr^\dot_\LCS G$ is a graded
Lie algebra over $\Z$ whose $m$th graded quotient is $\Gr^m_\LCS G := L^m
G/L^{m+1}G$.

The lower central series $L^\dot \g$ of a Lie algebra $\g$ is defined similarly.
A Lie algebra $\g$ is {\em nilpotent} if $L^N \g = 0$ for some $N\ge 0$.

\subsubsection{Hodge theory} All mixed Hodge structures will be $\Q$ mixed Hodge
structures unless otherwise stated. The category of $\Q$-mixed Hodge structures
will be denoted by $\MHS$. The category of $\R$-mixed Hodge structures will be
denoted by $\MHS_\R$.

Often we will abbreviate mixed Hodge structure by MHS, variation of MHS by VMHS,
mixed Hodge complex by MHC, cohomological MHC by CMHC. The category of
admissible VMHS over a smooth variety $X$ will be denoted by $\MHS(X)$.

\section{Preliminaries}

\subsection{Proalgebraic groups}
\label{sec:alg_gps}

In this paper, the term {\em algebraic group} will refer to a linear algebraic
group. Suppose that $F$ is a field of characteristic zero. A proalgebraic group
$G$ over $F$ is an inverse limit of algebraic $F$-groups $G_\alpha$. The
coordinate ring $\cO(G)$ of $G$ is the direct limit of the coordinate rings of
the $G_\alpha$. The Lie algebra $\g$ of $G$ is the inverse limit of the Lie
algebras $\g_\alpha$ of the $G_\alpha$.  It is a Hausdorff topological Lie
algebra. The neighbourhoods of $0$ are the kernels of the canonical projections
$\g \to \g_\alpha$.

The continuous cohomology of $\g=\varprojlim \g_\alpha$ is defined by
$$
H^\dot(\g) := \varinjlim_\alpha H^\dot(\g_\alpha).
$$
Its homology is the full dual:
$$
H_\dot(\g) := \Hom_F(H^\dot(\g),F) \cong \varprojlim H_\dot(\g_\alpha)
$$
Each homology group is a Hausdorff topological vector space; the neighbourhoods
of $0$ are the kernels of the natural maps $H_\dot(\g) \to H_\dot(\g_\alpha)$.

Continuous cohomology can be computed using continuous Chevalley-Eilenberg
cochains:
$$
\cC^\dot(\g) = \Hom^\cts_F(\Lambda^\dot\g,F)
:= \varinjlim_\alpha \Hom_F(\Lambda^\dot\g_\alpha,F)
$$
with the usual differential.

If, instead, $\g=\bigoplus_m \g_m$ is a graded Lie algebra, then the homology
and cohomology of $\g$ are also graded. This follows from the fact that the
grading of $\g$ induces a grading of the Chevalley-Eilenberg chains and cochains
of $\g$.

\subsection{Prounipotent groups and pronilpotent Lie algebras}

A prounipotent $F$-group is a proalgebraic group that is an inverse limit of
unipotent $F$-groups.

A pronilpotent Lie algebra over a $F$ is an inverse limit of finite dimensional
nilpotent Lie algebras. The Lie algebra of a prounipotent group is a
pronilpotent Lie algebra. The functor that takes a prounipotent group to its Lie
algebra is an equivalence of categories between the category of unipotent
$F$-groups and the category of pronilpotent Lie algebras over $F$.

The following useful result is an analogue for pronilpotent Lie algebras of a
classical result of Stallings \cite{stallings}. A proof can be found in \cite[\S
3]{hain:rel_wt}.

\begin{proposition}
\label{prop:stallings}
For a homomorphism $\varphi : \n_1 \to \n_2$ of pronilpotent Lie algebras,
the following are equivalent:
\begin{enumerate}

\item $\varphi$ is an isomorphism,

\item $\varphi^\ast : H^\dot(\n_2)\to H^\dot(\n_1)$ is an isomorphism,

\item $\varphi^\ast : H^j(\n_2) \to H^j(\n_1)$ is an isomorphism when
$j=1$ and injective when $j=2$.

\end{enumerate} \qed
\end{proposition}

Another useful fact that we shall need is the following exact sequence, which
is essentially due to Sullivan \cite{sullivan}.

\begin{proposition}
\label{prop:cup}
If $\n$ is a pronilpotent Lie algebra over $F$, then the sequence
$$
\xymatrix{
0 \ar[r] & \big(\Gr_\LCS^2 \n\big)^\ast \ar[r]^{[\blank,\blank]^\ast} &
\Lambda^2 H^1(\n) \ar[r]^{\text{cup}} & H^2(\n)
}
$$
is exact, where $(\blank)^\ast = \Hom^\cts(\blank,F)$ and $[\blank,\blank]^\ast$
denotes the continuous dual of the bracket $\Lambda^2 H_1(\n) \to \Gr^2_\LCS
\n$.
\end{proposition}

\begin{remark}
When $H_1(\n)$ is finite dimensional, one can dualize to obtain the exact
sequence
$$
\xymatrix{
H_2(\n) \ar[r]^(.45){\text{cup}^\ast} & \Lambda^2 H_1(\n)
\ar[r]^{[\blank,\blank]} & \Gr^2_\LCS \n \ar[r] & 0.
}
$$
However, when $H_1(\n)$ is infinite dimensional, this sequence is not exact.
Instead, one needs to replace $\Lambda^2 H_1(\n)$ by the alternating part of the
completed tensor product $H_1(\n)\comptensor H_1(\n)$.
\end{remark}

\subsection{Free Lie algebras}
\label{sec:free}

Suppose that $F$ is a field of characteristic 0 and that $V$ is a vector space
over $F$. Here we are not assuming $V$ to be finite dimensional. The free Lie
algebra generated by $V$ will be denoted by $\L(V)$. It is characterized by the
property that a linear map $V \to \g$ into a Lie algebra over $F$ induces a
unique Lie algebra homomorphism $\L(V) \to \g$. The Poincar\'e-Birkhoff-Witt
Theorem implies \cite{serre:lalg} that $\L(V)$ is the Lie subalgebra of the
tensor algebra $T(V)$ (with bracket $[A,B]=AB-BA$) generated by $V$ and that the
inclusion $\L(V) \to T(V)$ induces an isomorphism $U\L(V) \to T(V)$ from the
enveloping algebra of $\L(V)$ to $T(V)$. The cohomology of $\L(V)$ with trivial
coefficients vanishes in degrees $>1$.

If $\f$ is a Lie algebra, then any splitting of the projection $\f \to H_1(\f)$
induces a homomorphism $\L(H_1(\f)) \to \f$. If $\f$ is free, then this
homomorphism is an isomorphism \cite{serre:lalg}. It induces a {\em canonical}
isomorphism
\begin{equation}
\label{eqn:free}
\Gr^\dot_\LCS \f \cong \L(H_1(\f))
\end{equation}
of the graded Lie algebra associated to the lower central series (LCS) of
$\f$ with the free Lie algebra generated by its first graded quotient
$H_1(\f) = \f/L^2\f$.

The {\em free completed Lie algebra} $\L(V)^\wedge$ generated by $V$ is defined
to be
$$
\L(V)^\wedge = \varprojlim \n,
$$
where  $\n$ ranges over all finite dimensional nilpotent quotients of $\L(V)$.
It is viewed as a topological Lie algebra. It is useful to note that there is a
canonical isomorphism
$$
\L(V)^\wedge = \varprojlim_{W,n} \L(W)/L^n\L(W) 
$$
of topological Lie algebras, where $W$ ranges over all finite dimensional
quotients of $V$ and $n$ over all positive integers.

We can regard $V$ as a topological vector space: the neighbourhoods of 0 are the
subspaces of $V$ of finite codimension. Every continuous linear mapping $V \to
\u$ from $V$ into a pronilpotent Lie algebra induces a unique continuous
homomorphism $\L(V)^\wedge \to \u$. The continuous cohomology of $\L(V)^\wedge$
vanishes in degrees $\ge 2$.

If $\n$ is a pronilpotent Lie algebra, then any continuous section of the
quotient mapping $\n \to H_1(\n)$ induces a continuous surjective homomorphism
$\L(H_1(\n))^\wedge \to \n$. Applying Proposition~\ref{prop:stallings} to
this homomorphism, we obtain:

\begin{proposition}
A pronilpotent Lie algebra is free if and only if $H^2(\n)=0$. \qed
\end{proposition}

\part{Completed Path Torsors of Affine Curves}
\label{part:rel_compln}

\section{Relative Completion in the Abstract}

Suppose that $\G$ is a discrete group and that $R$ is a reductive algebraic
group over a field $F$ of characteristic zero. The completion of $\G$ relative
to a Zariski dense representation $\rho : \G \to R(F)$ is a proalgebraic
$F$-group $\cG$ which is an extension
$$
1 \to \U \to \cG \to R \to 1
$$
of $R$ by a prounipotent group $\U$, and a homomorphism $\rhohat : \G \to
\cG(F)$ such that the composite
$$
\xymatrix{
\G \ar[r]^(.45)\rhohat & \cG(F) \ar[r] & R(F)
}
$$
is $\rho$. It is universal for such groups: if $G$ is a proalgebraic $F$ group
that is an extension of $R$ by a prounipotent group, and if $\phi : \G \to G(F)$
is a homomorphism whose composition with $G \to R$ is $\rho$, then there is a
homomorphism $\phihat : \cG \to G$ of proalgebraic $F$-groups such that the
diagram
$$
\xymatrix{
\G \ar[r]^\rhohat\ar[d]_\phi & \cG(F) \ar[dl]_\phihat\ar[d] \cr
G(F) \ar[r] & R(F)
}
$$
commutes.

When $R$ is trivial, $\rho$ is trivial and $\cG = \U$ is the unipotent
completion of $\G$ over $F$.

Relative completion can be defined as follows:
Let $\cL(\G,R)$ denote the category of finite dimensional $F$-linear
representations $V$ of $\G$ that admit a filtration
$$
0 = V_0 \subset V_1 \subset \cdots \subset V_N = V
$$
by $\G$-submodules with the property that each graded quotient $V_j/V_{j-1}$ is
an $R$-module and the action of $\G$ on it factors through $\rho$. It is a
neutral tannakian category. The completion of $\G$ relative to $\rho$ is the
fundamental group of this category with respect to the fiber functor that takes
a representation to its underlying vector space.

We will generally be sloppy and not distinguish between a proalgebraic group $G$
and its group $G(F)$ of $F$-rational points. For example, in the context of
relative completion, $\rho$ will be a homomorphism $\G \to R$.

\subsection{Levi splittings}

The following generalization of Levi's Theorem implies that the relative
completion $\cG$ of a finitely generated group $\G$ can be expressed
(non-canonically) as a semi-direct product $\cG\cong R\ltimes\U$. The Lie
algebra $\u$ of $\U$ is then a pronilpotent Lie algebra in the category of
$R$-modules. The isomorphism type of $\cG$ is determined by $\u$ with its
$R$-action.

Suppose that $F$ is a field of characteristic 0 and that $R$ is a reductive
$F$-group. Call an extension
$$
1 \to \U \to \cG \to R \to 1
$$
of $R$ by a prounipotent group in the category of affine $F$-groups {\em
quasi-finite} if for all finite dimensional $R$-modules $V$,
$
\Hom_R(V,H_1(\U))
$
is finite dimensional. The results in the following section imply that the
completion of a finitely generated group $\G$ relative to a homomorphism $\rho :
\G \to R(F)$ is a quasi finite extension of $R$.

\begin{proposition}
\label{prop:levi}
Every quasi-finite extension of $R$ by a prounipotent group $\U$ is split.
Moreover, any two splittings are conjugate by an element of $\U(F)$.
\end{proposition}

\begin{proof}[Sketch of Proof]
The classical case where $\U$ is an abelian unipotent group (i.e., a finite
dimensional vector space) was proved by Mostow in \cite{mostow}. (See also,
\cite[Prop.~5.1]{borel-serre}.)

First consider the case where $\U$ is an abelian proalgebraic group. The
quasi-finiteness assumption implies that there are (finite dimensional) abelian
unipotent groups $\U_\alpha$ with $R$-action and an $R$-equivariant isomorphism
$$
\U \cong \prod_\alpha \U_\alpha.
$$
The extension of $\cG$ by $\U$ can be pushed out along the projection
$\U \to \U_\alpha$ to obtain extensions
$$
1 \to \U_\alpha \to \cG_\alpha \to R \to 1.
$$
The classical case, stated above, implies that each of these has a splitting
$s_\alpha$ and that this splitting is unique up to conjugation by an element of
$\U_\alpha$. These sections assemble to give a section $s=(s_\alpha)$ of $\cG
\to R$ that is defined over $F$. Every section of $\cG \to R$ is of this form.
Any two are conjugate by an element of $\U(F)$.

To prove the general case, consider the extensions
\begin{equation}
\label{eqn:gp_extn}
1 \to \U_n \to \cG_n \to R \to 1
\end{equation}
where $\cG_n = \cG/L^{n+1}\U$, $\U_n = \U/L^{n+1}\U$, and where $L^n\U$ denotes
the $n$th term of the LCS of $\U$. The result is proved by constructing a
compatible sequence of sections of these extensions. We have already
established the $n=1$ case. Suppose that $n>1$ and that we have constructed a
splitting of $s_{n-1}$ of $\cG_{n-1}\to R$ and shown that any two such
splittings are conjugate by an element of $\U_{n-1}$.

Pulling back the extension
$$
1 \to \Gr^n_\LCS \U \to \cG_n \to \cG_{n-1} \to 1
$$
along $s_{n-1}$ gives an extension
$$
1 \to \Gr^n_\LCS \U \to G \to R \to 1.
$$
The quasi-finite assumption implies that the $R$-module $\Gr^n_\LCS \U$ is a
product of finite dimensional $R$-modules. The $n=1$ case implies that this
extension is split and that any two splittings are conjugate by an element of
$\Gr^n_\LCS \U$. If $s$ is a section of $G \to R$, then the composition of $s$
with the inclusion $G \hookrightarrow \cG_n$ is a section $s_n$  of
(\ref{eqn:gp_extn}) that is compatible with $s_{n-1}$:
$$
\xymatrix{
1 \ar[r] & \Gr^n_\LCS \U \ar[r]\ar@{=}[d] & G \ar[r]\ar[d] &
R \ar[r]\ar[d]^{s_{n-1}}\ar@/_1pc/[l]_s \ar[ld]_{s_n} & 1
\cr
1 \ar[r] & \Gr^n_\LCS \U \ar[r] & \cG_n \ar[r] & \cG_{n-1} \ar[r] & 1
}
$$
The uniqueness of $s$ implies that any two such lifts of $s_{n-1}$ are conjugate
by an element of $\Gr^n_\LCS\U(F)$. This and the fact that $s_{n-1}$ is unique
up to conjugation by an element of $\U(F)$ implies that $s_n$ is as well.
\end{proof}

\subsection{Cohomology}
\label{sec:coho}
We continue with the notation above, where $\cG$ is the relative completion of
$\G$. When $R$ is reductive, the structure of $\g$ and $\u$ are closely related
to the cohomology of $\G$ with coefficients in rational representations of $R$.
We will assume also that $H^j(\G,V)$ is finite dimensional when $j\le 2$ for all
rational representations $V$ of $R$. This condition is satisfied when $\G$ is
finitely presented and thus by fundamental groups of all complex algebraic
varieties.

For each rational representation $V$ of $R$ there are natural isomorphisms
$$
\Hom_R^\cts(H_\dot(\u),V) \cong [H^\dot(\u)\otimes V]^R
\cong H^\dot(\cG,V).
$$
The homomorphism $\G \to \cG(F)$ induces a homomorphism
\begin{equation}
\label{eqn:homom}
H^\dot(\u,V)^R \cong H^\dot(\cG,V) \to H^\dot(\G,V)
\end{equation}
It is an isomorphism in degrees $\le 1$ and an injection in degree 2.

Denote the set of isomorphism classes of finite dimensional irreducible
representations of $R$ by $\Rdual$. Fix an $R$-module $V_\lambda$ in each
isomorphism class $\lambda \in \Rdual$. If each irreducible representation of
$R$ is absolutely irreducible\footnote{This is the case when $R=\Sp_g$ over any
field of characteristic zero.} and if $H^j(\G,V)$ is finite dimensional for all
rational representations $V$ of $R$ when $j=1,2$, then (\ref{eqn:homom}) implies
that there is an isomorphism
\begin{equation}
\label{eqn:h1u}
\prod_{\lambda \in \check{R}} [H^1(\G,V_\lambda)]^\ast \otimes_F V_\lambda
\cong H_1(\u)
\end{equation}
of topological modules, and that there is
a continuous $R$-invariant surjection
$$
\prod_{\lambda \in \check{R}} [H^2(\G,V_\lambda)]^\ast \otimes_F V_\lambda
\to H_2(\u).
$$
In both cases, the LHS has the product topology.

\subsection{Base change}
\label{sec:basechange}

When discussing the mixed Hodge structure on a relative completion of the
fundamental group of a complex algebraic manifold $X$, we need to be able to
compare  the completion of $\pi_1(X,x)$ over $\R$ (or $\Q$) with its completion
over $\C$. For this reason we need to discuss the behaviour of relative
completion under base change.

The cohomological properties of relative completion stated above imply that it
behaves well under base change. To explain this, suppose that $K$ is an
extension field of $F$. Then $\rho_K : \G \to R(K)$ is Zariski dense in
$R\times_F K$, so one has the completion $\cG_K$ of $\G$ relative to $\rho_K$.
It is an extension  of $R\times_F K$ by a prounipotent group. The universal
mapping property of $\cG_K$ implies that the homomorphism $\G \to \cG(K)$
induces a homomorphism $\cG_K \to \cG\times_F K$ of proalgebraic $K$-groups. The
fact that (\ref{eqn:homom}) is an isomorphism in degree 1 and injective in
degree 2 implies that this homomorphism is an isomorphism.

\subsection{Examples}

Here the coefficient field $F$ will be $\Q$. But because of base change, the
discussion is equally valid when $F$ is any field of characteristic 0.

\subsubsection{Free groups}
Suppose that $\G$ is a finitely generated free group and that $\rho : \G \to
R(F)$ is a Zariski dense reductive representation. Denote the completion of $\G$
with respect to $\rho$ by $\cG$ and its unipotent radical by $\U$. Denote their
Lie algebras by $\g$ and $\u$. Since $H^j(\G,V)$ vanishes for all $R$-modules
$V$ for all $j\ge 2$, $\u$ is free. Consequently, the homomorphism
(\ref{eqn:homom}) is an isomorphism in all degrees.

\subsubsection{Modular groups}
\label{sec:modular}
Suppose that $\G$ is a modular group --- that is, a finite index subgroup of
$\SL_2(\Z)$. Let $R = \SL_2$ and $\rho : \G \to \SL_2(\Q)$ be the inclusion.
This has Zariski dense image. Denote the completion of $\G$ with respect to
$\rho$ by $\cG$ and its unipotent radical by $\U$.

Every torsion free subgroup $\G'$ of $\SL_2(\Z)$ is the fundamental group of the
quotient $\G'\bs\h$ of the upper half plane by $\G'$. Since this is a
non-compact Riemann surface, $\G'$ is free. Since $\SL_2(\Z)$ has finite index
torsion free subgroups (e.g., the matrices congruent to the identity mod $m$ for
any $m\ge 3$), every modular group is virtually free. This implies that
$H^j(\G,V)=0$ whenever $j\ge 2$ and $V$ is a rational vector space. The results
of Section~\ref{sec:coho} imply that the Lie algebra $\u$ of $\U$ is a free
pronilpotent Lie algebra. As in the case of a free group, this implies that the
homomorphism (\ref{eqn:homom}) is an isomorphism in all degrees.

The set $\Rdual$ of isomorphism classes of irreducible $R$-modules is $\N$. The
natural number $n$ corresponds to the $n$th symmetric power $S^n H$ of the
defining representation $H$ of $\SL_2$. The results of Section~\ref{sec:coho}
imply that there is a non-canonical isomorphism
$$
\u \cong \L\Big(\bigoplus_{n\ge 0} H^1(\G,S^n H)^\ast\otimes S^n H\Big)^\wedge
$$
of pronilpotent Lie algebras in the category of $\SL_2$ modules. (Cf.\
Remarks~3.9 and 7.2 in \cite{hain:torelli}.) So we have a complete description
of $\cG$ as a proalgebraic group:
$$
\cG \cong \SL_2\ltimes \exp\u.
$$
In Section~\ref{sec:concrete} we give a method for constructing a
homomorphism $\G \to \SL_2(\C)\ltimes \exp\u$ that induces an isomorphism
$\cG \to \SL_2(\C)\ltimes \exp\u$.

\subsubsection{Unipotent completion of fundamental groups of punctured elliptic
curves}
\label{sec:unipt_fund}
Here $E$ is a smooth elliptic curve over $\C$ and $\G = \pi_1(E',x)$ where
$E'=E-\{0\}$ and $x\in E'$. In this case we take $R$ to be trivial. The
corresponding completion of $\G$ is the unipotent completion of $\pi_1(E',x)$.
Since $H^2(E')=0$, the results of Section~\ref{sec:coho} imply that the Lie
algebra $\p$ of the unipotent completion of $\pi_1(E',x)$ is (non-canonically
isomorphic to) the completion of the free Lie algebra generated by $H_1(E,\Q)$:
$$
\p \cong \L(H_1(E))^\wedge.
$$
This induces a canonical isomorphism $\Gr^\dot_\LCS \p \cong \L(H_1(E))$ of the
associated graded Lie algebra of the lower central series (LCS) of $\p$ with the
free Lie algebra generated by $H_1(E)$.

\subsection{Naturality and Right exactness}

The following naturality property is easily proved using either the universal
mapping property of relative completion or its tannakian description.

\begin{proposition}
\label{prop:naturality}
Suppose that $\G$ and $\G'$ are discrete groups and that $R$ and $R'$ are
reductive $F$-groups. If one has a commutative diagram
$$
\xymatrix{
\G' \ar[r]^{\rho'}\ar[d] & R'\ar[d]  \cr
\G \ar[r]^\rho & R
}
$$
in which $\rho$ and $\rho'$ are Zariski dense, then one has a commutative
diagram
$$
\xymatrix{
\G' \ar[r]\ar[d] & \cG'\ar[d]\ar[r] & R'\ar[d]  \cr
\G \ar[r] & \cG \ar[r] & R
}
$$
where $\cG$ and $\cG'$ denote the completions of $\G$ and $\G'$ with respect
to $\rho$ and $\rho'$. \qed
\end{proposition}

Relative completion is not, in general, an exact functor.  However, it is right
exact. The following is a special case of this right exactness. It can be proved
using the universal mapping property of relative completion. (A similar argument
can be found in \cite[\S 4.5]{hain-matsumoto:weighted}.)

\begin{proposition}
\label{prop:rt_exact}
Suppose that $\G$, $\G'$ and $\G''$ are discrete groups and that $R$, $R'$ and
$R''$ are reductive $F$-groups. Suppose that one has a diagram
$$
\xymatrix{
1 \ar[r] & \G' \ar[r]\ar[d]^{\rho'} & \G \ar[r]\ar[d]^\rho &
\G'' \ar[r]\ar[d]^{\rho''} & 1 \cr
1 \ar[r] & R' \ar[r] & R \ar[r] & R'' \ar[r] & 1 
}
$$
with exact rows in which $\rho$, $\rho'$ and $\rho''$ are Zariski dense, then
the corresponding diagram
$$
\xymatrix{
  & \cG' \ar[r]\ar[d] & \cG \ar[r]\ar[d] & \cG'' \ar[r]\ar[d] & 1 \cr
1 \ar[r] & R' \ar[r] & R \ar[r] & R'' \ar[r] & 1
}
$$
of relative completions has right exact top row. \qed
\end{proposition}

\begin{example}

The moduli space of $n\ge 1$ pointed genus 1 curves will be denoted by
$\M_{1,n}$. It will be regarded as an orbifold. It is isomorphic to (and will be
regarded as) the moduli space of elliptic curves $(E,0)$ with $n-1$ distinct
labelled points $(x_1,\dots,x_{n-1})$. The point of $\M_{1,n}$ that corresponds
to $(E,0,x_1,\dots,x_{n-1})$ will be denoted by $[E,x_1,\dots,x_{n-1}]$.

The fiber of the projection $\M_{1,2} \to \M_{1,1}$ that takes $[E,x]$ to $[E]$
is $E' := E-\{0\}$. Fix a base point $x_o = [E,x]$ of $\M_{1,2}$ and $t_o = [E]
\in \M_{1,1}$. The (orbifold) fundamental group of $\M_{1,2}$ is an extension
$$
1 \to \pi_1(E',x) \to \pi_1(\M_{1,2},x_o) \to \pi_1(\M_{1,1},t_o) \to 1.
$$
Denote the completion of $\pi_1(\M_{1,2},x_o)$ with respect to the natural
homomorphism to $\SL(H_1(E)) \cong \SL_2(\Q)$ by $\cGtilde$. Functoriality and
right exactness of relative completion implies that we have an exact sequence
$$
\pi_1(E',x)^\un \to \cGtilde \to \cG \to 1.
$$
In this case, we can prove exactness on the left as well.

This is proved using the conjugation action of $\pi_1(\M_{1,2},x_o)$ on
$\pi_1(E',x)$, which induces an action of $\pi_1(\M_{1,2},x_o)$ on the Lie
algebra $\p$ of $\pi_1(E',x)^\un$. This action preserves the lower central
series filtration of $\p$ and therefore induces an action on $\Gr^\dot_\LCS\p
\cong \L(H_1(E))$. This action is determined by its action on $H_1(E)$, and
therefore factors through the homomorphism $\pi_1(\M_{1,2},x_o) \to \SL_2(\Q)$.
The universal mapping property of relative completion implies that this induces
an action $\cGtilde \to \Aut\p$ and the corresponding Lie algebra homomorphism
$\gtilde \to \Der \p$. The composite $\p \to \gtilde \to \Der\p$ is the
homomorphism induced by the conjugation action of $\pi_1(E',x)^\un$ on itself
and is therefore the adjoint action. Since $\p$ is free of rank $>1$, it has
trivial center, which implies that the adjoint action is faithful and that $\p
\to \gtilde$ is injective.

\end{example}

\subsection{Hodge Theory}

Suppose that $X$ is the complement of a normal crossings divisor in a compact
K\"ahler manifold. Suppose that $F=\Q$ or $\R$ and that $\V$ is a polarized
variation of $F$-Hodge structure over $X$. Pick a base point $x_o \in
X$.\footnote{We also allow tangential base points.} Denote the fiber over $\V$
over $x_o$ by $V_o$. The Zariski closure of the image of the monodromy
representation
$$
\rho : \pi_1(X,x_o) \to \Aut(V_o)
$$
is a reductive $F$-group, \cite[4.2.6]{deligne:hodge2}. Denote it by $R$. Then
one has the relative completion $\cG$ of $\pi_1(X,x_o)$ with respect to $\rho :
\pi_1(X,x_o) \to R(F)$.

\begin{theorem}[\cite{hain:malcev}]
The coordinate ring $\cO(\cG)$ is a Hopf algebra in the category of Ind-mixed
Hodge structures over $F$. It has the property that $W_{-1}\cO(\cG)=0$ and
$W_0\cO(\cG) = \cO(R)$.
\end{theorem}

A slightly weaker version of the theorem is stated in terms of Lie algebras.
Denote the prounipotent radical of $\cG$ by $\U$. Denote their Lie algebras by
$\g$ and $\u$, and the Lie algebra of $R$ by $\r$.

\begin{corollary}[\cite{hain:malcev}]
\label{cor:hodge_la}
The Lie algebra $\g$ is a Lie algebra in the category of pro-mixed Hodge
structures over $F$. It has the property that
$$
\g = W_0 \g,\ \u= W_{-1}\g, \text{ and } \Gr^W_0 \g \cong \r.
$$
If $\V$ is a PVHS over $X$ with fiber $V_o$ over the base point $x_o$, then the
composite
$$
H^\dot(\u,V_o)^R \to H^\dot(\G,V_o) \to H^\dot(X,\V)
$$
of (\ref{eqn:homom}) with the canonical homomorphism is a morphism of MHS. It is
an isomorphism in degrees $\le 1$ and injective in degree 2. When $X$ is an
(orbi) curve, it is an isomorphism in all degrees.
\end{corollary}

The existence of the mixed Hodge structure on $\u$ in the unipotent case and
when $X$ is not necessarily compact is due to Morgan \cite{morgan} and Hain
\cite{hain:malcev}. The results in this section also hold in the orbifold case.

\begin{example}
\label{ex:monod}
The local system $R^1f_\ast \Q$ over $\M_{1,1}$ associated to the universal
elliptic curve $f:\E \to \M_{1,1}$ is a polarized variation of Hodge structure
of weight 1. This variation and its pullback to $\M_{1,n}$ will be denoted by
$\H$. It has fiber $H^1(E)$ over $[E]$. The Zariski closure of the monodromy
representation $\pi_1(\M_{1,1},[E]) \to \Aut H^1(E)$ is $\SL(H_1(E))$, which is
isomorphic to $\SL_2$. Poincar\'e duality $H_1(E)\cong H^1(E)(1)$ induces an
isomorphism of $\H(1)$ with the local system over $\M_{1,1}$ whose fiber over
$[E]$ is $H_1(E)$.

The choice of an elliptic curve $E$ and a non-zero point $x$ of $E$ determines
compatible base points of $E'$, $\M_{1,1}$ and $\M_{1,2}$. Denote the
Lie algebras of the relative completions of $\pi_1(\M_{1,1},[E])$ and
$\pi_1(\M_{1,2},[E,x])$ by $\g$ and $\gtilde$, respectively. Denote the
Lie algebra of the unipotent completion of $\pi_1(E',x)$ by $\p$. The
results of this section imply that each has a natural MHS and that the
sequence
$$
0 \to \p \to \gtilde \to \g \to 0
$$
is exact in the category of MHS. The adjoint action of $\gtilde$ on $\p$
induces an action
$$
\gtilde \to \Der\p
$$
Since the inclusion $\p \to \gtilde$ is a morphism of MHS, this homomorphism is
a morphism of MHS.

Since the functor $\Gr^\dot_W$ is exact on the category of MHS, one can study
this action by passing to its associated graded action
$$
\Gr^W_\dot \gtilde \to \Der\L(H_1(E)).
$$
\end{example}

\section{A Concrete Approach to Relative Completion}
\label{sec:concrete}

Suppose that $M$ is the orbifold quotient\footnote{For a detailed and elementary
description of what this means, see \cite[\S3]{hain:china}.} $\G\bs X$ of a
simply connected manifold $X$ by a discrete group $\G$. We suppose that $\G$
acts properly discontinuously and virtually freely\footnote{That is, $\G$ has a
finite index subgroup that acts freely on $X$.} on $X$. Our main example will be
when $X=\h$ and $\G$ is a finite index subgroup of $\SL_2(\Z)$.

Suppose that $R$ is a complex (or real) Lie group and that $(\u_\alpha)_\alpha$
is an inverse system of finite dimensional nilpotent Lie algebras in the
category of left $R$-modules. Its limit 
$$
\u = \varprojlim_\alpha \u_\alpha
$$
is a pronilpotent Lie algebra in the category of $R$-modules. Denote the
unipotent Lie group corresponding to $\u_\alpha$ by $U_\alpha$. The prounipotent
Lie group corresponding to $\u$ is the inverse limit of the $U_\alpha$.

The action of $R$ on $\u$ induces an action of $R$ on $\U$, so we can form the
semi-direct product\footnote{To be clear, the group $R\ltimes U$ is the set
$U\times R$ with multiplication
$$
(u_1,r_1)(u_2,r_2) = (u_1 (r_1\cdot u_2),r_1r_2),
$$
where $r\cdot u$ denotes the action of $R$ on $U$. We will omit the dot when it
is clear from the context that $ru$ means the action of $R$ on $U$.} $R\ltimes
\U$. This is the inverse limit
$$
R\ltimes \U = \varprojlim_\alpha (R\ltimes U_\alpha) 
$$
If $R$ is an algebraic group, then $R\ltimes \U$ is a proalgebraic group.

Suppose that $\rho : \G \to R$ is a representation. At this stage, we do not
assume that $\rho$ has Zariski dense image. The following assertion is easily
proved.

\begin{lemma}
Homomorphisms $\rhohat : \G \to R\ltimes\U$ that lift $\rho$ correspond to
functions $F : \G \to \U$ that satisfy the 1-cocycle condition
$$
F(\gamma_1\gamma_2) = F(\gamma_1)(\gamma_1\cdot F(\gamma_2)).
$$
The homomorphism $\rhohat$ corresponds to the function $\G \to \U\times R$
$$
\gamma \mapsto \big(F(u),\rho(\gamma)\big)
$$
under the identification of $R\ltimes \U$ with $\U\times R$. \qed
\end{lemma}

Cocycles can be constructed from $\G$-invariant 1-forms on $X$ with values
in $\u$. Define
$$
E^\dot(X)\comptensor \u := \varprojlim_\alpha E^\dot(X)\otimes \u_\alpha,
$$
where $E^\dot(X)$ denotes the complex of smooth $\C$-valued forms on $X$. The
group $\G$ acts on $E^\dot(X)\comptensor \u$ by
$$
\gamma \cdot \w = \big((\gamma^\ast)^{-1}\otimes \gamma\big)\w.
$$
Such a form $\w$ is invariant if
$$
(\gamma^\ast \otimes 1)\w = (1\otimes \gamma)\w
$$
for all $\gamma \in \G$.

Let $\G$ act on $X\times \U$ diagonally: $\gamma : (u,x) \to (\gamma u, \gamma
x)$. The projection
$$
X \times \U \to X
$$
is a $\G$-equivariant principal right $\U$-bundle. Its sections correspond to
functions $f : X \to \U$. Each $\w \in E^1(X)\comptensor \u$ defines a
connection on this bundle invariant under the right $\U$ action via the formula
$$
\nabla f = df + \w f,
$$
where $f$ is a $\U$-valued function defined locally on $X$. The connection is
$\G$-invariant if and only if $\w$ is $\G$-invariant. It is flat if and only if
$\w$ is integrable:
$$
d\w + \frac{1}{2}[\w,\w] = 0 \quad \text{ in } E^2(X)\comptensor \u.
$$
In this case, parallel transport defines a function
$$
T : PX \to \U
$$
from the path space of $X$ into $\U$. With our conventions, this satisfies
$T(\alpha\ast\beta) = T(\beta)T(\alpha)$. When $\w$ is integrable, $T(\alpha)$
depends only on the homotopy class of $\gamma$ relative to its endpoints. Chen's
transport formula implies that the inverse transport function is given by the
formula
\begin{equation}
\label{eqn:transport}
T(\alpha)^{-1} =
1 + \int_\alpha \w + \int_\alpha \w\w + \int_\alpha \w\w\w + \cdots
\end{equation}
Cf.\ \cite[Cor.~5.6]{hain:kzb}.

Fix a point $x_o \in X$. Since $X$ is simply connected, for each $\gamma \in \G$
there is a unique homotopy class $c_\gamma$ of paths from $x_o$ to $\gamma\cdot
x_o$.

\begin{proposition}
\label{prop:monodromy}
If $\w \in E^1(X)\comptensor \u$ is $\G$-invariant and integrable, then the
function $\Theta_{x_o} : \G \to \U$ defined by
$$
\Theta_{x_o}(\gamma) = T(c_\gamma)^{-1}
$$
is a well-defined (left) 1-cocycle with values in $\U$:
$$
\Theta_{x_o}(\gamma\mu) =
\Theta_{x_o}(\gamma)\big(\gamma\cdot\Theta_{x_o}(\mu)\big).
$$
Consequently, the function $\rhotilde_{x_o} : \G \to R\ltimes \U$ defined by
$\gamma \mapsto \big(\Theta_{x_o}(\gamma),\rho(\gamma)\big)$ is a homomorphism.
\end{proposition}

\begin{proof}
This follows directly from the fact that $c_{\gamma\mu} = c_\gamma \ast
(\gamma\cdot c_\mu)$ and the transport formula above.
\begin{figure}[!ht]
\epsfig{file=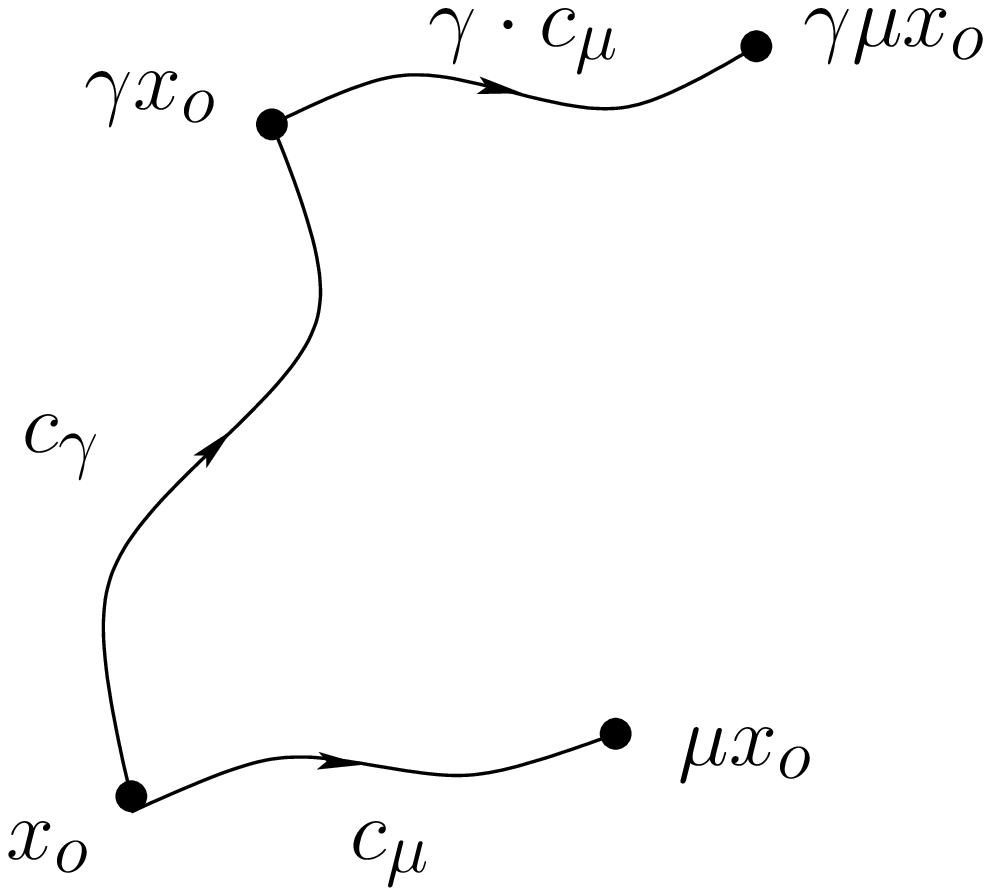, width=1.5in}
\caption{The cocycle relation $c_{\gamma\mu} = c_\gamma \ast
(\gamma\cdot c_\mu)$}
\label{fig:cocycle}
\end{figure}
\end{proof}

\begin{remark}
\label{rem:cobound}
The dependence of $\Theta_{x_o}$ and $\rhotilde_{x_o}$ on $x_o$ is easily
determined. Suppose that $x'$ is a second base point. If $e$ is the unique
homotopy class of paths in $X$ from $x_o$ to $x'$, then
$
c_\gamma' := e^{-1}\ast c_\gamma \ast (\gamma\cdot e)
$
is the unique homotopy class of paths in $X$ from $x'$ to $\gamma\cdot x'$. Thus
$$
T(c_\gamma')^{-1} = T(e)T(c_\gamma)^{-1} (\gamma\cdot T(e)^{-1}).
$$
Since, for $u,v\in\U$ and $r\in R$,
$$
(v,1)(u,r)(v^{-1},1) = (v (r\cdot v^{-1}),\gamma)
$$
in $R\ltimes \U$, the previous formula implies that $\rhotilde_{x'}$ is
obtained from $\rhotilde_{x_o}$ by conjugation by $T(e)\in \U$ and that
$$
\Theta_{x'}(\gamma) = T(e)\Theta_{x_o}(\gamma\cdot T(e)^{-1}).
$$
\end{remark}

\subsection{Variant}
\label{sec:variant}

Later we will sometimes need a slightly more general setup. This material is
standard. More details can be found in \cite[\S5]{hain:kzb}. As above,
$R$ acts on $\U$ on the left.

Recall that a {\em factor of automorphy} is a smooth function
$$
M : \G \times X \to R,\quad (\gamma,x)\mapsto M_\gamma(x)
$$
that satisfies $M_{\gamma\mu}(x) = M_\gamma(\mu x) M_\mu(x)$ for all $x\in X$
and $\gamma,\mu\in\G$. One checks easily that the function
$$
\gamma : (x,u) \mapsto (\gamma x, M_\gamma(x)u)
$$
defines a left action of $\G$ on the right principal $\U$ bundle $X\times \U$.

Suppose that $\nabla_0$ is a connection on the bundle $X\times \U \to X$. A
1-form $\w \in E^1(X)\comptensor \u$ defines a right $\U$-invariant connection
$\nabla$ on this bundle by
$$
\nabla f = \nabla_0 f + \w f,
$$
where $f : X \to \U$ is a section. This connection is $\G$-invariant if and only
if
$$
(\gamma^\ast\otimes 1)\w =
(1\otimes M_\gamma)\w - (\nabla_0 M_\gamma) M_\gamma^{-1}
\text{ all }\gamma \in \G
$$
and flat if and only if $\nabla_0\w + [\w,\w]/2 = 0$.

When $\nabla$ is a flat $\G$-invariant connection, the monodromy representation
$$
\Theta_{x_o} : \G \to R\ltimes \U
$$
is given by
$$
\Theta_{x_o} : \gamma \mapsto \big(T(c_\gamma)^{-1},M_\gamma(x_o)\big),
$$
where for a path $\alpha \in X$, $T(\alpha)$ is given by the formula
(\ref{eqn:transport}).

The setup of the previous section is a special case where $M_\gamma(x) =
\rho(\gamma)$ and where the connection $\nabla_0$ is trivial; that is, $\nabla_0
= d$.

\subsection{A characterization of relative unipotent completion}
\label{sec:charn}

To give a concrete construction of the completion of $\G$ relative to $\rho : \G
\to R$, we need a useful criterion for when a homomorphism $\G \to R\ltimes \U$
induces an isomorphism $\cG \to R\ltimes \U$ with the relative completion. In
the present situation, the criterion is cohomological.

Suppose that $\w \in E^1(X)\comptensor \u$ is an integrable $\G$-invariant
1-form on $X$, as above. Suppose that $V$ is a finite dimensional $R$-module and
that $\U \to \Aut V$ is an $R$-invariant homomorphism.\footnote{Equivalently,
$V$ is an $R\ltimes \U$-module.} Denote the action of $u\in\U$ on $v\in V$ by
$u\cdot v$. Let $M$ be the orbifold quotient of $X\times V$ by the diagonal
$\G$-action: $\gamma : (x,v)\mapsto \big(\gamma x,\rho(\gamma)v\big)$. The
formula
$$
\nabla v = dv + \w\cdot v
$$
defines a flat $\G$-invariant connection on the vector bundle $X\times V \to X$.
The complex
$$
(E^\dot(X)\otimes V)^\G :=
\{\w \in E^\dot(X)\otimes V : (\gamma^\ast\otimes 1)\w = (1\otimes \gamma)\w\}
$$
of $\G$-invariant $V$-valued forms on $X$ has differential defined by
$$
\nabla (\eta \otimes v) = d\eta\otimes v + \eta\wedge (\w\cdot v),
$$
where $\eta\in E^\dot(X)$ and $v\in V$. It computes the cohomology
$H^\dot(M,\V)$ of the orbifold $M$ with coefficients in the orbifold local
system $\V$ of locally constant sections of the flat vector bundle $\G\bs
(X\times V)$ over $M$.

\begin{lemma}
If $R$ is reductive, each integrable $\G$-invariant 1-form $\w \in
E^1(X)\comptensor \u$ induces a ring homomorphism
$$
H^\dot(\u,V)^R \to H^\dot(M,\V).
$$
\end{lemma}

\begin{proof}
Note that $\G$ acts on $\u$ on the left by the formula $\gamma : u \mapsto
\rho(\gamma)\cdot u$. A 1-form $\w \in E^1(X)\comptensor \u$ can be regarded as
a function
$$
\theta_\w : \Hom^\cts(\u,\C) \to E^1(X),
\quad \varphi \mapsto (1\otimes\varphi)\w.
$$
It is $\G$-equivariant if and only if $\w$ is. The induced algebra homomorphism
$$
\cC^\dot(\u) = \Lambda^\dot \Hom^\cts(\u,\C) \to E^\dot(X)
$$
commutes with differentials if and only if $\w$ is integrable. So a
$\G$-invariant and integrable 1-form  $\w\in E^1(X)\comptensor \u$ induces a dga
homomorphism
$$
\cC^\dot(\u,V)^R = [(\Lambda^\dot\Hom^\cts(\u,\C))\otimes V]^R
\to (E^\dot(X)\otimes V)^\G.
$$
Since $R$ is reductive, the natural map
$$
H^\dot(\cC^\dot(\u,V)^R) \to H^\dot(\cC^\dot(\u,V))^R
$$
is an isomorphism. The result follows.
\end{proof}

The desired characterization of relative completion is:

\begin{proposition}
\label{prop:char}
If $R$ is reductive and $\rho : \G\to R$ has Zariski dense image, then the
homomorphism $\rhotilde_{x_o} : \G \to R\ltimes \U$ constructed from $\w\in
(E^1(X)\comptensor\u)^\G$ above is the completion of $\G$ with respect to $\rho$
if and only if the homomorphism
$$
\theta_\w^\ast : [H^j(\u)\otimes V]^R \to H^j(M,\V)
$$
induced by $\theta_\w$ is an isomorphism when $j=0,1$ and injective when $j=2$
for all $R$-modules $V$.
\end{proposition}

When $X$ is a universal covering of a manifold, such universal 1-forms can be
constructed using a suitable modification of Chen's method of power series
connections \cite{chen}.

\begin{proof}
Denote the completion of $\Gamma$ relative to $\rho$ by $\cG$. The universal
mapping property of relative completion implies that the homomorphism
$\rhotilde_{x_o}:\G \to R\ltimes\U$ induces a homomorphism $\Psi : \cG \to
R\ltimes \U$ that commutes with the projections to $R$. Denote the prounipotent
radical of $\cG$ by $\cN$ and its Lie algebra by $\n$. Then $\Psi$ induces a
homomorphism $\cN \to \U$ and an $R$-invariant homomorphism
$$
\Psi^\ast : H^\dot(\u) \to H^\dot(\n).
$$
Then for each finite dimensional $R$-module $V$, one has the commutative diagram
$$
\xymatrix{
[H^\dot(\n)\otimes V]^R \ar[r] & H^\dot(\G,V) \ar[d] \cr 
[H^\dot(\u)\otimes V]^R \ar[u]^{\Psi^\ast}\ar[r]_(.55){\theta_\w^\ast} &
H^\dot(M,\V)
}
$$
where the right-hand vertical mapping is induced by the orbifold morphism $M \to
B\G$ of $M$ into the classifying space of $\G$. Standard topology implies that
this is an isomorphism in degrees 0 and 1 and injective in degree 2. Results in
Section~\ref{sec:coho} imply that the top row is an isomorphism in degrees 0 and
1, and injective in degree 2. The assumption implies that the left hand vertical
map is an isomorphism in degrees 0 and 1 and injective in degree 2. Since
$H^\dot(\u)$ and $H^\dot(\n)$ are direct limits of finite dimensional
$R$-modules, by letting $V$ run through all finite dimensional irreducible
$R$-modules, we see that $H^\dot(\u) \to H^\dot(\n)$ is an isomorphism in
degrees 0 and 1 and and injective in degree 2. Proposition~\ref{prop:stallings}
implies that $\n \to \u$ is an isomorphism. This implies that $\Psi$ is an
isomorphism.
\end{proof}

\subsection{Rational structure}

To construct a MHS on the completion of (say) a modular group $\G$, we will
first construct its complex form together with its Hodge and weight filtrations
using an integrable, $\G$-invariant $1$-form, as above. An easy formal argument,
given below, implies that this relative completion has a natural $\Q$ structure
provided that $R$ and $\rho$ are defined over $\Q$. To understand the MHS on
$\cO(\cG)$, we will need a concrete description of this $\Q$-structure on
$R\ltimes\U$ in terms of periods. Explaining this is the goal of this section.

In general, we are not distinguishing between a proalgebraic $F$-group and its
group of $F$ rational points. Here, since we are discussing Hodge theory, we
will distinguish between a $\Q$-group, and its groups of $\Q$ and $\C$ rational
points.

Suppose that $R$ is a reductive group that is defined over $\Q$ and that $\rho :
\G \to R$ takes values in the $\Q$-rational points $R(\Q)$ of $R$. Denote the
completion of $\G$ with respect to $\rho$ over $\Q$ by $\cG$ and its
prounipotent radical by $\cN$. These are proalgebraic $\Q$-groups. We also have
the completion $\cG_\C$ of $\G$ over $\C$ relative to $\rho$, where the
coefficient field is $\C$. Base change (cf.\ Section~\ref{sec:basechange})
implies that the natural homomorphism $\cG\otimes_\Q\C \to \cG_\C$ is an
isomorphism.

When the hypotheses of Proposition~\ref{prop:char} are satisfied we obtain
a canonical $\Q$-structure on $R\ltimes \U$ from the isomorphism
$$
\xymatrix{
\psi : \cG\otimes_\Q\C \ar[r]^(.55)\simeq & R\ltimes \U.
}
$$
induced by $\rhotilde_{x_o}$. The $\Q$-structure on $\U$ is the image of the
restriction $\cN\otimes_\Q\C \to \U$ of this isomorphism to $\cN$. This induces
a canonical $\Q$-structure on $\u$ via the isomorphism $\u \cong
\n\otimes_\Q\C$.

\begin{proposition}
\label{prop:Qstr}
The canonical $\Q$-structure on $\u$ is the $\Q$-Lie subalgebra of $\u$
generated by the set $\{\log \Theta_{x_o}(\gamma) : \gamma \in \G\}$.
\end{proposition}

\begin{proof}
Fix a $\Q$-splitting $s$ of the surjection $\cG \to R_\Q$. This gives an
identification of $\cG$ with $R_\Q \ltimes \cN$. Levi's Theorem
(Prop.~\ref{prop:levi}) implies that there is a $u\in \U$ such that the
composition of
$$
\xymatrix{
R_\Q \ar[r]^s & \cG \ar[r]^(.45)\psi & R\ltimes \U
}
$$
is conjugated to the section $R_\Q \to R\ltimes \U$ that takes $r$ to $(1,r)$ by
$u$. That is, the first section takes $r\in R$ to $(u(r\cdot u)^{-1},r)\in
R\ltimes \U$. The composite
$$
\xymatrix{
R_\Q \ltimes \cN \ar[r]^(.6)\simeq & \cG \ar[r] &
\cG_\C \ar[r]^(.4)\simeq_(.4)\psi & R\ltimes \U
}
$$
is thus given by the formula
\begin{equation}
\label{eqn:map}
(n,r) \mapsto
(\psi(n),1)(u(r\cdot u)^{-1},r) = (\psi(n)u(r\cdot u)^{-1},r) \in R\ltimes\U.
\end{equation}

The composite
$$
\xymatrix{
\G \ar[r] & \cG(\Q) \ar[r]^(.35)\simeq & R(\Q)\ltimes\cN(\Q)
}
$$
takes $\gamma\in \G$ to $(F(\gamma),\gamma) \in \cN(\Q)\times R(\Q)$ for some
1-cocycle $F : \G \to \cN(\Q)$.
Note that every divisible subgroup of the rational points of a prounipotent
$\Q$-group $N$ is the set of $\Q$-rational points of a $\Q$-subgroup of $N$.
This implies that the smallest subgroup of $\cN(\Q)$ that contains
$\{F(\gamma):\gamma\in \G\}$ is the set of $\Q$-points of a $\Q$-subgroup $\cS$
of $\cN$. The subgroup $\cS$ must be $\cN$. This is because the cocycle
condition implies that it is a $\G$-invariant subgroup:
$$
\gamma_1\cdot F(\gamma_2) = F(\gamma_1)^{-1}F(\gamma_1\gamma_2) \in \cS(\Q).
$$
Since $\G$ is Zariski dense in $\cG$, $\cS$ is a normal subgroup of $\cG$. This
implies that $R_\Q\ltimes \cS$ is a subgroup of $\cG=R_\Q\ltimes\cN$ that
contains the image of the canonical homomorphism $\rhotilde : \G \to \cG(\Q)$.
But since $\rhotilde$ is Zariski dense, we must have $\cG = R_\Q\ltimes \cS$.

Formula (\ref{eqn:map}) now implies that the image of $\cN(\Q)$ in $\U(\C)$ is
the smallest divisible subgroup of $\U(\C)$ that contains the set $\{\psi\circ
F(\gamma):\gamma\in \G\}$. But the formula (\ref{eqn:map}) and the commutativity
of the diagram
$$
\xymatrix{
\G \ar[r] \ar[dr] &  \cG(\Q) \ar[d]^\psi \cr & R\ltimes \U
}
$$
imply that $\psi\circ F = \Theta_{x_o}$, so that the image of $\cN(\Q)$ in
$\U$ is the smallest divisible subgroup of $\U$ that contains the set
$\{\Theta_{x_o}(\gamma):\gamma\in \G\}$. The result now follows from the
Baker-Campbell-Hausdorff formula.
\end{proof}

\begin{remark}
One might think that this $\Q$ structure on $\U$ can be constructed as the image
of the unipotent completion over $\Q$ of $\ker \rho$ in $\U$. This often works,
but it does not when $\G$ is a modular group as $\rho : \G \to R$ is injective
and $\U$ is non-trivial in this case.
\end{remark}

\subsubsection{Complements}

The coordinate ring $\cO(\cG)$ of $\cG$ is isomorphic, as a ring, to
$\cO(R)\otimes\cO(\U)$. Its coproduct is twisted by the action of $R$ on $\U$.
The $\Q$-form of $\cO(\cG)$ consists of those elements of $\cO(R)\otimes\cO(\U)$
that take rational values on the image of $\G \to R\ltimes\U$. Since the
exponential map $\u \to \U$ is an isomorphism of affine schemes, the coordinate
ring of $\U$ consists of the polynomial functions on $\u$ that are continuous
functions $\u \to \C$. Since the coefficients of the logarithm and exponential
functions are rational numbers, $\cO(\U_\Q)$ is the ring of continuous
polynomials on $\u_\Q$ .

\section{Relative Completion of Path Torsors}
\label{sec:path_torsor}

This section can be omitted on a first reading. Here we consider the relative
completion $\cG_{x,y}$ of the torsor of paths from $x$ to $y$ in a manifold $M$
with respect to a reductive local system $\H$. This can be described using
tannakian formalism.\footnote{The category of local systems of finite
dimensional $F$-vector spaces over $M$ that admit a filtration whose graded
quotients are local systems that correspond to representations of the Zariski
closure of $\G \to \Aut H_x$ is tannakian. The completion $\cG_{x,y}$ of the
torsor $\Pi(M;x,y)$ of paths in $M$ from $x$ to $y$ is the torsor of
isomorphisms between the fiber functors at $x$ and $y$. It is an affine scheme
over $F$.} Here we outline a direct approach partly because it is more concrete
and better suits our needs.

We use the setup of the previous section. So $M = \G\bs X$ where $X$ is a simply
connected manifold, $\G$ is a discrete group that acts properly discontinuously
and virtually freely on $X$, and $\rho : \G \to R$ is a representation of $\G$
into an affine $F$-group ($F=\R$ or $\C$), not yet assumed to be reductive or
Zariski dense.

If $M$ is a manifold and $x,y\in M$, then $\Pi(M;x,y)$ denotes the set of
homotopy classes of paths from $x$ to $y$. We need to define what we mean by
$\Pi(M;x,y)$ when the action of $\G$ on $X$ is not free, in which case, $M$ is
an orbifold, but not a manifold. Choose a fundamental domain $D$ for the action
of $\G$ on $X$. The orbit of each $x\in M$ contains a unique point $\xtilde \in
D$. Suppose that $x,y\in M$. Elements of $\Pi(M;x,y)$ can be represented by
pairs $(\gamma,c_\gamma)$, where $\gamma \in \G$ and $c_\gamma$ is a homotopy
class of paths from $\xtilde$ to $\gamma\ytilde$. (This homotopy class is unique
as $X$ is simply connected.) The composition map 
$$
\Pi(M;x,y)\times \Pi(M;y,z) \to \Pi(M;x,z)
$$
is given by
$$
\big((\gamma,c_\gamma),(\mu,c_\mu)\big) \mapsto (\gamma\mu,c_{\gamma\mu})
:= \big(\gamma\mu,c_\gamma\ast(\gamma\cdot c_\mu)\big).
$$
Note that $\Pi(M;x,y)$ is a torsor under the left action of $\pi_1(M,x) :=
\Pi(M,x,x)$ and a torsor under the right action of $\pi_1(M,y)$.
This definition of $\Pi(M;x,y)$ agrees with the standard definition when the
action of $\G$ is fixed point free.

Now suppose that $H$ is a finite dimensional vector space over $F=\C$ (or $\R$),
that $R$ is a reductive subgroup of $\GL(H)$, and $\rho : \G \to R(F) = R$ is a
Zariski dense representation. Let $\H$ be the corresponding (orbifold) local
system over the orbifold $M$.

As in the previous section, we suppose that $\w \in E^1(X)\comptensor \u$ is
an integrable, $\G$-invariant 1-form on $X$. Define
$$
\Theta_{x,y} : \Pi(M;x,y) \to R\ltimes \U \text{ by } (\gamma,c_\gamma)
\mapsto \big(\rho(\gamma),T(c_\gamma)^{-1}\big).
$$
Note that, unless $x=y$, this is not a group homomorphism. The universal
mapping property of $\cG_{x,y}$ implies that $\Theta_{x,y}$ induces a morphism
$\cG_{x,y} \to R\ltimes \U$ of affine schemes such that the diagram
$$
\xymatrix{
\Pi(M;x,y) \ar[r]\ar[dr] & \cG_{x,y} \ar[d] \cr
& R\ltimes \U.
}
$$
If $\w$ satisfies the assumptions of Proposition~\ref{prop:char}, then the
vertical morphism is an isomorphism. This follows as, in this case, $\cG_{x,y}$
and $R\ltimes \U$ are both torsors under the left action of the relative
completion $\cG_{x,x} \cong R\ltimes\U$ of $\G \cong \pi_1(M,x)$ with respect
to $\rho$.

If $R$ and $\rho$ are defined over $\Q$, then $\cG_{x,y}$ has a natural $\Q$
structure consisting of those elements of $\cO(\cG_{x,y})$ that take rational
values on the image of $\Pi(M;x,y)$ in $\cG_{x,y}$.

\section{Zucker's Mixed Hodge Complex}
\label{sec:mhc}

In this section we recall the construction of the natural MHS on the cohomology
of a smooth curve with coefficients in a polarized variation of Hodge structure.

Suppose that $C$ is a compact Riemann surface and that $D$ is a finite subset,
which we assume to be non-empty. Then $C':= C-D$ is a smooth affine curve.
Suppose that $\V$ is a polarized variation of Hodge structure over $C'$ of
weight $m$. For simplicity, we assume that the  local monodromy about each $x\in
D$ is unipotent. Zucker \cite[\S 13]{zucker} constructs a cohomological Hodge
complex $\bK(\V)$ that computes the MHS on $H^\dot(C',\V)$. In this section we
recall the definition of its complex component $\bK_\C(\V)$, together with its
Hodge and weight filtrations. We first recall a few basic facts about mixed
Hodge complexes.

\subsection{Review of mixed Hodge complexes}

This is a very brief outline of how one constructs a mixed Hodge structure on a
graded vector space using a mixed Hodge complex. Full details can be found in
\cite{deligne:hodge2}.

The standard method for constructing a mixed Hodge structure on a graded
invariant $M^\dot$ of a complex algebraic variety is to express the invariant as
the cohomology of a mixed Hodge complex (MHC). Very briefly, a MHC $\bK$
consists of:
\begin{enumerate}

\item two complexes $K^\dot_\Q$ and $K^\dot_\C$, each endowed with a weight
filtration $W_\dot$ by subcomplexes,

\item a $W_\dot$ filtered quasi-isomorphism between $K_\Q\otimes\C$ and $K_\C$.

\item a Hodge filtration $F^\dot$ of $K_\C^\dot$ by subcomplexes.

\end{enumerate}
These are required to satisfy several technical conditions, which we shall
omit, although the lemma below encodes some of them. The complexes $K_\Q^\dot$
and $K_\C^\dot$ compute the $\Q$- and $\C$-forms of the invariant $M$:
$$
M_\Q^\dot \cong H^\dot(K_\Q^\dot) \text{ and } M_\C^\dot \cong H^\dot(K_\C^\dot)
$$
The quasi-isomorphism between them is compatible with these isomorphisms. The
weight filtration of $K_\Q^\dot$ induces a weight filtration of $M_\Q$ by
$$
W_m M^j_\Q = \im\{H^j(W_{m-j}K^\dot_\Q) \to M^j_\Q\}.
$$
The assumption that the quasi-isomorphism between $K^\dot_\Q\otimes\C$ and
$K^\dot_\C$ be $W_\dot$-filtered implies that the weight filtrations of
$K_\Q^\dot$ and $\K_\C^\dot$ induce the same weight filtration on $M^\dot$. That
is,
$$
(W_m M^j_\Q)\otimes\C = \im\{H^j(W_{m-j}K^\dot_\C) \to M^j_\C\}
$$
Finally, the Hodge filtration of $K_\C^\dot$ induces the Hodge filtration
of $M^\dot_\C$ via
$$
F^p M^\dot_\C := \im\{H^\dot(F^p K^\dot_\C) \to M_\C^\dot\}.
$$
If $\bK$ is a MHC, then $M^\dot$ is a MHS with these Hodge and weight
filtrations.

We shall need the following technical statement. As pointed out above, this is
equivalent to several of the technical conditions satisfied by a MHC.

\begin{lemma}[{\cite[3.2.8]{hain:dht1}}]
\label{lem:exact}
Suppose that $(K^\dot_\C,W_\dot,F^\dot)$ is the complex part of a MHC. If $u \in
F^pW_m K^j_\C$ is exact in $K^\dot_\C$, then there exists $v\in
F^pW_{m+1}K_\C^{j-1}$ such that $dv = u$. 
\end{lemma}

Finally, a {\em cohomological mixed Hodge complex} (CMHC) is a collection of
filtered complexes of sheaves on a variety (or a topological space) with the
property that the global sections of a collection of acyclic resolutions of its
components is a MHC. For details, see \cite{deligne:hodge2}.

\subsection{Zucker's cohomological MHC}
\label{sec:zucker}

We'll denote Zucker's cohomological MHC for computing the MHS on $H^\dot(C',\V)$
by $\bK(\V)$. We describe only its complex part $\bK_\C(\V)$. When trying to
understand the definition of its weight filtration, the reader may find it
useful to read the following section on the limit MHS of $\V$ at a tangent
vector to a cusp $P\in D$.

Set $\cV = \V\otimes\cO_{C'}$. This has a canonical flat connection $\nabla$.
This extends to a meromorphic connection
$$
\nabla : \Vbar \to \Omega^1_C(\log D) \otimes \Vbar
$$
on Deligne's canonical extension $\Vbar$ of it to $C$. Schmid's Nilpotent Orbit
Theorem \cite{schmid} implies that the Hodge sub-bundles of $\cV$ extend to
holomorphic sub-bundles $F^p\Vbar$ of $\Vbar$.

As a sheaf, $\bK_\C(\V)$ is simply $\Omega^\dot_C(\log D)\otimes \Vbar$ with the
differential $\nabla$. Standard arguments imply that $H^\dot(C,\bK_\C(\V))$ is
isomorphic to $H^\dot(C',\V)$. Its Hodge filtration is defined in the obvious
way:
$$
F^p \bK_\C(\V)
:= \sum_{s+t = p} \big(F^s \Omega^\dot_C(\log D)\big)\otimes F^t\Vbar.
$$
In degree 0 the weight filtration is simply
$$
0 = W_{m-1} \bK_\C^0(\V) \subseteq W_m \bK_\C^0(\V) = \bK_\C^0(\V).
$$
In degree 1, $W_r \bK_\C^1(\V)$ vanishes when $r<m$. To define the remaining
terms in degree 1, consider the reside mapping
$$
\Res_P : \Omega_C^1(\log D)\otimes \Vbar \to V_P
$$
at $P\in C$, which takes values in the fiber $V_P$ of $\Vbar$ over $P$. The
residue $N_P$ of the connection $\nabla$ on $\Vbar$ at $P\in C$ is the local
monodromy logarithm divided by $2\pi i$. It acts on $V_P$. When $r\ge 0$, the
stalk of $W_{m+r} \bK_\C^1(\V)$ at $P$ is
$$
W_{m+r} \bK_\C^1(\V)_P := \Res_P^{-1} (\im N_P + \ker N_P^{r}).
$$
Note that, when $P\notin D$, the reside map vanishes (and so does $N_P$), so
that the stalk of $W_m\bK_\C^1(\V)$ when $P\notin D$ is $\bK_\C(\V)_P$.

The Hodge and weight filtrations on $H^j(C',\V)$ are defined by
$$
F^p H^j(C',\V) = \im\{H^j(C,F^p\bK_\C(\V))\to H^j(C',\V)\}
$$
and
$$
W_m H^j(C',\V) = \im\{H^j(C,W_{m-j}\bK_\C(\V))\to H^j(C',\V)\}.
$$
The definition of the weight filtration implies that $H^0(C',\V)$ has weight $m$
and that the weights on $H^1(C',\V)$ are $\ge 1+m$.

\begin{remark}
Let $j : C' \to C$ denote the inclusion. The complex of sheaves $W_m \bK_\C(\V)$
on $C$ is a cohomological Hodge complex that is easily seen to be
quasi-isomorphic to the sheaf $j_\ast \V$ on $C$. It therefore computes the
intersection cohomology $I\!H^j(C,\V)$ and shows that it has a canonical pure
Hodge structure of weight $m+j$. For more details, see Zucker's paper
\cite{zucker}.
\end{remark}

\subsection{The limit MHS on $V_P$}

Suppose that $P\in D$. For each choice of a non-zero tangent vector $\vv \in
T_PC$ there is a limit MHS, denoted $V_\vv$ on $V_P$. The $p$th term of the
Hodge filtration is the fiber of $F^p\Vbar$ over $P$. The weight filtration is
the monodromy weight filtration shifted so that its average weight is $m$, the
weight of $\V$. The $\Q$ (or $\Z$ structure, if that makes sense) is constructed
by first choosing a local holomorphic coordinate $t$ defined on a disk $\D$
containing $P$ where $t(P)=0$. Assume that $\D\cap D=\{P\}$. Standard ODE theory
(cf.\ \cite[Chapt.~II]{wasow}) implies that there is a trivialization $\D\times
V_P$ of the restriction of $\Vbar$ to $\D$ such that the connection $\nabla$ is
given by
$$
\nabla f = df + N_P(f) \frac{dt}{t}
$$
with respect to this trivialization, where $f : \D \to V_P$ . Suppose that $Q\in
\D-\{P\}$. The $\Q$-structure on $V_P$ corresponding to the tangent vector $\vv
= t(Q)\partial/\partial t$ is obtained from the $\Q$ structure $V_{Q,\Q}$  on
$V_Q$ by identifying $V_P$ with $V_Q$ via the trivialization. This MHS depends
only on the tangent vector and not on the choice of the holomorphic coordinate
$t$.

For all non-zero $\vv \in T_P C$, the monodromy logarithm $N_P : V_\vv \to
V_\vv$ acts as a morphism of type $(-1,-1)$. This implies that $V_\vv/\im N_P$
has a natural MHS for all $\vv\neq 0$. Since $N_P$ acts trivially on this, the
MHS on $V_P/\im N_P$ has a natural MHS that is independent of the choice of the
tangent vector $\vv\in T_P C$. The definition of the weight filtration on
$V_\vv$ implies that the weight filtration on $V_P/\im N_P$ is
$$
W_{m+r} (V_P/\im N_P) = (\im N_P + \ker N_P^{r-1})/\im N_P
$$
when $r\ge 0$ and $W_{m+r} (V_P/\im N_P) = 0$ when $r<0$.

There is a canonical isomorphism
$$
H^1(\D-\{P\},\V) \cong V_P/\im N_P.
$$
from which it follows that this cohomology group has a canonical MHS for each
$P\in D$.

\subsection{An exact sequence}
\label{sec:exact_seqce}

Observe that
$$
\bK_\C(\V)/W_m\bK_\C(\V) = \bigoplus_{P\in D}i_{P\ast} (V_P/N_PV_P)(-1)[-1].
$$
This and the exact sequence of sheaves
$$
0 \to W_m \bK(\V) \to \bK(\V) \to \bK(\V)/W_m \to 0
$$
on $C$ gives the exact sequence of MHS
\begin{multline}
\label{eqn:exact_seqce}
0 \to W_{m+1}H^1(C',\V) \to H^1(C',\V)
\cr
\to \bigoplus_{P\in D} (V_P/\im N_P)(-1) \to I\!H^2(C,\V) \to 0.
\end{multline}
Here we are assuming that we are in the ``interesting case'' where $D$ is
non-empty.

Since $H^0(C',\V) = I\!H^0(C,\V)$, and since this is dual to $I\!H^2(C,\V)$, we
see that the sequence
$$
0 \to W_{m+1}H^1(C',\V) \to H^1(C',\V) \to \bigoplus_{P\in D} (V_P/\im N_P)(-1)
\to 0
$$
is exact when $H^0(C',\V)=0$.

\subsection{A MHC of smooth forms}

To extend this MHS from the cohomology groups $H^1(C',\V)$ to one on relative
completion of its fundamental group, we will need the complex part of a global
MHC of smooth forms. The construction of this from $\bK_\C(\V)$ is standard.
We recall the construction.

The resolution of $\bK_\C(\V)$ by smooth forms is the total complex of the
double complex
\begin{equation}
\label{eqn:sheaf}
\K^{\dot\dot}_\C(\V) := \bK_\C(\V) \otimes_{\cO_C} \E^{0,\dot}_C,
\end{equation} 
where $\E^{0,\dot}_C$ denotes the sheaf of smooth forms on $C$ of type
$(0,\dot)$. The Hodge and weight filtrations extend as
$$
F^p\K^{\dot\dot}_\C(\V) := \K^{\ge p,\dot}_\C(\V)
= (F^p\bK_\C(\V)) \otimes_{\cO_C} \E^{0,\dot}_C
$$
and
$$
W_r\K^{\dot,\dot}_\C(\V) := (W_r\bK_\C(\V)) \otimes_{\cO_C} \E^{0,\dot}_C.
$$
The global sections
$$
K^\dot(C,D;\V) := H^0(C,\tot\K^{\dot,\dot}_\C(\V)).
$$
of (\ref{eqn:sheaf}) is a sub dga of $E^\dot(C',\V)$. It has Hodge and weight
filtrations defined by taking the global sections of the Hodge and weight
filtrations of (\ref{eqn:sheaf}). It is the complex part of a mixed Hodge
complex.

The Hodge and weight filtrations on $H^j(C',\V)$ are
$$
F^p H^j(C',\V) = \im\{H^j(F^p K^\dot(C,D;\V)) \to H^j(C',\V)\}
$$
and
$$
W_m H^j(C',\V) = \im\{H^j(W_{m-j} K^\dot(C,D;\V)) \to H^j(C',\V)\}.
$$

Zucker's MHC and its resolution by smooth forms are natural in the local system
$\V$ and are compatible with tensor products: if $\V_1$, $\V_2$ and $\V_3$ are
PVHS over $C'$, then a morphism $\V_1 \otimes \V_2 \to \V_3$ of PVHS induces
morphism
$$
\bK(\V_1)\otimes \bK(\V_2) \to \bK(\V_1\otimes \V_2) \to \bK(\V_3).
$$
of CMHCs and dga homomorphism
$$
K^\dot(C,D;\V_1)\otimes K^\dot(C,D;\V_2) \to K^\dot(C,D;\V_1\otimes \V_2)
\to K^\dot(C,D;\V_3)
$$
that preserve the Hodge and weight filtrations.

\subsection{Remarks about the orbifold case}
\label{rem:orbi-case}

Zucker's work extends formally to the orbifold case. For us, an orbi-curve
$C'=C-D$ is the orbifold quotient of a smooth curve $X'=X-E$ by a finite group
$G$. This action does not have to be effective. (That is, $G \to \Aut X$ does
not have to be injective.) An orbifold variation of MHS $\V$ over $C'$ is an
admissible variation of MHS $\V_X$ over $X'$ together with a $G$-action such
that the projection $\V \to X'$ is $G$-equivariant. For each $g\in G$, we
require that the map $g : \V_X \to \V_X$ induce an isomorphism of variations of
MHS $g^\ast \V_X \cong \V_X$.

With these assumptions, $G$ acts on $\bK^\dot(X,E;\V_X)$ and
$\bK^\dot(X,E;\V_X)^G$ is a sub MHC. Define
$$
\bK^\dot(C,D;\V) = \bK^\dot(X,E;\V_X)^G.
$$
This computes the cohomology $H^\dot(C',\V)$ and implies that it has a MHS such
that the canonical isomorphism $H^\dot(C',\V) \cong H^\dot(X',\V_X)^G$ is an
isomorphism of MHS.

\section{Relative Completion of Fundamental Groups of Affine Curves}
\label{sec:relcomp_aff}

In this section we use the results of the last two sections to construct, under
suitable hypotheses, a MHS on the relative completion of the fundamental group
of an affine curve. As in the previous section, we suppose that $C$ is a compact
Riemann surface and that $D$ is a finite subset of $C$. Here we suppose, in 
addition, that $D$ is non-empty, so that $C'=C-D$ is an affine curve. Suppose
that $\V$ is a polarized variation of $\Q$-HS over $C'$ with unipotent monodromy
about each $P\in D$. Denote the fiber of $\V$ over $x$ by $V_x$. The Zariski
closure of the monodromy representation
$$
\rho : \pi_1(C',x) \to \Aut(V_x)
$$
is a reductive $\Q$-group \cite[Lem.~2.10]{simpson}, which we will denote by
$R_x$. Fix a base point $x_o\in C$. Set $R=R_{x_o}$. Each monodromy group $R_x$
is isomorphic to $R$; the isomorphism is unique mod inner automorphisms. We thus
have Zariski dense monodromy representations
$$
\rho_x : \pi_1(C',x) \to R_x(\Q)
$$
for each $x\in X$. Denote the completion of $\pi_1(C',x)$ with respect to
$\rho_x$ by $\cG_x$, and its Lie algebra by $\g_x$. We will construct natural
MHSs on $\cO(\cG_x)$ and on $\g_x$ that are compatible with their algebraic
structures (Hopf algebra, Lie algebra). Before doing this we need to show that
the connection form $\Omega$ can be chosen to have coefficients in Zucker's MHC
and be compatible with the various Hodge and weight filtrations.

For simplicity, we make the following assumptions:
\begin{enumerate}

\item Every irreducible representation of $R$ is absolutely irreducible. That
is, it remains irreducible when we extend scalars from $\Q$ to $\C$.

\item For each irreducible representation $V_\lambda$ of $R$, the corresponding
local system $\V_\lambda$ over $C'$ underlies a PVHS over $C'$. The Theorem of
the Fixed Part (stated below) implies that this PVHS is unique up to Tate twist.

\end{enumerate}
These hold in our primary example, where $C'$ is a quotient $X_\G$ of the upper
half plane by a finite index subgroup of $\SL_2(\Z)$ and $R\cong \SL_2$.

\subsection{The bundle $\bu$ of Lie algebras}

Denote the set of isomorphism classes of irreducible $R$-modules by $\Rdual$.
Fix a representative $V_\lambda$ of each $\lambda \in \Rdual$ and the structure
of a PVHS on the corresponding local system $\V_\lambda$ over $C'$. Filter
$$
\Rdual_1 \subset \Rdual_2 \subset \Rdual_3 \subset \cdots \subset
\bigcup_n \Rdual_n = \Rdual
$$
$\Rdual$ by finite subsets such that if $\lambda \in \Rdual_n$ and $\mu \in
\Rdual_m$, then the isomorphism class of $V_\lambda \otimes V_\mu$ is in
$\Rdual_{m+n}$. For example, when $R=\SL_2$, one can take $\Rdual_n$ to be the
set of all symmetric powers $S^m H$ with $m\le n$ of the defining representation
$H$ of $\SL_2$.

For each $\lambda \in \Rdual$, the variation MHS $H^1(C',\V_\lambda)^\ast\otimes
\V_\lambda$, being the tensor product of a constant MHS with a PVHS, is an
admissible variation of MHS over $C'$. Note that the VMHS structure on it does
not change when $\V_\lambda$ is replaced by $\V_\lambda(n)$, so that the VMHS
$H^1(C',\V_\lambda)^\ast\otimes \V_\lambda$ is independent of the choice of the
PVHS structure on $\V_\lambda$. Since the weights of $H^1(C,\V_\lambda)$ are at
least $1+$ the weight of $\V_\lambda$, the weights of
$H^1(C',\V_\lambda)^\ast\otimes \V_\lambda$ are $\le -1$. 

The inverse limit
$$
\bu_1 := \varprojlim_{n}\,\bigoplus_{\lambda \in \Rdual}
H^1(C',\V_\lambda)^\ast\otimes \V_\lambda
$$
is pro-variation of MHS over $\C$ of negative weight.\footnote{Note that this is
a very special kind of variation of MHS --- it is a direct sum of variations
that are the tensor product of a constant MHS with a PVHS. Their asymptotic
behaviour is determined by that of the PVHS that occur in the summands.} That
is, $\bu_1=W_{-1}\bu_1$. Observe that its fiber
$$ 
\varprojlim_{n}\,\bigoplus_{\lambda \in \Rdual}
H^1(C',\V_\lambda)^\ast\otimes V_{\lambda,o} =
\prod_{\lambda\in \Rdual}H^1(C',\V_\lambda)^\ast\otimes V_{\lambda,o}
$$
over $x_o$ is the abelianization of the prounipotent radical of the
completion of $\pi_1(C',x_o)$ with respect to the homomorphism to $R$.

The degree $n$ part $V \mapsto \L_n(V)$ of the free Lie algebra is a Schur
functor, so that it makes sense to apply it to bundles. Set
$$
\bu_n = \L_n(\bu_1) :=
\varprojlim_{n}\,\L_n\big(\bigoplus_{\lambda \in \Rdual}
H^1(C',\V_\lambda)^\ast\otimes \V_\lambda\big)
$$
and
$$
\bu := \varprojlim_{n} \bigoplus_{j=1}^n\bu_j \text{ and }
\bu^N := \varprojlim_{n\ge N} \bigoplus_{j=N}^n\bu_j .
$$
These are admissible pro-variations of MHS over $C'$. Denote the fiber of $\bu$
over $x$ by $\u_x$. It is abstractly isomorphic to the prounipotent radical of
the completion of $\pi_1(C',x)$ relative to the monodromy representation
$\pi_1(C',x) \to R$. The fiber of $\bu^N$ over $x$ is the $n$th term $L^n\u_x$
of the lower central series (LCS) of $\u_x$.

\subsection{Some technicalities}
\label{sec:tech}

The Theorem of the Fixed Part states that if $\A$ is an admissible VMHS over a
smooth variety $X$, then $H^0(X,\A)$ has a natural MHS with the property that
for each $x\in X$, the natural inclusion $H^0(X,\A) \to A_x$ is a morphism of
MHS. In the algebraic case, it is enough to prove this when $X$ is a curve. When
$\A$ is pure, this follows from Zucker's MHS \cite{zucker} on $H^\dot(X,\V)$.
The general case follows from Saito's theory of Hodge and mixed Hodge modules
\cite{saito:hm,saito:mhm}.

The following is a direct consequence of the Theorem of the Fixed Part. Its
proof is left as an exercise.

\begin{lemma}
\label{lem:isotypic}
Assume that $\A$ is an admissible VMHS over $C'$ whose monodromy representation
$\pi_1(C',x_o) \to \Aut A_{x_o}$ factors through $\pi_1(C',x_o) \to R$. With the
assumptions above, each $R$ isotypical component of $\A$ over $C'$ is an
admissible VMHS. If $\V_\lambda$ is a PVHS that corresponds to the irreducible
$R$-module $V_\lambda$, then the natural mapping
$$
\bigoplus_{\lambda \in \Rdual} H^0(C',\V_\lambda)\otimes \V_\lambda \to \A
$$
is an isomorphism of admissible VMHS. In particular, the structure of a PVHS
on $\V_\lambda$ is unique up to Tate twist. \qed
\end{lemma}

Every pro object of the category of admissible VMHS $\A$ over $C'$ is thus of
the form
$$
\A = \prod_{\lambda \in \Rdual}  \V_\lambda \otimes A_\lambda,
$$
where each $A_\lambda$ is a MHS. Define
$$
\bK(\A) = \prod_{\lambda \in \Rdual} \bK( \V_\lambda)\otimes A_\lambda.
$$
This is a pro-CMHC. In particular, its complex part
$$
K^\dot(C,D;\A) =
\prod_{\lambda \in \Rdual} K^\dot(C,D;\V_\lambda)\otimes A_\lambda
$$
has naturally defined Hodge and weight filtrations; its differential is strict
with respect to the Hodge and weight filtrations (cf.\ Lemma~\ref{lem:exact}).
In particular, for all $n\ge 1$, the complexes $K^\dot(C,D;\bu_n)$ have this
strictness property.

\subsection{The connection form $\Omega$}
\label{sec:omega}

Observe that $H^1(C',H^1(C',\V_\lambda)^\ast\otimes\V_\lambda)$ is naturally
isomorphic (as a MHS) to
$$
H^1(C',\V_\lambda)\otimes H^1(C',\V_\lambda)^\ast
\cong \Hom(H^1(C',\V_\lambda),H^1(C',\V_\lambda)).
$$
So, for each $\lambda \in \Rdual$, there is an element $\xi_\lambda \in F^0 W_0
H^1(C',H^1(C',\V_\lambda)^\ast\otimes\V)$ that corresponds to the identity
mapping $H^1(C',\V_\lambda)\to H^1(C',\V_\lambda)$. Lemma~\ref{lem:exact}
implies that this is represented by a 1-form
$$
\w_\lambda \in F^0 W_{-1} K^1(C,D;H^1(C';\V_\lambda)^\ast\otimes \V_\lambda).
$$
Set
$$
\Omega_1 := \prod_{\lambda \in \Rdual} \w_\lambda \in K^1(C,D;\bu_1).
$$
Note that $d\Omega_1 =0$ and that
$$
\frac{1}{2}[\Omega_1,\Omega_1] \in F^0 W_{-2}K^2(C,D;\bu_2).
$$
Since $C$ is a surface, $[\Omega_1,\Omega_1]$ is closed. Since $C'$ is not
compact, it is exact. Lemma ~\ref{lem:exact} implies that we can find $\Xi_2$
in $F^0W_{-1}K^1(C,D;\bu_2)$ such that $d\Xi_2 =
\frac{1}{2}[\Omega_1,\Omega_1]$.
Set
$$
\Omega_2 = \Omega_1 - \Xi_2 \in F^0 W_{-1} K^1(C,D;\bu_1\oplus\bu_2).
$$
Then
$$
d\Omega_2 + \frac{1}{2}[\Omega_2,\Omega_2] \in F^0 W_{-2} K^2(C,D;\bu^3).
$$
Its component that lies in $K^2(C,D;\bu_3)$ is closed and thus exact. So it is
the exterior derivative of some $\Xi_3 \in F^0 W_{-2}K^1(C,D;\bu_3)$. Set
$$
\Omega_3 = \Omega_2 - \Xi_3 \in F^0 W_{-1} K^1(C,D;\bu_1\oplus\bu_2\oplus\bu_3).
$$
Then
$$
d\Omega_3 + \frac{1}{2}[\Omega_3,\Omega_3] \in F^0 W_{-2} K^2(C,D;\bu^4).
$$
Continuing this way, we obtain a sequence of elements $\Xi_n \in
F^0W_{-1}K^1(C,D;\bu_n)$ such that for all $N\ge 2$
$$
\Omega_N := \Omega_1 -
(\Xi_2 + \cdots + \Xi_N) \in F^0 W_{-1}K^1(C,D;\oplus_{n=1}^N\bu_n)
$$
satisfies
$$
d\Omega_N + \frac{1}{2}[\Omega_N,\Omega_N] \in F^0 W_{-2} K^2(C,D;\bu^{N+1}).
$$
Then the $\bu$-valued 1-form
\begin{equation}
\label{eqn:filts}
\Omega := \varprojlim_N \Omega_N \in F^0 W_{-1} K^1(C,D;\bu)
\end{equation}
is integrable:
$$
d\Omega + \frac{1}{2}[\Omega,\Omega] = 0.
$$

To understand the significance of the form $\Omega$, note that  the bundle
$\bu_1$ over $C'$, and hence each $\bu_n = \L_n(\bu_1)$, is a flat bundle over
$C'$. The monodromy of each factors through the representation $\rho_x :
\pi_1(C',x) \to R$. Summing these, we see that for each $N\ge 1$, the bundle
$$
\bu/\bu^{N+1} \cong \bu_1 \oplus \cdots \oplus \bu_N
$$
is flat with monodromy that factors through $\rho_x$. Denote the limit of these
flat connections by $\nabla_0$. Then
$$
\nabla := \nabla_0 + \Omega
$$
defines a new connection on the bundle $\bu$ over $C'$ which is flat as $\Omega$
is integrable. Here we view $\bu$ (and hence $\Omega$) as acting on each fiber
by inner derivations.

The restriction of the filtration
\begin{equation}
\label{eqn:lcs}
\bu = \bu^1 \supset \bu^2 \supset \bu^3 \supset \cdots
\end{equation}
of $\bu$ to each fiber is the lower central series. Note that
$$
\Gr_\LCS^\dot \bu := \bu^n/\bu^{n+1}
$$
is naturally isomorphic to $\bu_n$.

\begin{lemma}
\label{lem:lcs}
Each term $\bu^n$ of the lower central series filtration of $\bu$ is a flat
sub-bundle of $(\bu,\nabla)$. The induced connection on $\Gr^n_\LCS\bu \cong
\bu_n$ is $\nabla_0$.
\end{lemma}

\begin{proof}
This follows from the fact that $\Omega$ takes values in $\bu$ and that the
inner derivations act trivially on $\Gr^\dot_\LCS\bu$.
\end{proof}

\subsection{Hodge and weight bundles and their extensions to $C$}
\label{sec:holo}

The flat connection $\nabla$ on $\bu$ defines a new complex structure as a (pro)
holomorphic vector bundle over $C'$. To understand it, write $\Omega = \Omega' +
\Omega''$, where $\Omega'$ is of type $(1,0)$ and $\Omega''$ is of type $(0,1)$.
Set
$$
D' = \nabla_0 + \Omega' \text{ and } D'' = \delbar + \Omega''
$$
so that $\nabla = D' + D''$. Then $D''$ is a $(0,1)$-valued form taking values
in $\bu$. Note that $(D'')^2=0$.

\begin{lemma}
A section $s$ of $\bu$ is holomorphic with respect to the complex structure on
$\bu$ defined by the flat connection $\nabla$ if and only if $D''s=0$.
\end{lemma}

\begin{proof}
Since $D''$ is $\cO_{C'}$ linear, it suffices to show that $D''s=0$ when $s$ is
a flat local section of $\bu$. But this follows as $D''s$ is the $(0,1)$
component of $\nabla s$, which vanishes as $s$ is flat.
\end{proof}

Denote $\bu$ with this complex structure by $(\bu,D'')$. Since all holomorphic
sections of $(\bu,D'')$ are $\cO$-linear combinations of flat sections,
Lemma~\ref{lem:lcs} implies:\footnote{This also follows from the fact that
$\Omega''$ is $\bu$-valued thus acts trivially on the graded quotients of the
lower central series filtration of $\bu$.}

\begin{lemma}
The lower central series filtration (\ref{eqn:lcs}) of $\bu$ is a filtration by
holomorphic sub-bundles. The isomorphism
$$
\Gr^n_\LCS (\bu,D'')\cong \bu_n
$$
is an isomorphism of holomorphic vector bundles. \qed
\end{lemma}

Denote the canonical extension of $\bu_1$ to $C$ by $\bubar_1$. Then, since the
local monodromy operators are unipotent, $\bubar_n := \L_n(\bubar_1)$ is the
canonical extension of $\bu_n$ to $C$. Define
$$
\bubar := \varprojlim_{n} \bigoplus_{j=1}^n\bubar_j \text{ and }
\bubar^N := \varprojlim_{n\ge N} \bigoplus_{j=N}^n\bubar_j .
$$
Then $(\bubar,\nabla_0)$ is the canonical extension of $(\bu,\nabla_0)$ to $C$.

Our next task is to show that $(\bubar,\nabla)$ is the canonical extension of
$(\bu,\nabla)$ to $C$. Since smooth logarithmic $(0,1)$-forms on $C$ with poles
on $D$ are smooth on $C$, it follows that $\Omega''$ is a smooth,
$\bubar$-valued $(0,1)$-form on $C$. It follows that $D''$ extends to a $(0,1)$
form-valued operator on smooth sections of $\bubar$. Since $(D'')^2=0$, it
defines a complex structure on $\bubar$. A smooth locally defined section $s$ of
$\bubar$ is holomorphic if and only if $D''s = 0$. We'll denote this complex
structure by $(\bubar,D'')$.

Suppose that $P\in D$ and that $t$ is a local holomorphic coordinate on $C$
centered at $P$. Since $t\Omega$ is a smooth $\bubar$-valued form on $C$ in a
neighbourhood $\Delta$ of $P$, and since $t\nabla_0$ takes smooth sections of
$\bubar$ defined on $\Delta$ to smooth 1-forms with values in $\bubar$, it
follows that $t\nabla$ is a differential operator on sections of $\bubar$ over
$\Delta$. This implies that $\nabla$ is a meromorphic connection on
$(\bubar,D'')$ with regular singular points along $D$.

\begin{proposition}
\label{prop:canon_extn}
The bundle $(\bubar,D'')$ with the connection $D'$ is Deligne's canonical
extension of $(\bu,\nabla)$ to $C$.
\end{proposition}

\begin{proof}
Since all singularities of $\nabla$ are regular singular points, it suffices to
check that the residue of $\nabla$ at each $P\in D$ is pronilpotent
endomorphism of $\u_P$, the fiber of $\bubar$ over $P$. This is an immediate
consequence of the fact that the residue of $(\Gr^\dot_\LCS\bu,\nabla)
\cong(\bu,\nabla_0)$ at $P$ is pronilpotent by assumption and that the residue
of $\Omega$ at $P$ is an element of $\u_P$, which acts trivially on
$\Gr^\dot_\LCS\u_P$.
\end{proof}

We now turn our attention to the behaviour of the Hodge bundles. Since each
$\bu_n$ is a sum of variations of MHS that are tensor products of a constant MHS
with a PVHS, they behave well near each cusp $P\in D$. In particular, the Hodge
bundles $F^p\bu_n$ extend to sub-bundles of $\bubar_n$. This implies that the
the Hodge bundles $F^p\bu$ extend to holomorphic sub-bundles of $\bubar$.
Consequently, they extend as $C^\infty$ sub-bundles of $(\bubar,D'')$.

\begin{lemma}
\label{lem:hodge_bdles}
The Hodge sub-bundles of $\bubar$ are holomorphic and the connection $\nabla$
satisfies Griffiths transversality: if $s$ is a local holomorphic section of
$F^p\bubar$, then $\nabla s$ is a local section $\Omega^1_C(\log D)\otimes
F^p\bubar$.
\end{lemma}

\begin{proof}
To prove that $F^p\bubar$ is a holomorphic sub-bundle with respect to the
complex structure $D''$, it suffices to show that if $s$ is a local $C^\infty$
section of $F^p\bubar$, then $D''s$ is a $(0,1)$-form with values in
$F^p\bubar$. Since $F^p\bu$ is a holomorphic sub-bundle of $(\bu,\delbar)$, it
follows that $\delbar s$ is a $(0,1)$-form with values in $F^p\bubar$. And since
$$
\Omega'' \in F^0 K^{0,1}(C,D;\bu) = E^{0,1}(C)\comptensor F^0\bubar
$$
it follows that $\Omega''(s) \in E^{0,1}(C)\comptensor F^p\bubar$, which implies
that $D''s \in E^{0,1}(C)\comptensor F^p\bubar$, as required.

Griffiths transversality follows for similar reasons. Suppose that $s$ is a
local $C^\infty$ section of $F^p\bu$. Since $(\bu,\nabla_0)$ satisfies Griffiths
transversality, $\nabla_0 s$ is a 1-form valued section of $F^{p-1}\bubar$.
Since
$$
\Omega \in F^0 K^1(C,D;\bu) \subseteq E^1(C\log D)\comptensor F^{-1}\bubar,
$$
$\Omega(s)$ is a 1-form valued section of $F^{p-1}\bubar$. It follows that
$\nabla(s)$ takes values in $F^{-1}\bubar$.
\end{proof}

\begin{lemma}
\label{lem:wt_bdles}
The weight sub-bundles $W_m \bu$ are flat sub-bundles of $(\bubar,\nabla)$.
Moreover, the identity induces an isomorphism of PVHS
$$
(\Gr^W_m\bubar,\nabla) \cong (\Gr^W_m\bubar,\nabla_0).
$$
\end{lemma}

\begin{proof}
Both assertions follow from the fact that $\Omega$ (and hence $\Omega''$ as
well) takes values in $\bu$ and that $\bu = W_{-1}\bu$. This implies that the
adjoint action of $\Omega$ and $\Omega''$ on $\Gr^W_\dot\bu$ is trivial. It
follows that $\nabla$ respects the weight filtration of $\bu$ and that the
induced connection on $\Gr^W_\dot\bu$ is $\nabla_0$.
\end{proof}

Since $\Omega$ acts trivially on $\Gr^\dot_\LCS\bu$, we have:

\begin{lemma}
\label{lem:gr_lcs}
For all $n\ge 1$ there is a natural isomorphism of 
$$
(\Gr^n_\LCS\bu,\nabla) \cong (\Gr^n_\LCS\bu,\nabla_0) \cong \bu_n
$$
local systems. \qed
\end{lemma}

\subsection{The ind-variation $\OO(P_x)$}
\label{sec:OO}

Here we dispense of a few technicalities, in preparation for the construction of
a MHS on $\cO(\cG_x)$ in the next section.

Denote the fibers of $\V$ over $x,y\in C'$ by $V_x$ and $V_y$, respectively.
Denote the Zariski closure of the image of the parallel transport mapping
$$
\rho_{x,y} : \Pi(C';x,y) \to \Hom(H_x,H_y)
$$
by $R_{x,y}$. Set $R_x = R_{x,x}$. Then $R_{x,y}$ is a left $R_x$ torsor and a
right $R_y$ torsor.

The theorem of the fixed part (Sec.~\ref{sec:tech}) implies that the coordinate
ring of $R_{x,y}$ is an algebra in the category of ind-Hodge structures of
weight 0. This HS is characterized by the property that the monodromy coaction
\begin{equation}
\label{eqn:coaction}
E_x \to \cO(R_{x,y}) \otimes E_y
\end{equation}
is a morphism of HS for all polarized variations of Hodge structure $\bE$ over
$C'$ whose monodromy representation factors through $\rho_x : \pi_1(C',x) \to
R_x$. (Here $E_z$ denotes the fiber of $\bE$ over $z\in C'$.) Under our
assumption that each irreducible representation $V_\lambda$ of $R$ is absolutely
irreducible and that the corresponding local system $\V_\lambda$ underlies a
PVHS over $C'$, there is an isomorphism of HS
$$
\cO(R_{x,y}) \cong \bigoplus_\lambda \Hom(V_{\lambda,x},V_{\lambda,y})^\ast.
$$
When $C' = \G\bs \h$, $\G$ a finite index subgroup of $\SL_2(\Z)$, and $\H$ is
the standard variation of HS,
$$
\cO(R_{x,y}) = \bigoplus_{n\ge 0} \Hom(S^nH_x,S^n H_y)^\ast.
$$

Denote the local system over $C'$ whose fiber over $y$ is $\cO(R_{y,x})$ by
$\OO_x$. In concrete terms
$$
\OO_x = \bigoplus_\lambda \Hom(\V_\lambda,V_{\lambda,x})^\ast
\cong \bigoplus_\lambda \V_\lambda\otimes V_{\lambda,x}^\ast.
$$
Note that this is local system of algebras and that there is a natural left
$R_x$ action that preserves the algebra structure.

For every local system $\bE$ whose monodromy representation factors through
$\rho : \pi_1(X,x) \to R_x$, there is a natural isomorphism
\begin{equation}
\label{eqn:isom}
\V \cong [\OO_x \otimes E_x]^R
\end{equation}

The $R$-finite vectors in the de~Rham complex of $C'$ with coefficients in
$\OO_x$ form a complex $E^\dot_\fin(C',\OO_x)$. In concrete terms, this is
$$
E^\dot_\fin(C',\OO_x) =
\bigoplus_\lambda E^\dot(C',\V_\lambda)\otimes V_{\lambda,x}^\ast.
$$
It is a (graded commutative) differential graded algebra. Similarly, one
defines the ind-MHC
$$
\bK(\OO_x) = \bigoplus_\lambda \bK(\V_\lambda)\otimes V_{\lambda,x}^\ast
$$
whose complex part is
$$
K^\dot(C,D;\OO_x) =
\bigoplus_\lambda K^\dot(C,D;\V_\lambda)\otimes V_{\lambda,x}^\ast.
$$
The relevance of these complexes is that the iterated integrals of their
elements are elements of the coordinate ring of $\cG_x$. (Cf.\
\cite{hain:malcev}.)

\begin{lemma}
\label{lem:isom}
If $\bE$ is a PVHS over $C'$, then the isomorphism (\ref{eqn:isom}) induces an
isomorphism of bifiltered complexes:
$$
K^\dot(C,D;\bE) \cong (K^\dot(C,D;\OO_x)\otimes \bE)^R. \qed
$$
\end{lemma}

The relevance of the preceding discussion is that iterated integrals of
elements of
$$
E^\dot(C',\V_\lambda)\otimes V_{\lambda,x}^\ast
$$
that occur in Section~\ref{sec:concrete} are iterated integrals of elements of
$E^\fin(C',\OO_x)$ considered in \cite{hain:malcev}. This implies that the
iterated Shimura integrals considered by Manin \cite{manin1,manin2} are examples
of the iterated integrals constructed in \cite{hain:malcev}.

\subsection{The MHS on the relative completion of $\pi_1(C',x)$}
\label{sec:mhs}

Denote the fiber of $\bu$ over $x\in C'$ by $\u_x$. Denote the corresponding
prounipotent group by $\U_x$. The Hodge and weight bundles of $\bu$ restrict to
Hodge and weight filtrations on $\u_x$. Set $\G = \pi_1(C',x)$ and write $C'$ as
the quotient $\G\bs X$ of a simply connected Riemann surface by $\G$. Implicit
here is that we have chosen a point $\xtilde\in X$ that lies over $x\in C'$. 

Trivialize the pullback of each local system $\bu_n$ to $X$ so that the flat
sections are constant.\footnote{It is also natural to trivialize $\bu_1$ (and
hence all $\bu_n$) so that the Hodge bundles are trivialized. In this case, we
are in the setup of Section~\ref{sec:variant}. The monodromy representation will
be the same as it is with this ``constant trivialization". Trivializing the
Hodge bundles is better when computing the MHS on completed path torsors. For
example, in Section~\ref{sec:H}, we can trivialize $\cH:=\H\otimes\cO_\h$ over
the upper half plane $\h$ using the flat sections $\a,\b$, as we do here, or we
can use the sections $\a,\bw$ that are adapted to the Hodge filtration. The
second trivialization is better suited to studying asymptotic properties of
VMHS.} This determines a trivialization
$$
\xymatrix{
X\times \u_x \ar[r]\ar[d] & \bu \ar[d] \cr
X \ar[r] & C'
}
$$
of the pullback of $\bu$ to $X$ as the product of the pullbacks of the $\bu_n$.
The pullback of the connection $\nabla = \nabla_0 + \Omega$ on $\bu$ is
$\G$-invariant with respect to the diagonal $\G$-action on $X\times\u_x$ and is
of the form $d + \Omegatilde$, where
$$
\Omegatilde \in (E^1(X)\comptensor \u_x)^\G.
$$
Proposition~\ref{prop:char} and the fact that $\Omega_1$  represents the product
of the identity maps $H^1(C',\V_\lambda) \to H^1(C',\V_\lambda)$  for all
$\lambda \in \Rdual$ imply that the transport of $\Omegatilde$ induces an
isomorphism
$$
\Theta_x : \cG_x \to R_x \ltimes \U_x.
$$
The MHS on $\cO(\cG_x)$ is constructed by pulling back the natural Hodge
and weight filtrations on $\cO(R_x)\otimes\cO(\U_x)$, which we now recall.

To describe $\cO(\U_x)$ we recall some basic facts about prounipotent groups.
Suppose that $\cN$ is a prounipotent group over $F$ with Lie algebra $\n$. The
enveloping algebra $U\n$ of $\n$ is a Hopf algebra. The exponential mapping
$\exp : \n \to \cN$ is a bijection, so we can identify $\cN$ with the subspace
$\n$ of $U\n$. The Poincar\'e-Birkhoff-Witt Theorem implies that the restriction
mapping induces a Hopf algebra isomorphism
$$
\cO(\cN) \cong \Hom^\cts(U\n,F) \cong
\Sym^\cts(\n) := \varinjlim_\alpha \Sym \n_\alpha.
$$
In particular, $\cO(\U_x) \cong \Hom^\cts(U\u_x,\C)$. The construction of
the Hodge and weight filtrations of $\u_x$ and (\ref{eqn:coaction}) imply
that the coaction
\begin{equation}
\label{eqn:coact_u}
U\u_x \to \cO(R_x)\otimes U\u_x
\end{equation}
that defines the semi direct product $R_x\ltimes \U_x$ preserves the Hodge and
weight filtrations.\footnote{In fact, this is a morphism of MHS --- and thus
strict with respect with respect to the Hodge and weight filtrations --- if we
give $\u_x$ MHS via the identification $\u_x \cong \prod_n (\Gr^\LCS_n \u)_x$.
The canonical MHS on $\u_x$ has the same underlying complex vector space, the
same Hodge and weight filtrations, but its $\Q$-structure is deformed using the
deformed connection $\nabla_0+\Omega$.}

The Hodge and weight filtrations of $\u_x$ induce Hodge and weight filtrations
on $\Hom^\cts(U\u_x,\C)$, and thus on $\cO(\U_x)$. Define Hodge and weight
filtrations on the coordinate ring of $R_x\ltimes \U_x$ via the canonical
isomorphism
$$
\cO(R_x\ltimes \U_x) \cong \cO(R_x)\otimes \cO(\U_x).
$$
These pullback to Hodge and weight filtrations on $\cO(\cG_x)$ along the
isomorphism $\Theta_x$. Equation (\ref{eqn:coact_u}) implies that these are
compatible with the product and coproduct of $\cO(R_x\times \U_x)$.

A filtration $W_\dot$ of a Lie algebra $\u$ is said to be {\em compatible with
the bracket} if $[W_i\u,W_j\u]\subseteq W_{i+j}\u$ for all $i,j$. Similarly, a
filtration $W_\dot$ of a Hopf algebra $A$ is {\em compatible with its
multiplication $\mu$ and comultiplication $\Delta$} if $\mu(W_iA\otimes W_jA)
\subseteq W_{i+j}A$ and $\Delta(W_m A) \subset \sum_{i+j} W_iA\otimes W_j A$.
The Hodge and weight filtrations on $\cO(R_x\ltimes \U_x)$ defined above are
compatible with its product and coproduct.

Denote the maximal ideal of $\cO(\cG_x)$ of functions that vanish at the
identity by $\m_e$. The Lie algebra $\g_x$ of $\cG_x$ is isomorphic to
$\Hom(\m_e/\m_e^2,F)$ and its bracket is induced by the Lie cobracket of
$\cO(\cG_x)$. The Hodge and weight filtrations of $\cO(\cG_x)$ thus induce Hodge
and weight filtrations on $\g_x$ that are compatible with its bracket.

\begin{theorem}
\label{thm:mhs}
These Hodge and weight filtrations define a MHS on $\cO(\cG_x)$ for which the
multiplication and comultiplication are morphisms. This MHS agrees with the one
constructed in \cite{hain:malcev}.
\end{theorem}

\begin{proof}
The natural isomorphism $\Theta_x^\ast : \cO(R_x\ltimes \U_x) \to \cO(\cG_x)$
respects $\Q$-structures (essentially by definition). To prove the result it
suffices to show that it takes the Hodge and weight filtrations on
$\cO(R_x\ltimes \U_x)$ constructed above onto the Hodge and weight filtrations
of the canonical MHS on $\cO(\cG_x)$ constructed in \cite{hain:malcev}. This
will imply that the weight filtration defined on $\cO(R_x\ltimes \U_x)$ above is
defined over $\Q$ and that $\cO(R_x\otimes\U_x)$ has a MHS and that this MHS is
isomorphic to the canonical MHS on $\cO(\cG_x)$ via $\Theta_x^\ast$.

The first point is that Saito's MHC for computing the MHS on $H^\dot(C',\V)$,
which is used in \cite{hain:malcev}, is a generalization of Zucker's MHC used
here and agrees with it in the curve case. The iterated integrals of elements of
$K^\dot(C,D;\V)\otimes V_x^\ast$ in this paper are a special case of the
iterated integrals defined in \cite[\S5]{hain:malcev} by Lemma~\ref{lem:isom}.

The next point is that, by Equation~(\ref{eqn:filts}) and Lemma~\ref{lem:isom},
$$
\Omega \in F^0 W_{-1} K^1(C,D;\bu)
\cong F^0 W_{-1}\big(K^1(C,D;\OO_x)\comptensor \u_x\big).
$$
This implies, and this is the key point, that --- with the Hodge and weight
filtrations on the bar construction defined in \cite[\S13]{hain:malcev} and the
Hodge and weight filtrations on $\u_x$ defined above --- the $U\u_x$-valued {\em
formal transport}
$$
T := 1 + [\Omegatilde] +[\Omegatilde|\Omegatilde] +
[\Omegatilde|\Omegatilde|\Omegatilde] + \cdots 
\in B\big(\C,K^\dot_\fin(C,D;\OO_x),\C\big)\comptensor U\u_x
$$
of $\Omegatilde$, which takes values in the completed enveloping algebra of
$\u_x$, satisfies
$$
T \in
F^0W_0H^0\Big(B\big(\C,K^\dot_\fin(C,D;\OO_x),\C\big)\comptensor U\u_x\Big).
$$
This implies that the induced Hopf algebra homomorphism
$$
f : \cO(\U_x) \cong \Hom^\cts(U\u_x,\C) \to
H^0\Big(B\big(\C,K^\dot_\fin(C,D;\OO_x),\cO(R_x)\big)\Big) \cong \cO(\cG_x)
$$
that corresponds to the function
$$
\xymatrix{
\cG_x \ar[r]^(.4){\Theta_x} & R_x\ltimes \U_x \ar[r]^(.55){\text{projn}} & \U_x
}
$$
preserves the Hodge and weight filtrations.

The constructions in \cite{hain:malcev} imply that the homomorphism $\pi^\ast
:\cO(R_x) \to \cO(\cG_x)$ induced by the projection $\pi:\cG_x \to R_x$ is a
morphism of MHS. Since the Hodge and weight filtrations of $\cO(\cG_x)$ are
compatible with multiplication and since both $f$ and $\pi^\ast$ preserve
the Hodge and weight filtrations filtrations, the homomorphism
$$
\xymatrix{
\cO(R_x)\otimes\cO(\U_x) \ar[r]^(.6){\pi^\ast\otimes f} & \cO(\cG_x)
} 
$$
does too. This homomorphism is $\Theta_x^\ast$ when $\cO(R_x)\otimes \cO(\U_x)$
is identified with $\cO(R_x\ltimes\U_x)$.

It remains to prove that this isomorphism is an isomorphism of bifiltered
vector spaces. Since $\pi^\ast$ is a morphism of MHS, it suffices to show
that the isomorphism $j^\ast\circ f$
$$
\xymatrix{
\cO(\U_x) \ar[r]^f & \cO(\cG_x) \ar[r]^{j^\ast} & \cO(\cN_x),
}
$$
where $j:\cN_x\to \cG_x$ is the inclusion of the prounipotent radical of
$\cG_x$, is a bifiltered isomorphism. To prove this, it suffices to show that
$\n_x \to \u_x$ is a bifiltered isomorphism, where $\n_x$ denotes the Lie
algebra of $\cN_x$. But this follows from Lemma~\ref{lem:gr_lcs}, which implies
that $H_1(\u_x)$ has a MHS and that the induced homomorphism
$$
\Gr^\LCS_\dot\n_x \to \Gr^\LCS_\dot\u_x
$$
is an isomorphism of graded MHS.
\end{proof}

The result also gives an explicit description of the MHS on $\u_x$.

\begin{corollary}
The Hodge and weight filtrations of the natural MHS on $\u_x$ are those induced
on it from $\bu$; its $\Q$-structure is the one described in
Proposition~\ref{prop:Qstr}. \qed
\end{corollary}

To complete the story, we show that the $\bu$ is a pro-admissible variation of
MHS.

\begin{theorem}
\label{thm:admissible}
With the Hodge and weight filtrations and $\Q$-structure defined above,
$(\bu,\nabla)$ is a pro-object of the category of admissible variation of MHS
over $C'$. Its lower central series is a filtration of $\bu$ by pro-admissible
variations of MHS. The natural isomorphism
$$
(\Gr^n_\LCS\bu,\nabla_0) \cong (\Gr^n_\LCS\bu,\nabla)
$$
is an isomorphism of admissible variations of MHS. In particular, there are
natural MHS isomorphisms
$$
H_1(\u_x) \cong
\prod_{\lambda \in \Rdual} H^1(C',\V_\lambda)^\ast\otimes V_{\lambda,x}
$$
for all base points $x$ of $C'$.\footnote{This statement holds, even when
$x$ is a tangential base point.} Finally, for all PVHS $\V$ over $C'$ whose
monodromy representation factors through $\rho$, the natural homomorphism
$$
H^\dot(\cG_x,V_x) = [H^\dot(\u_x)\otimes V_x]^R \to H^\dot(C',\V)
$$
is an isomorphism of MHS.
\end{theorem}

With a little more work, one can show that the local system with fiber
$\cO(\cG_x)$ over $x\in C'$ is an admissible VMHS.

\begin{proof}
Proposition~\ref{prop:canon_extn} and Lemma~\ref{lem:hodge_bdles} imply that the
Hodge bundles are holomorphic sub-bundles of $\bu$ and extend to holomorphic
sub-bundles of $\bubar$.  Lemma~\ref{lem:wt_bdles} implies that the weight
bundle's are flat and extend to $\bubar$. Theorem~\ref{thm:mhs} implies that
$\bu$ has a natural $\Q$-form and that, with respect to these structures, each
fiber of $\bu$ has a natural MHS.

To complete the proof, we need to show that at each $P\in D$ there is a relative
weight filtration of the fiber $\u_P$ of $\bubar$ over $P$. First an easily
verified fact. Suppose that $\V$ is a PVHS over $C'$  of weight $m$. Let
$M_\dot$ be the monodromy weight filtration of its fiber $V_P$ over $P\in D$
shifted so that it is centered at $m$. Then if $A$ is a constant MHS, then the
filtration
$$
M_r (A\otimes V_P) := \sum_{i+j=r} W_i A\otimes M_j V_P
$$
defines a relative weight filtration of the fiber over $P\in D$ of the
admissible variation of MHS $A\otimes \V$. From this it follows that the fiber
over $P\in D$ of each
$$
H^1(C',\V_\lambda)^\ast\otimes \V_\lambda
$$
has a relative weight filtration. Adding these implies that the fiber over $P\in
D$ of each of the pro-variations $\bu_n$ has a relative weight filtration. Write
the residue at $P\in D$ of $\nabla$ as the sum
$$
N_P = N_0 + N_\u
$$
of the residues of $\nabla_0$ and $\Omega$. The discussion above implies that
the product of the weight filtrations on the fibers over $P$ of the $\bu_n$
defines a relative weight filtration for $N_0$ on the fiber $\u_P$ of $\bubar$
over $P$. We have to show that this is also a relative weight filtration
for $N$. To prove this, it suffices to show that $N_\u \in W_{-2}\u_P$. (See the
definition of the relative weight filtration in \cite{steenbrink-zucker}.)

Let $t$ be a local holomorphic coordinate on $C$ centered at $P$. Then, near
$P$, we can write
$$
\Omega = N_\u\otimes\frac{dt}{t}
+ \text{ a smooth $1$-form with values in }\bubar 
$$
Since $\Omega^1\in W_{-1}K^1(C,D;\bu)$, and since $dt/t$ has weight $1$, we see
that $N_\u \in W_{-2}\u_P$, as required.

The last statement is a direct consequence of (\ref{sec:modular}), the de~Rham
construction of the homomorphism in Section~\ref{sec:charn}, and the fact that
$\Omega \in F^0W_{-1}K^1(C,D;\bu)$, which implies that $\theta_\Omega :
H^0\big(C',\Hom(\Lambda^\dot\bu,\V)\big) \to K^\dot(C,D,\V)$ preserves the Hodge
filtration and satisfies
$$
\theta_\Omega\Big(W_m H^0\big(C',\Hom(\Lambda^j\bu,\V)\big)\Big)
\subseteq W_{m-j}K^j(C,D,\V)
$$
in degree $j$.\footnote{This is equivalent to the statement that $\theta_\Omega$
is a filtration preserving dga homomorphism to $\Dec_W K^j(C,D,\V)$, where
$\Dec_W$ is Deligne's shifting functor (with respect to the weight filtration),
\cite{deligne:hodge2}.} The last ingredient is the Theorem of the Fixed Part,
which implies that for each $x\in C'$, the restriction mapping
$$
H^0\big(C',\Hom(\Lambda^\dot\bu,\V)\big) \to \cC^\dot(\u_x,V_x)^R
$$
is an isomorphism of MHS for all $x\in C'$.
\end{proof}

\subsection{A MHS on completed path torsors}

Here we give a brief description of how to extend the methods of the previous
section to construct the canonical MHS on the coordinate ring $\cO(\cG_{x,y})$
of the relative completion of the path torsor $\Pi(C';x,y)$ of $C'$ with respect
to a polarized VHS $\V$. Here $C'$ may be an orbifold of the form $\G\bs X$.

As in Section~\ref{sec:mhs}, write $C'$ as the quotient of a simply connected
Riemann surface $X$ by a discrete group $\G$ isomorphic to $\pi_1(C',x)$.
Trivialize the pullback of each $\bu_n$ to $X$ using the flat sections and use
this to trivialize the pullback of $\bu$ to $X$. Here, unlike in the previous
section, it is useful to denote the common fiber of the trivialization by $\u$
so that the pullback of $\bu$ to $X$ is $X\times \u$. The fiber over $t\in X$
will be regarded as $\u$ with a Hodge filtration $F^\dot_t$, which depends on
$t$, and a weight filtration $W_\dot$, which does not. We can therefore identify
$\Hom^\cts(\u_s,\u_t)$ with $\End^\cts \u$. The Hodge and weight filtrations of
$\u_s$ and $\u_t$ induce Hodge and weight filtrations on $\Hom^\cts(\u_s,\u_t)$.

For $s,t \in X$, set $\U_{s,t}$ be the subscheme of $\Isom^\cts(\u_s,\u_t)$ that
corresponds to the subgroup $\U := \exp\u$ of $\Aut \u$. It is a proalgebraic
variety. Denote the subspace of $\Hom^\cts(\u_s,\u_t)$ that corresponds to the
image of
$$
U\u\in \End(\u), \quad u \mapsto \{v \mapsto uv\}
$$
by $U\u_{s,t}$. The coordinate ring of $\U_{s,t}$ is $\cO(\U_{s,t}) \cong
\Hom^\cts(U\u_{s,t},\C)$. It has natural Hodge and weight filtrations induced
from those on $\Hom^\cts(\u_s,\u_t)$.

The pullback connection is $d + \Omegatilde$, where $\Omegatilde \in
E^1(X)\comptensor \u$. Since the structure group of this connection is $\U$, the
parallel transport map $T_{s,t}:\u_s \to \u_t$ lies in $\U_{s,t}$.

As in Section~\ref{sec:path_torsor}, we choose a fundamental domain $D$ of the
action  of $\G$ on $X$. Denote the unique lift of $z\in C'$ to $D$ by $\ztilde$.
For each homotopy class of paths in $C'$ from $x$ to $y$ one has $\gamma \in \G$
and a homotopy class $c_\gamma$ of paths from $\xtilde$ to $\gamma \ytilde$.
Parallel transport defines a function
$$
\Theta_{x,y} : \Pi(C';x,y) \to \U_{x,y}\times R_{x,y} \cong \cG_{x,y}
$$
by
$(\gamma,c_\gamma) \mapsto \big(T(c_\gamma)^{-1},\rho_{x,y}(\gamma)\big)$.
This is the relative completion of $\Pi(C';x,y)$.

The isomorphism
$$
\cO(\cG_{x,y}) \cong \cO(\U_{x,y})\otimes \cO(R_{x,y})
$$
induces Hodge and weight filtrations on the coordinate ring of $\cG_{x,y}$. It
also has a natural $\Q$-structure as relative completion is defined over $\Q$
and behaves well under base change from $\Q$ to $\C$.

The following theorem generalizes Theorems~\ref{thm:mhs} and
\ref{thm:admissible}. It is proved using a similar arguments. A more general
version of all but the last statement is proved in \cite[\S12--13]{hain:malcev}.

\begin{theorem}
\label{thm:hodge_torsor}
These Hodge and weight filtrations define a MHS on $\cO(\cG_{x,y})$, making it a
ring in the category of ind-mixed Hodge structures. This MHS coincides with the
one constructed in \cite{hain:malcev}. It has the property that if $x,y,z\in X$,
then the maps $\cG_{x,y} \to \cG_{y,x}$ and $\cG_{x,y}\times \cG_{y,z} \to
\cG_{x,z}$ induced by inverse and path multiplication, respectively, induce
morphisms of MHS
$$
\cO(\cG_{y,x}) \to \cO(\cG_{x,y}) \text{ and }
\cO(\cG_{x,z}) \to \cO(\cG_{x,y})\otimes \cO(\cG_{y,z}).
$$
In addition the local system $\bcG_{x,\underline{\phantom{y}}}$ over $C'$ whose
fiber over $x,y$ is $\cO(\cG_{x,y})$ is an Ind object of the category of
admissible variations of MHS over $C'$.
\end{theorem}

\subsection{Tangential base points and limit MHSs}

Theorem~\ref{thm:admissible} implies that for each choice of a non-zero tangent
vector $\vv$ of $C$ at $P\in D$, there is a limit MHS on the fiber $\u_P$ of
$\bubar$ at $P$. We will denote this MHS by $\u_\vv$. It is natural to think of
it as a MHS on the unipotent radical of the relative completion of the
fundamental group $\pi_1(C',\vv)$ of $C'$ with (tangential) base point $\vv$.

More generally, we consider path torsors between tangential base points,. We
first recall the definition from \cite{deligne:p1} of the torsor of paths
$\Pi(C';\vv,\vw)$. Here $P,Q \in D$ and $\vv \in T_P C$, $\vw \in T_Q C$ are
non-zero tangent vectors. Elements of $\Pi(C;\vv,\vw)$ are homotopy classes
of piecewise smooth paths $\gamma : I \to C$ satisfying
\begin{enumerate}

\item $\gamma(0)=P$, $\gamma(1)=Q$,

\item $\gamma(t) \in C'$, when $0<t<1$,

\item $\gamma'(0)=\vv$ and $\gamma'(1)=-\vw$.

\end{enumerate}
This definition can be modified to define $\Pi(C;\vv,x)$ and $\Pi(C;x,\vw)$ when
$x\in C'$. One defines $\pi_1(C',\vv) = \Pi(C',\vv,\vv)$.
\begin{figure}[!ht]
\epsfig{file=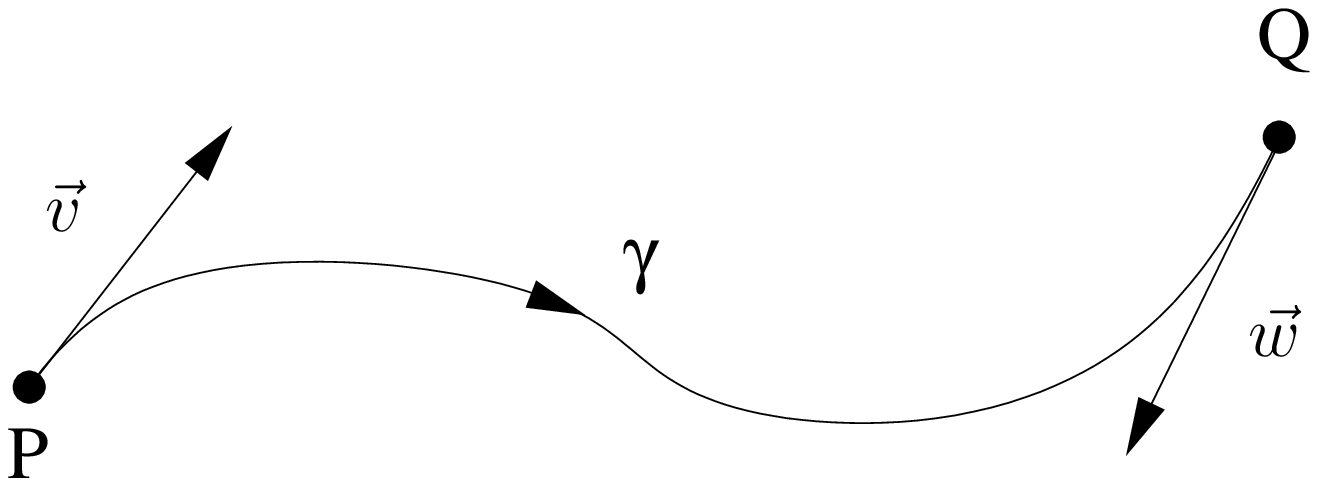, width=2in}
\caption{A path from $\vv$ to $\vw$ in $C'$}
\label{fig:path}
\end{figure}
For any 3 base points (tangential or regular) $b$, $b'$, $b''$, there are
composition maps
$$
\Pi(C';b,b') \times \Pi(C';b',b'') \to \Pi(C';b,b'').
$$

\subsubsection{The MHS on $\u_\vv$}
Suppose that $P\in D$ and that $\vv$ is a non-zero tangent vector of $C$ at $P$.
The complex vector space underlying the limit MHS on $\u_\vv$ is the fiber
$\u_P$ of $\bubar$ over $P$. Its Hodge and weight filtrations $F^\dot$ and
$W_\dot$ are the restrictions to $\u_P$ of the Hodge and weight filtrations of
$\bubar$ that were constructed above. There is also the relative weight
filtration $M_\dot$ of $\u_P$, which was constructed in the proof of
Theorem~\ref{thm:admissible}. These data depend only on $P$ and not on $\vv$.

To construct the $\Q$-structure, choose a local holomorphic coordinate $t:\D \to
\C$ on $C$ centered at $P$ with the property that $\vv = \partial/\partial t$.
Then there is a unique trivialization $\bubar|_\D \cong \D\times \u_P$ of
$\bubar$ over $\D$ such that (1) the trivialization is the identity on the fiber
$\u_P$ over $P$, (2) $\D\cap D = \{P\}$, and (3) $\nabla = d + N_P dt/t$, where
$N_P$ is the residue of $\nabla$ at $P$. This trivialization allows us to
identify fibers of $\bu$ over points near $P$ with $\u_P$. Note that this
identification depends on the choice of the local coordinate $t$.

For $t\in \D$, the $\Q$-structure on $\u_P$ corresponding to $t\vv$ is simply
the $\Q$ structure on $\u_P$ obtained by identifying $\u_P$ with the fiber
$\u_t$ of $\bu$ over $t\in \D$ and taking the $\Q$-structure to be that of
$\u_t$. The $\Q$-structure corresponding to $\vv$ is
$$
\u_{\vv,\Q} := t^{N_P} \u_{t\vv,\Q}.
$$
That is, it is the unique $\Q$-structure on $\u_P$ such that
$$
u_{t\vv,\Q} := t^{-N_P} \u_{\vv,\Q}.
$$
for all $t\in \C^\ast$. Although the trivialization above depends on the choice
of the local coordinate $t$, the $\Q$-structure on $\u_P$ corresponding to $\vv$
depends only on $\vv$ the tangent vector $\vv$.

With a little more effort, one can construct the limit MHS on $\cO(\cG_\vv)$.
Full details will appear in \cite{hain-pearlstein}. As in the case of
cohomology, where periods of limit MHSs can be computed by regularizing
integrals, the periods of the limit MHS on $\u_\vv$ are regularized iterated
integrals.

\subsubsection{Limit MHS on completed path torsors} Similarly, one can construct
the limit MHS on $\cO(\cG_{x,\vw})$ and $\cO(\cG_{\vv,\vw})$, etc. Full details
will appear in \cite{hain-pearlstein}.

\section{Completed Path Torsors and Admissible Variations of MHS}
\label{sec:avmhs}

Here we state two results that relate the Hodge theory of relative completion of
fundamental groups and path torsors to admissible variations of MHS. These are
special cases of results in \cite{hain-pearlstein}.

Let $C$, $D$ and $C'=C-D$ be as above. Let $\V$ be a PVHS over $C'$ and $R_x$
the Zariski closure of the monodromy representation $\pi_1(C',x) \to \Aut V_x$.
Let $\rho_x$ be the homomorphism $\pi_1(X,x)\to R_x$. Let $\cG_x$ be the
completion of $\pi_1(C',x)$ with respect to $\rho_x$. For base points (regular
or tangential) $b,b'$ of $C'$, let $\cG_{x,y}$ denote the relative completion of
$\Pi(C';b,b')$.

Denote by $\MHS(C',\V)$ the category of admissible VMHS $\A$ over $C'$ with
the property that the monodromy representation of $\Gr^W_\dot \A$ factors
through $\rho_x$. This condition implies that the monodromy representation
$$
\pi_1(C',x) \to \Aut(A_x)
$$
factors through $\pi_1(C',x) \to \cG_x$.

\begin{theorem}
For all variations $\A$ in $\MHS(C',\V)$ and all base points $b,b'$ (possibly
tangential) of $C'$ the parallel transport mapping
$$
A_b \to A_{b'}\otimes \cO(\cG(\cG_{b,b'}))
$$
induced by $A_b \times \Pi(C';b,b') \to A_{b'}$ is a morphism of MHS. When $b$
or $b'$ is tangential, then the monodromy preserves both the weight filtration
$W_\dot$ and the relative weight filtration $M_\dot$.
\end{theorem}

\begin{definition}
\label{def:hrep}
Suppose that $\cG$ is a proalgebraic group whose coordinate ring $\cO(\cG)$
is a Hopf algebra in the category of ind mixed Hodge structures.
A {\em Hodge representation} of $\cG$ on a MHS $A$ is a
homomorphism for which the action
$$
A \to A\otimes \cO(\cG)
$$
is a morphism of MHS. The category of Hodge representations of $\cG$ will
be denoted by $\HRep(\cG)$.
\end{definition}

The previous result implies that taking the fiber at $b$ defines a functor from
$\MHS(C',\V)$ to $\HRep(\cG_b)$. The next theorem follows from
Theorem~\ref{thm:mhs} by a tannakian argument. Full details will be given in
\cite{hain-pearlstein}.

\begin{theorem}
\label{thm:equivalence}
For all base points $b$ of $C'$, the ``fiber at $b$'' functor
$\MHS(C',\V) \to \HRep(\cG_b)$ is an equivalence of categories.
\end{theorem}

\part{Completed Path Torsors of Modular Curves}
\label{part:mod_gps}

In this part, we apply the general constructions of the first part to explore
the relative completions of modular groups, mainly in the case of the full
modular group $\SL_2(\Z)$. Throughout we use the following notation.

Suppose that $\G$ is a finite index subgroup of $\SL_2(\Z)$. The associated
curve $X_\G$ is the quotient $\G\bs \h$ of the upper half plane by $\G$. It is a
smooth affine curve when $\G$ is torsion free. When $\G$ has torsion, it will be
regarded as an orbifold as follows: Choose a finite index, torsion free normal
subgroup $\G'$ of $\G$. Set $G=\G/\G'$. Then $X_\G$ is the orbifold quotient of
$X_{\G'}$ by $G$. To work on the orbifold $X_\G$, one can work either
$G$-equivariantly on $X_{\G'}$ or $\G$-equivariantly on $\h$.

The (orbi) curve $X_\G$ can be completed to a smooth (orbi) curve by adding the
finite set $D := \G\bs \P^1(\Q)$ of ``cusps''. Denote the compactified curve by
$\Xbar_\G$. When $\G=\SL_2(\Z)$, the modular curve $X_\G$ is the moduli space
$\M_{1,1}$ of elliptic curves and $\Xbar_\G$ is $\Mbar_{1,1}$, its the
Deligne-Mumford compactification, which is obtained by adding a single cusp.

If $P\in D$ is in the orbit of $\infty \in \P^1(\Q)$, then
$$
\G \cap \begin{pmatrix} 1 & \Z \cr 0 & 1 \end{pmatrix} =
\begin{pmatrix} 1 & n\Z \cr 0 & 1 \end{pmatrix}.
$$
for some $n\ge 1$. A punctured neighbourhood of $P$ in $X_\G$ is  the quotient
of $\{\tau \in \C : \Im(\tau) > 1\}$ by this group. This is a punctured disk
with coordinate $e^{2\pi in\tau}$. In particular, when $\G=\SL_2(\Z)$, the
coordinate about the cusp is $q := e^{2\pi i\tau}$.

\section{The Variation of Hodge Structure $\H$}
\label{sec:moduli}

\subsection{The Local System $\H$}
\label{sec:H}

The universal elliptic curve $f : \E \to X_\G$ over $X_\G$ is the quotient of
$\C\times\h$ by $\G\ltimes \Z^2$, which acts via
$$
(m,n) : (z,\tau)
\mapsto \bigg(z + \begin{pmatrix} m & n \end{pmatrix}
\begin{pmatrix} \tau \cr 1 \end{pmatrix}, \tau \bigg)
$$
and $\gamma : (z,\tau) \mapsto \big((c\tau+d)^{-1}z,\gamma \tau\big)$, where 
$$
\gamma = \begin{pmatrix} a & b \cr c & d \end{pmatrix} \in \G.
$$
When $\G$ is not torsion free, it should be regarded as an orbifold family of
elliptic curves.

The local system $\H$ over $X_\G$ is $R^1 f_\ast \Q$. When $\G$ is torsion 
free, this is the local system over $X_\G$ with fiber $H^1(f^{-1}(x),\Q)$ over
$x$. It is a polarized variation of HS of weight 1. Since Poincar\'e duality
induces an isomorphism
$$
H_1(E) \cong H^1(E)(1)
$$
for all elliptic curves $E$, the polarized variation $\H(1)$ of weight $-1$ is
the local system over $X_\G$ whose fiber over $x$ is $H_1(f^{-1}(x),\Q)$. The
polarization is the intersection pairing. Denote the corresponding holomorphic
vector bundle $\H\otimes\cO_{X_\G}$ by $\cH$. Its Hodge filtration satisfies
$$
\cH = F^0 \cH \supset F^1 \cH \supset F^2 \cH = 0.
$$
The only interesting part is $F^1\cH$.

In general we will work with the pullback $\H_\h$ of $\H$ to $\h$. Its fiber
over $\tau\in \h$ is $H^1(E_\tau,\Z)$, where
$$
E_\tau := \C/\Lambda_\tau \text{ and } \Lambda_\tau := \Z \oplus \Z\tau.
$$
Denote the basis of $H_1(E_\tau,\Z)$ corresponding to the elements $1$ and
$\tau$ of $\Lambda_\tau$ by $\a,\b$. Denote the dual basis of $H^1(E_\tau)$ by
$\adual,\bdual$. These trivialize $\H_\h$.

If we identify $H^1(E_\tau)$ with $H_1(E_\tau)$ via Poincar\'e duality, then
$$
\adual = -\b \text{ and } \bdual = \a.
$$
We regard these as sections of $\H_\h$. 

For each $\tau \in \h$, let $\w_\tau \in H^0(E_\tau,\Omega^1)$ be the unique
holomorphic differential that takes the value 1 on $\a$. It spans $F^1
H^1(E_\tau)$. In terms of the framing, it is given by
\begin{equation}
\label{eqn:wtau}
\w_\tau = \adual + \tau\bdual  = \tau \a - \b =
\begin{pmatrix} \a & -\b \end{pmatrix}\begin{pmatrix}\tau \cr 1 \end{pmatrix}.
\end{equation}
The map $\w : \tau \mapsto \w_\tau$ is thus a holomorphic section of $\cH_\h :=
\H_\h\otimes \cO_\h$ whose image spans $F^1\cH_\h$.

Since $\h$ is contractible, $H_1(\E,\Z) \cong \Z\a\oplus \Z\b$. The left action
of $\SL_2(\Z)$ on $\E$ induces a {\em left} action on $H_1(\E)$, and therefore a
{\em right} action on frames. The following result gives a formula for this
action on frames.

\begin{lemma}
\label{lem:invariance}
For all $\gamma\in \G$,
\begin{enumerate}

\item
$\gamma : \begin{pmatrix} \a & -\b \end{pmatrix}
\mapsto \begin{pmatrix} \a & -\b \end{pmatrix}\gamma$,

\item
the section $\w$ of $\cH_\h$ satisfies
$(1\otimes \gamma) \w = (c\tau+d)(\gamma^\ast\otimes 1)\w$.
\end{enumerate}
\end{lemma}

\begin{proof}
Let $\gamma = \begin{pmatrix} a & b \cr c & d \end{pmatrix}$. Regard $\a$ and
$\b$ as sections of $\H_\h$. Denote the values of $\a$ and $\b$ at $\tau$
by $\a,\b$ and at $\gamma(\tau)$ by $\a',\b'$.
\begin{figure}[!ht]
$$
\def\latticebody{\drop{\cdot}}
\xy
*\xybox{0;<4pc,0pc>:<-.7pc,3pc>:: %% this line determines the basis
,0
,{\xylattice{-1}3{0}3}
,(0,0)="id"*{\bullet}*+!UL{0}
;(1,0)*{\bullet}*+!UL{1} **@{-}
,(.5,0)*+!DL{\a}
%
%,(1,0)*{\bullet}*+!UL{1}
,"id"
,(0,1)*{\bullet}*+!CR{\tau} **@{-}
,(0,.5)*+!CR{\b}
,"id"
;(1,2)*{\bullet} **@{-}
,(1.2,2.1)*+!DR{a\tau + b}
,(.7,1.4)*+!DR{\b'=a\b+b\a}
,"id"
;(2,3)*{\bullet}*+!CL{c\tau+d} **@{-}
,(1.2,1.7)*+!CL{\a'=c\b+d\a}
}="L"
,{"L"+L \ar "L"+R*+!L{\R}}
%,{"L"+D \ar "L"+U*+!D{i\R}}
\endxy
$$
\caption{The $\SL_2(\Z)$ action on frames}
\label{fig:bases}
\end{figure}
Then (cf.\ Figure~\ref{fig:bases})
$$
\begin{pmatrix} \a' & -\b' \end{pmatrix}
= \begin{pmatrix} \a & -\b \end{pmatrix}
\begin{pmatrix} d & -b \cr -c & a \end{pmatrix}
= \begin{pmatrix} \a & -\b \end{pmatrix}\gamma^{-1}.
$$
Thus
$
\gamma_\ast\begin{pmatrix} \a(\tau) & -\b(\tau) \end{pmatrix}
= \begin{pmatrix} \a(\gamma\tau) & -\b(\gamma\tau) \end{pmatrix}\gamma
$,
from which the first assertion follows.

The second assertion now follows:
$$
 (1\otimes\gamma)\w
 =
\begin{pmatrix} \a & -\b \end{pmatrix} \gamma
\begin{pmatrix}\tau \cr 1 \end{pmatrix}
= 
(c\tau + d)
\begin{pmatrix} \a & -\b \end{pmatrix}
\begin{pmatrix}\gamma \tau \cr 1 \end{pmatrix}
=
(c\tau + d)(\gamma^\ast\otimes 1)\w.
$$
\end{proof}

Define $\bw$ to be the section
\begin{equation}
\label{eqn:def_w}
\bw : \tau \mapsto 2\pi i \w_\tau = 2\pi i(\tau \a -\b).
\end{equation}
of $\cH_\h$. Since
$\langle \bw,\a \rangle = 2\pi i \langle \tau\a - \b, \a \rangle = 2\pi i$.
$\bw,\a$ is a framing of $\cH_\h$. The sections $\a$ and $\bw$ trivialize
$\cH_\h$ over $\h$. As we shall see below, this trivialization is better suited
to computing limit MHSs. The following computation is immediate. The proof of an
equivalent formula can be found in \cite[Ex.~3.4]{hain:kzb}.

\begin{corollary}
\label{cor:automorphy}
The factor of automorphy associated to the trivialization
$$
\cH_\h \cong (\C\a \oplus \C\bw)\times \h
$$
of the pullback of $\cH$ to $\h$ is
$$
M_\gamma(\tau) : \begin{pmatrix} \a & \bw \end{pmatrix}
= \begin{pmatrix} \a & \bw \end{pmatrix}
\begin{pmatrix}
(c\tau+d)^{-1} & 0 \cr c/2\pi i & c\tau + d
\end{pmatrix}.
$$
That is, the action of $\SL_2(\Z)$ on $(\C\a \oplus \C\bw)\times \h$ is
$$
\gamma : (\a,\tau)\mapsto
\big((c\tau+d)^{-1}\a +(c/2\pi i)\bw,\gamma\tau\big) \text{ and }
\gamma : (\bw,\tau) \mapsto \big((c\tau+d)\bw,\gamma\tau\big). \qed
$$
\end{corollary}

The sections $\a$ and $\bw$ of $\cH_\h$ are both invariant under $\tau \mapsto
\tau+1$. They therefore descend to sections of the quotient $\cH_{\bD^\ast} \to
\bD^\ast$ of $\cH_\h \to \h$ by $\begin{pmatrix} 1 & \Z \cr 0 & 1
\end{pmatrix}$. They trivialize $\cH_{\bD^\ast}$. The vector bundle
$$
\Hbar_\bD := \cO_{\bD}\a \oplus \cO_{\bD}\bw
$$
is thus an extension of $\cH_{\bD^\ast}$ to the $q$-disk. The Hodge bundle
$F^1\cH_{\bD^\ast}$ extends to the sub-bundle $\cO_{\bD}\bw$ of $\Hbar_\bD$.

Denote the natural flat connection on $\cH$ by $\nabla_0$. Equation
(\ref{eqn:wtau}) implies that
$$
\nabla_0 \begin{pmatrix} \a & \bw \end{pmatrix}
=
\begin{pmatrix} \a & \bw \end{pmatrix}
\begin{pmatrix} 0 & 2\pi i \cr 0 & 0 \end{pmatrix} d\tau
=
\begin{pmatrix} \a & \bw \end{pmatrix}
\begin{pmatrix} 0 & 1 \cr 0 & 0 \end{pmatrix} \frac{dq}{q}
$$
from which it follows that the connection on $\Hbar_{\bD}$ is
\begin{equation}
\label{eqn:nabla0}
\nabla_0 = d + \a\frac{\partial}{\partial\bw}\otimes\frac{dq}{q}.
\end{equation}
Since this connection is meromorphic with a simple pole at $q=0$ and nilpotent
residue, the bundle $\Hbar_{\bD}$ is the canonical extension of $\cH_{\bD^\ast}$
to $\bD$. It is the pullback to the $q$-disk of the canonical extension of
$\Hbar \to \Mbar_{1,1}$ of $\cH\to \M_{1,1}$. Denote the fiber of $\Hbar$ over
$q=0$ by $H$. It has basis $\a,\bw$, where $\bw$ spans $F^1H$.

We now compute the limit MHSs $H_\vv$ on $\H$ at $q=0$ associated to the
non-zero tangent vector $\vv$ of $q=0$. It will have integral lattice $H_\Z$,
complexification $H$ and $F^0 H$ defined by
$$
H_\Z = \Z\a \oplus \Z\b,\ H = \C\a \oplus \C\bw \text{ and }
F^0 H = \C\bw \subset H.
$$
To specify the MHS, we give an isomorphism $(H_\Z)\otimes\C \cong H$, which will
depend on the tangent vector $\vv$.\footnote{For the more arithmetically
inclined, $\a,\b$ is a basis of the Betti component $H^B$ of the MHS $H$, and
$\a,\bw$ is a basis of the $\Q$-de Rham component $H^\DR$ of $H$.}

\begin{proposition}
\label{prop:limitMHS}
The $\Z$-MHS $H_\vv(1)$ on $H$ that corresponds to the non-zero tangent
vector $\vv = z\partial/\partial q$ is the MHS determined by the linear
isomorphism $(H_\Z)\otimes\C \cong H_\C$ given by
$$
\begin{pmatrix} \a & \bw \end{pmatrix}
=
\begin{pmatrix} \a & -\b \end{pmatrix}
\begin{pmatrix} 1 & \log z \cr 0 & 2\pi i \end{pmatrix}.
$$
It is the extension of $\Z(-1)$ by $\Z(0)$ that corresponds to $z\in \C^\ast$
under the standard isomorphism $\Ext^1_{\MHS}\big(\Z(-1),\Z(0)\big)\cong
\C^\ast$. It splits if and only if $\vv = \partial/\partial q$.
\end{proposition}

\begin{proof}
Most of this is proved above. The integral lattice of $H_{z\partial/\partial q}$
is computed using the standard prescription and the fact that the value of $\bw$
at $z\in \bD^\ast$ is
$$
\bw(z) = \log z\, \a + 2\pi i (-\b),
$$
which follows from (\ref{eqn:wtau}). The weight filtration $M_\dot$ of the
limit MHS is the monodromy
filtration of the nilpotent endomorphism $\a\partial/\partial\bw$ shifted
by the weight $1$ of $H$:
\begin{equation}
\label{eqn:monod_wt}
M_{-1} H = 0,\ M_0 H = M_1 H = \Q\a,\ M_2 H = H.
\end{equation}
\end{proof}

This result can also be stated by saying that the limit MHS on $H(1)$ of the
variation $\H(1)$ associated to $z\partial/\partial q$ is the extension of
$\Z(0)$ by $\Z(1)$ corresponding to $z\in \C^\ast \cong
\Ext^1_\MHS(\Z(0),\Z(1))$.

\begin{remark}
\label{rem:slH}
Note that if $\vv \in \Q^\times\partial/\partial q$, then $H_\vv$ splits as an
extension of $\Q(-1)$ by $\Q(0)$ if and only if $\vv = \pm \partial/\partial q$.
These are also the only two tangent vectors of the cusp $q=0$ of $\Mbar_{1,1}$
that are defined over $\Z$ and remain non-zero at every prime $p$. For this
reason it is natural to identify the fiber $H$ of $\Hbar$ over $q=0$ with
$H_{\partial/\partial q}$. In particular, $\bw = -2\pi i \b$.

These considerations suggest a natural choice of Cartan subalgebra of $\sl(H)$
and of the positive root vectors. Namely, the Cartan is the one that acts
diagonally with respect to the isomorphism
$$
H_{\pm\partial/\partial q} = \Q(0) \oplus \Q(-1) = \Q\a \oplus \Q\b.
$$
The natural choice of a positive weight vector is the one with positive Hodge
theoretic weight. Since $\b$ has Hodge weight 2 and $\a$ has weight $0$
with respect to the weight filtration (\ref{eqn:monod_wt}), the
natural choice of positive root vectors in $\sl(H)$ is $\Q\b\partial/\partial
\a$. This choice determines a Borel subalgebra of $\sl(H)$.

With this choice of Cartan subalgebra, there are two natural choices of
symplectic bases of $H_\Z$. Namely $\a,\b$ and $-\b,\a$. Because of formula (i)
in Lemma~\ref{lem:invariance}, which is dictated by the standard formula for the
action of $\SL_2(\Z)$ on $\h$, we will use the basis $-\b,\a$. This choice
determines corresponding isomorphisms $\sl(H) \cong \sl_2$ and $\SL(H)\cong
\SL_2$. With respect to the above choice of Cartan subalgebra, $\b$ has
$\sl_2$-weight $1$ and $\a$ has $\sl_2$-weight $-1$.
\end{remark}

\section{Representation Theory of $\SL_2$}
\label{sec:sl2}

This is a quick review of the representation theory of $\sl_2$ and $\SL_2$. Much
of the time, $\SL_2$ will be $\SL(H)$, where $H=H_{\partial/ \partial q}$ is the
fiber of $\H$ over $\partial/\partial q$. As pointed out above, this has a
natural basis, which leads us to a natural choice of Cartan and Borel
subalgebras, which we make explicit below.

Let $F$ denote $\Q$, $\R$ or $\C$. Let $V$ be a two dimensional vector space
over $F$ endowed with a symplectic form (i.e., a non-degenerate, skew symmetric
bilinear form)
$$
\langle \blank,\blank \rangle : V\otimes V \to F.
$$
The choice of a symplectic basis $\v_1,\v_2$ of $V$ determines actions of
$\SL_2(F)$ and $\sl_2(F)$ on $V$ via the formula
\begin{equation}
\label{eqn:formula}
\begin{pmatrix}
a & b \cr c & d
\end{pmatrix}
:
\begin{pmatrix}
\v_2 & \v_1
\end{pmatrix}
\begin{pmatrix}
x_1 \cr x_2
\end{pmatrix}
\mapsto
\begin{pmatrix}
\v_2 & \v_1
\end{pmatrix}
\begin{pmatrix}
a & b \cr c & d
\end{pmatrix}
\begin{pmatrix}
x_1 \cr x_2
\end{pmatrix}
\end{equation}
where $x_1,x_2 \in F$, and isomorphisms $\SL_2 \cong \SL(V)$ and
$\sl_2 \cong \sl(H)$.

We will fix the choice of Cartan subalgebra of $\sl_2$ to be the diagonal
matrices. This fixes a choice of Cartan subalgebra of $\sl(V)$. We will take
$\v_1$ to have $\sl_2$ weight $-1$ and $\v_2$ to have $\sl_2$-weight $+1$. The
element
$$
\begin{pmatrix} 0 & 1 \cr 0 & 0 \end{pmatrix}
$$
of $\sl_2$ corresponds to $\v_2\partial/\partial \v_1 \in \sl(V)$ and has weight
$-2$. Denote it by $\e_0$.

Isomorphism classes of irreducible representations of $\SL(V)$ correspond to
non-negative integers. The integer $n \in \N$ corresponds to the $n$th symmetric
power $S^n V$ of the defining representation $V$. To distinguish distinct
but isomorphic representations of $\SL_2$, we use the notation
$$
S^n(\e) := \{\e_0^j\cdot\e : \e_0^{n+1}\cdot \e = 0\}
$$
to denote the $\SL_2$-module with highest weight vector $\e$ of weigh $n$.

For example, motivated by the discussion in Remark~\ref{rem:slH}, we will
typically work in the following situation:
\begin{enumerate}

\item $V = H$ and $\langle\blank,\blank\rangle$ is the intersection pairing,

\item $\v_1 = -\b$ and $\v_2 = \a$, so that $\b$ has $\sl_2$ weight $+1$ and
$\a$ has $\sl_2$ weight $-1$,

\item $S^n H = S^n(\b^n)$,

\item $\e_0 = - \a\frac{\partial}{\partial \b}$, which has $\sl_2$ weight $-2$.
It is also $-2\pi i$ times the residue at $q=0$ of the connection $\nabla_0$ on
$\Hbar$.

\end{enumerate}

Note that the weight filtration
$$
0 = M_{-1} S^n H \subset M_{0}S^n H \subset \cdots
\subset M_{2n-1}S^n H \subset M_{2n} S^n H = S^n H
$$
associated to the nilpotent endomorphism $\e_0$ of $S^n H$, shifted by the
weight $n$ of $S^n H$, is the filtration obtained by giving $\b$ weight 2 and
$\a$ weight 0:
$$
M_m S^n H = \Span \{\a^{n-j}\b^j : 2j \le m\},
$$
Observe that
\begin{equation}
\label{eqn:hw}
\Gr^M_{2n} S^n H = S^n H/\im \e_0 \cong \Q(-n).
\end{equation}
It is generated by the highest weight vector $\b^n$.

\section{Modular Forms and Eichler-Shimura}

Suppose that $\G$ is a finite index subgroup of $\SL_2(\Z)$. Recall that a
holomorphic function $f : \h \to \C$ is a {\em modular form of weight $w$} for
$\G$ if
\begin{enumerate}

\item $f(\gamma \tau) = (c\tau + d)^w f(\tau)$ for all $\gamma \in \G$. This
implies that $f$ has a Fourier expansion $\sum_{k=-\infty}^\infty a_k q^{k/n}$
for some $n\ge 1$.

\item $f$ is holomorphic at each cusp $P\in \G\bs \P^1(\Q)$. If, for example, $P
= \infty$, this means that the coefficients $a_k$ of the Fourier expansion of
$f$ at $P$ vanish when $k<0$.

\end{enumerate}
A modular form $f$ is a {\em cusp form} if it vanishes at each cusp --- that is,
its Fourier coefficients $a_k$ vanish for all $k\le 0$.

Assume the notation of Section~\ref{sec:sl2}. For an indeterminate $\e$ we have
the $\SL_2$-module $S^m(\e)$ that is isomorphic to $S^m H$. Denote the
corresponding local system over $X_\G$ by $\bS^m(\e)$. As a local system, it is
isomorphic to $S^m\H$. Lemma~\ref{lem:isotypic} implies that if $\bS^m(\e)$ has
the structure of a PVHS, then it is isomorphic to $S^m\H(r)$ for some $r\in \Z$.

For the time being, we will suppose that the PVHS $\bS^m(\e)$ over $X_\G$ has
weight $m$, so that it is an isomorphic copy of $S^m\H$. Define a function $v :
\h \to V$ by
$$
v(\tau,\e) := \exp(\tau \e_0)\e.
$$
The discussion preceding Lemma~\ref{lem:invariance} implies that $v(\tau,\e)$
is a trivializing section of $F^m(\bS^m(\e)\otimes\cO_{\h})$.

\begin{lemma}
For all $\gamma\in
\SL_2(\Z)$ we have $(c\tau+d)^m \gamma^\ast v = \gamma_\ast v$.
\end{lemma}

\begin{proof}
Write $\e = \b^m$, where $\b \in H$. Then $\exp(\tau\e_0)\e =
(\exp(\tau\e_0)\b\big)^m$. So it suffices to consider the case $m=1$. In this
case $v(\tau,\b) = (-\a,\b)(\tau,1)^T$ and
$$
(\gamma_\ast v)(\tau,\b) =
-
\begin{pmatrix}
\a & -\b
\end{pmatrix}
\begin{pmatrix}
a & b \cr c & d
\end{pmatrix}
\begin{pmatrix}
\tau \cr 1
\end{pmatrix}
= - (c\tau + d)
\begin{pmatrix}
\a & -\b
\end{pmatrix}
\begin{pmatrix}
\gamma \tau \cr 1
\end{pmatrix}
= \big(\gamma^\ast v\big)(\tau,\b).
$$
\end{proof}

For a modular form $f$ of $\G$ weight $w=m+2$ and an indeterminate $\e$,
define\footnote{The particular scaling by $2\pi i$ is chosen so that, when
$\G=\SL_2(\Z)$, if $v(\tau,\e)$ corresponds (locally) to a $\Q$-rational section
of the canonical extension $S^m\cH(r)$ of $S^m\H$ to $\Mbar_{1,1}$, then $\w_f$
is $\Q$-rational in the sense that $\w_f(\e) \in H^1_\DR(\M_{1,1/\Q},S^m\cH)$.
See, e.g., \cite[\S21]{hain:kzb}.}
$$
\w_f(\e) = 2\pi i\,f(\tau) v(\tau,\e) d\tau \in E^1(\h)\otimes S^m(\e).
$$
A routine calculation shows that $\w_f(\e)$ is $\G$-invariant in the sense
that
$$
(\gamma^\ast \otimes 1) \w_f(\e) = (1\otimes \gamma) \w_f(\e).
$$
It follows that
$$
\w_f(\e) \in (E^1(\h)\otimes S^m(\e))^\G \cong E^1(X_\G,\bS^m(\e)).
$$
Since $\w_f(\e)$ is closed, it determines a class in $H^1(X_\G,\bS^m(\e))$. Its
complex conjugate
$$
\overline{\w}_f(\e) = \overline{f(\tau)} v(\taubar,\e)^m d\taubar
$$
also defines a class in $H^1(X_\G,\bS^m(\e))$.

Recall that $X_\G = \Xbar_\G - D$, where $D=\G\bs \P^1(\Q)$. And recall from
Section~\ref{sec:mhc} that $K^\dot(\Xbar,D;\bS^n(\e))$ is Zucker's mixed Hodge
complex that computes the MHS on $H^\dot(X_\G,\bS^n(\e))$. It is straightforward
to check that:

\begin{proposition}[Zucker \cite{zucker}]
If $f$ is a modular form of $\G$ of weight $w=m+2$, then
$$
\w_f(\e) \in F^{m+1}K^1(\Xbar_\G,D;\bS^m(\e)).
$$
If $f$ is cusp form, then $\w_f(\e) \in F^{m+1}W_m K^1(\Xbar_\G,D;\bS^m(\e))$.
\end{proposition}

When $m>0$ the exact sequence (\ref{eqn:exact_seqce}) from
Section~\ref{sec:exact_seqce} becomes
\begin{equation}
\label{eqn:exact_seqce:mod}
0 \to W_{m+1}H^1(X_\G,\bS^m(\e)) \to H^1(X_\G,\bS^m(\e))
\to \bigoplus_{P\in D} \big((S^m V_P)/\im N_P\big)(-1) \to 0,
\end{equation}
where $V_P$ denotes the fiber of the canonical extension of $\bS^m(\e)\otimes
\cO_{X_\G}$ to $C$ and $N_P$ the associated local monodromy operator. (Cf.\
Section~\ref{sec:zucker}.) It is an exact sequence of MHS. Note that each $S^m
V_P/\im N_P$ is one dimensional and is isomorphic to $\Q(-m)$ by (\ref{eqn:hw}).

The following result is equivalent to Eichler-Shimura combined with the
observations that
$$
K^1(\Xbar_\G,D;\bS^m(\e)) = W_{2m+1}K^1(\Xbar_\G,D;\bS^m(\e)) \text{ and }
F^{m+2}K^1(\Xbar_\G,D;\bS^m(\e)) = 0.
$$
This version was proved by Zucker in \cite{zucker}.

Denote the space of modular forms of $\G$ of weight $w$ by $\fM_w(\G)$ and the
subspace of cusp forms by $\fM_w^o(\G)$. When $\G = \SL_2(\Z)$, we will omit
$\G$ and simply write $\fM_w$ and $\fM_w^o$.

\begin{theorem}[Shimura, Zucker]
\label{thm:eichler_shimura}
If $\G$ is a finite index subgroup of $\SL_2(\Z)$, then $H^1(X_\G,\bS^m(\e))$ is
spanned by the classes of modular forms of weight $w=m+2$ and their complex
conjugates. The only non-vanishing Hodge numbers $h^{p,q}$ occur when $(p,q)=
(m+1,0)$, $(0,m+1)$ and $(m+1,m+1)$. The weight $m+1$ part is spanned by the
classes of cusp forms and their complex conjugates. Moreover, the function that
takes a modular form $f$ to the class of $\w_f(\e)$ induces an isomorphism
$$
\fM_{m+2}(\G) \cong F^{m+1}H^1(X_\G,\bS^m(\e)).
$$
The map that takes a cusp form $f$ to the class of its complex conjugate
$\wbar_f(\e)$ induces a conjugate linear isomorphism of $\fM_{k+2}^o(\G)$ with
the $(0,m+1)$ part of the Hodge structure $W_{m+1}H^1(X_\G,\bS^m(\e))$.
\end{theorem}

When $\G=\SL_2(\Z)$, there is only one cusp. It is routine to show that the
cohomology with coefficients in $\bS^m(\e)$ vanishes when $m$ is odd. So for
each $n>0$ we have the exact sequence
$$
0 \to W_{2n+1}H^1(\M_{1,1},\bS^{2n}(\e)) \to H^1(\M_{1,1},\bS^{2n}(\e))
\to  \Q(-2n-1) \to 0
$$
The class of the 1-form $\w_f(\e)$ associated to the Eisenstein series
$f=G_{2n+2}$ of weight $2n+2$ projects to a generator of $\Q(-2n-1)$. The MHS
on $H^1(X_\G,\bS^{2n}(\e))$ can be described in terms of modular symbols. We
will return to this in Section~\ref{sec:hodge}.

\subsection{Cohomology of congruence subgroups}

Recall that a congruence subgroup of $\SL_2(\Z)$ is one that contains a
principal congruence subgroup
$$
\SL_2(\Z)[N] :=
\{\gamma \in \SL_2(\Z) : \gamma \equiv \text{ identity }\bmod N\}.
$$
When $\G$ is a congruence subgroup, one has Hecke operators
$$
T_p \in \End_\MHS H^1(X_\G,S^m\H)
$$
for each prime number $p$. Since the Hecke algebra (the algebra generated by the
$T_p$) is semi-simple, $H^1(X_\G,S^m\H_\Q)$ decomposes into simple factors. Each
is a MHS.

Let $\B_w(\G)$ be the set of normalized Hecke eigen cusp forms of $\G$ of weight
$w\ge 2$. This is a basis of $\fM_w^o(\G)$, the weight $w$ cusp forms. For each
$f\in \B_w(\G)$, let $K_f$ be the subfield of $\C$ generated by its Fourier
coefficients.  Since the restriction of $T_p$ to the cusp forms is self adjoint
with respect to the Petersson inner product, its eigenvalues are real. Since this
holds for all $p$, it implies that $K_f \subset \R$. Consequently, the smallest
subspace $V_f$ of $H^1(X_\G,S^m\H_{K_f})$ whose complexification contains
$\w_f(\e)$ is a $K_f$-sub HS of the MHS $H^1(X_\G,S^m\H)$.

Denote by $M_f$ the smallest $\Q$-Hodge sub-structure of $H^1(X_\G,S^m\H_\Q)$
with the property that $M_f\otimes K_f$ contains $V_f$. It is a sum of the Hodge
structures $V_h$ of the eigenforms conjugate to $f$. Call two eigenforms $f$ and
$h$ {\em equivalent} if $M_f = M_h$.

When $f$ is a normalized Eisenstein series the smallest sub-MHS of
$H^1(X_\G,S^m\H_\Q)$ that contains the corresponding cohomology class is one
dimensional and spans a Tate MHS $\Q(-m-1)$.

\begin{theorem}
\label{thm:manin-drinfeld}
If $\G$ is a congruence subgroup of $\SL_2(\Z)$, then the MHS Hodge structure on
$H^1(X_\G,S^m\H)$ splits. In particular, when $m>0$, there is a canonical
isomorphism
$$
H^1(X_\G,S^m\H) \cong
\bigoplus_{f} M_f \oplus \bigoplus_{P\in D} \Q(-m-1),
$$
where $f$ ranges over the equivalence classes of eigen cusp forms of weight
$m+2$. As a real mixed Hodge structure
$$
H^1(X_\G,S^m\H) \cong
\bigoplus_{f\in \B_{m+2}(\G)} V_f \oplus \bigoplus_{P\in D} \R(-m-1).
$$
\end{theorem}

\begin{proof}[Sketch of Proof] The splitting of the MHS on $H^1(X_\G,S^m\H)$
follows from the fact that each Hecke operator $T_p$ is a morphism of MHS. The
weight filtration splits because $T_p$ acts on $\Gr^W_{m+1}$ with eigenvalues of
modulus bounded by $Cp^{1+m/2}$ and on $\Gr^W_{2m+2}$ with eigenvalues of size
$O(p^{m+1})$. (For $\SL_2(\Z)$ this is proved in \cite{serre}. See p.~94 and
p.~106. For general $\G$ see \cite[Lem.~2, p.~13]{lang}.) The $V_f$ are the
common eigenspaces of the $T_p$ acting on $W_{m+1}H^1(X_\G,S^m\H)$. The $M_f$
are their Galois orbits and are $\Q$-HS.
\end{proof}

\section{Hodge Theory of the Relative Completion of Modular Groups}

Here we make the construction of the mixed Hodge structure on the unipotent
radical of the completion of a modular group with respect to its inclusion into
$\SL_2(\Q)$ explicit.

We retain the notation of previous sections: $\G$ is a finite index subgroup of
$\SL_2(\Z)$, $X_\G = \G\bs\h$ is the associated curve, $D=\G\bs\P^1(\Q)$ is the
set of cusps, and $\Xbar_\G = X_\G \cup D$ is its smooth completion. As in
Section~\ref{sec:moduli}, $H$ denotes the fiber over the unique cusp $q=0$ of
$\Mbar_{1,1}$ of the canonical extension $\Hbar$ of the local system $\H$. The
pullback of $\Hbar$ along the quotient morphism $\Xbar_\G \to \Mbar_{1,1}$ is
the canonical extension of $\H\otimes\cO_{X_\G}$ to $\Xbar_\G$, so that the
fiber of $\Hbar_{\Xbar_\G}$ over each $P\in D$ is naturally isomorphic to $H$.

Fix a base point $x_o$ of $X_\G$. We allow $x_o$ to be a non-zero tangent vector
at a cusp $P\in D$. Denote the completion of $\pi_1(X_\G,x_o)$ with respect to
its inclusion into $\SL_2(\Q)$ by $\cG_{x_o}$ and its prounipotent radical by
$\U_{x_o}$. Their Lie algebras (and coordinate rings) have natural mixed Hodge
structures. Recall that the polarized variation $\H$ over $X_\G$ has weight 1.
Denote its fiber over $x_o$ by $H_o$.

We also fix a lift $\tau_o$ of $x_o$ to $\h$. This determines an isomorphism
$\pi_1(X_\G,x_o) \cong \G$ and isomorphisms of $\cG_{x_o}$ and $\U_{x_o}$ with
the completion $\cG$ of $\G$ with respect to the inclusion $\G \hookrightarrow
\SL_2(\Q)$, and $\U$, its prounipotent radical.

\subsection{General considerations}

As pointed out in Section~\ref{sec:modular}, $\u$ is free. So, up to a
non-canonical isomorphism, it is determined by its abelianization $H_1(\u)$.
Theorem~\ref{thm:mhs}, implies that $\u$ has negative weights, so there is an
exact sequence
$$
0 \to W_{-2}H_1(\u)_\eis \to H_1(\u) \to \Gr^W_{-1}\u \to 0.
$$
of pro-MHS with $\SL_2$ action. Eichler-Shimura (Thm.~\ref{thm:eichler_shimura})
and the computation (\ref{eqn:h1u}) imply that the weight $-1$ quotient comes
from cusp forms:
$$
\Gr^W_{-1}\u
= \prod_{m\ge 0} (W_{m+1}H^1(X_\G,S^m\H))^\ast\otimes S^m H_o
= \prod_{m\ge 0} I\!H^1(X_\G,S^m\H)^\ast\otimes S^m H_o.
$$
The exact sequence (\ref{eqn:exact_seqce}) implies that the weight $<-1$ part
$$
W_{-2}H_1(\u) =
\widetilde{H}_0(D;\Q(1)) \oplus \bigoplus_{P\in D} \prod_{m\ge 0} S^m H_o(m+1)
$$
is a direct product of Hodge structures. Note that when $x_o$ is a finite base
point (i.e., $x_o\in X_\G$), then $S^mH_o(m+1)$ has weight $-m-1$. If $x_o$ is a
tangent vector at a cusp, then $S^m H_o(m+1)$ has weight graded quotients
$\Q(1),\ \Q(2),\dots,\Q(m+1)$. In this case, $W_{-2}H_1(\u)$ is mixed Tate.

The Manin-Drinfeld Theorem (Thm.~\ref{thm:manin-drinfeld}) implies:

\begin{proposition}
If $\G$ is a congruence subgroup of $\SL_2(\Z)$, then $H_1(\u)$ is the
product
$$
H_1(\u) \cong \prod_{r<0} \Gr^W_r H_1(\u)
$$
of its weight graded quotients in the category of pro-MHS with $\SL_2$ action.
\qed
\end{proposition}

\subsection{Hodge theory of congruence subgroups}

Now assume that $\G$ is a congruence subgroup of $\SL_2(\Z)$. The first step in
writing down a formal connection $\Omega \in K^1(\Xbar_\G,D;\bu)$ is to write
down, for each $m > 0$, a form
$$
\Omega_{1,m} \in
F^0W_{-1} \big(K^1(\Xbar_\G,D;S^{m}\H)\otimes H^1(X_\G,S^{m}\H)^\ast\big)
$$
that represents the identity $H^1(X_\G,S^{m}\H)\to H^1(X_\G,S^{m}\H)$.

Since $\G$ is a congruence subgroup, the Hecke algebra acts on the modular forms
$\fM_w(\G)$ of $\G$ of weight $w$. When $w>2$, for each cusp $P\in D$, there is
a normalized Eisenstein series $E_{w,P}(\tau)$ that vanishes at the other cusps.
When $w=2$, Eisenstein series give elements of $H^1(X_\G,\C)$ with non-zero
residues at least two cusps. Fix a cusp $P_o\in D$. For each $P\in D':=
D-\{P_o\}$, choose an Eisenstein series $E_{2,P}$ that is non-zero at $P$ and
vanishes at all other points of $D'$.

Now suppose that $m\ge 0$. Identify $S^m\H$ with $\bS^m(\b^m)$. For each $f\in
\B_{m+2}(\G)$ we have the 1-forms $\w_f(\b^m) \text{ and } \wbar_f(\b^m)$. When
$m>0$ (resp.\ $m=0$) and $P\in D$ (resp.\ $P\in D'$) set
$$
\psi_{m,P}(\b^m) := \w_{E_{m,P}}(\b^m).
$$
This will be viewed as an element of $K^1(\Xbar_\G,D;S^m\H)$ and of
$E^1(\h,S^nH)^\G$. Then
\begin{multline}
\label{eqn:basis}
\{\w_f(\b^m),\ \wbar_f(\b^m) : f\in \B_{m+2}(\G)\}
\cup
\cr
\{\psi_{m+2,P}(\b^m): (P\in D
\text{ and } m>0) \text{ or } (P\in D' \text{ and } m=0)\}
\end{multline}
is a subset of $K^1(\Xbar_\G,D;S^m\H)$ that represents a basis of
$H^1(X_\G,S^m\H)$. Let
$$
\{u_f',u_f'',u_{m+2,P}: f\in \B_{m+2}(\G),\ P\in D\}
$$
be a basis of $H^1(X_\G,S^m\H)$ dual to the cohomology classes of the closed
forms (\ref{eqn:basis}). The Hodge types of $u_f'$, $u_f''$ and $u_{m+2,P}$ are
$(-m-1,0)$, $(0,-m-1)$ and $(-m-1,-m-1)$, respectively
 
Set
\begin{equation}
\label{eqn:gens}
\e_f' := \b^m\otimes u_f ,\ \e_f'' := \b^m\otimes \ubar_f,\
\e_{m+2,P} = \b^m\otimes u_{m+2,P}.
\end{equation}
These are elements of $S^m H\otimes H^1(X_\G,S^m\H)^\ast$. Then
\begin{multline*}
\w_f(\b^m)\otimes u_f' = \w_f(\e_f'),\ \wbar_f(\b^m)\otimes u_f''
= \wbar_f(\e_f''),
\cr
\text{ and } \psi_{m+2,P}(\b^m)\otimes u_{m+2,P}
= \psi_{m+2,P}(\e_{m+2,P}).
\end{multline*}
All are elements of $K^1(\Xbar_\G,D;S^m\H)\otimes H^1(X_\G,S^m\H)^\ast$ and
\begin{equation}
\label{eqn:omega_1}
\Omega_{1,m} =
\sum_{f\in \B_{m+2}(\G)} \big(\w_f(\e_f') + \wbar_f(\e_f'')\big) + \sum_{P}
\psi_{m+2,P}(\e_{m+2,P})
\end{equation}
is a closed 1-form that represents the identity $H^1(X_\G,S^m\H) \to
H^1(X_\G,S^m\H)$. Here the second sum is over $P\in D$ when $m>0$ and $P\in D'$
when $m=0$.

\begin{lemma}
For each $m\ge 0$
$$
\Omega_{1,m} \in F^0 W_{-1}
\big(K^1(\Xbar_\G,D;S^m\H)\otimes H^1(X_\G,S^m\H)^\ast\big).
$$
\end{lemma}

\begin{proof}
Since the Hodge types of $u_f'$, $u_f''$ and $u_{m+2,P}$ are $(-m-1,0)$,
$(0,-m-1)$ and $(-m-1,-m-1)$,  the definitions of the Hodge and weight
filtrations of Zucker's mixed Hodge complex $K^\dot(\Xbar_\G,D;S^m\H)$ imply
that
$$
\w_f(\b^m) \in F^{m+1}W_{m}K^1(\Xbar_\G,D;S^m\H) \text{ and }
\wbar_f(\b^m) \in F^0W_{m}K^1(\Xbar_\G,D;S^m\H)
$$
when $f\in \B_{m+2}(\G)$ and $\psi_{m+2,P}(\b^m) \in
F^{m+1}W_{2m+1}K^1(\Xbar_\G,D;S^m\H)$ for each $P$. The result follows as
$u_f'$, $u_f''$ and $u_{m+2,P}$ have Hodge types $(-m-1,0)$, $(0,-m-1)$ and
$(-m-1,-m-1)$, respectively.
\end{proof}

The Lie algebra $\u$ of the prounipotent radical $\U$ of the relative completion
of $\G$ is the free pronilpotent Lie algebra $\u = \L(V)^\wedge$, where
$
V = \bigoplus_{m\ge 0}V_m
$
and
\begin{align*}
V_m &:= H^1(X_\G,S^m\H)^\ast\otimes S^m H
\cr
&=
\begin{cases}
H_1(X_\G,\C) & m = 0,\cr
\bigoplus_{f\in \B_{m+2}(\G)}
\big(S^m(\e_f')\oplus S^m(\e_f'')\big)\oplus\bigoplus_{P\in D} S^m(\e_{m+2,P})
& m > 0.
\end{cases}
\end{align*}
It is a Lie algebra in the category of pro-representations of $\SL_2$.

The 1-form
$$
\Omega_1 := \sum_{m\ge 0} \Omega_{1,m} \in E^1(\h)\comptensor \u
$$
is $\G$-invariant and represents the identity. It can thus be completed to a
power series connection
$$
\Omega \in F^0W_{-1}K^1(\Xbar_\G,D;\bu)
$$
using the method described in Section~\ref{sec:relcomp_aff}, which determines
the MHS on $\u_{x_o}$.

Before discussing the case $\G=\SL_2(\Z)$, note that since
\begin{align*}
S^{m}(\e_f') &= \Span\{\e_0^j \cdot \e_f' : \e_0^{m+1} \cdot \e_f' =0\},
& f\in \B_{m+2}(\G)
\cr
S^{m}(\e_f'') &= \Span\{\e_0^j \cdot \e_f'' : \e_0^{m+1} \cdot \e_f'' =0\},
& f\in \B_{m+2}(\G)
\cr
S^{m}(\e_{m+2,P}) &=
\Span\{\e_0^j\cdot\e_{m+2,P} : \e_0^{m+1}\cdot\e_{m+2,P}=0\}, & P\in D,
\end{align*}
$\u$ is the free Lie algebra topologically generated by
$$
\{\e_{2,P}:P\in D'\}\cup
\bigcup_{m>0}
\{
\e_0^j\cdot\e_f',\ \e_0^j\cdot\e_f'',\ \e_0^j\cdot\e_{P,m+2}:
0 \le j \le m,\ f \in \B_{m+2},\ P\in D\}.
$$
The Hodge and weight filtrations of $\u$ are defined by giving $\b \in H$
type $(1,0)$. The generators (\ref{eqn:gens}) thus have types given in
Figure~\ref{fig:hodge_types}.
\begin{figure}[!ht]
$$
\begin{array}{|c|c|c|c|c|}
\hline
& & \text{Hodge type} & W\text{-weight} & M\text{-weight} \cr
\hline
\e_0 &  & (-1,1) & \phantom{-}0 & -2\cr
\e_f' & f \in \B_{m+2}(\G) & (-1,0) & -1 & m-1 \cr
\e_f'' & f \in \B_{m+2}(\G) & (m,-m-1) & -1 & m-1 \cr
\e_{m+2,P} & P\in D & (-1,-m-1) & -m-2 & -2\cr
\hline
\end{array}
$$
\caption{Hodge types of the generators of $\u$}
\label{fig:hodge_types}
\end{figure}
So, for example, $\e_0^j\cdot \e_f'$ has type $(-1-j,j)$. The Hodge and weight
filtrations on the generators extend naturally to Hodge and weight filtrations
on $\u$.

The Hodge types on the $\sl_2$ module with highest weight vectors
$\e_f',\e_f'',\e_{m+2,P}$ are illustrated in Figure~\ref{fig:types}.
\begin{figure}[!ht]
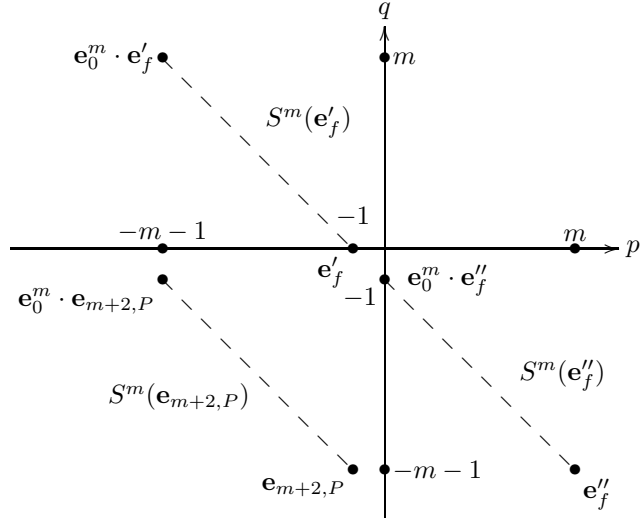

$$
\def\latticebody{\drop{}}
\xy
*\xybox{0;<1pc,0pc>:<0pc,1pc>::0
,{\xylattice{-8}7{-8}7}
,(-7,6)*{\bullet}*+!CR{\e_0^m\cdot\e_f'}
;(-1,0)*{\bullet}**@{--}*+!UR{\e_f'}
,(-4,4)*+!CL{S^m(\e_f')}
,(-1,1)*{-1}
,(0,-1)*{\bullet}*+!UR{-1}
;(6,-7)*{\bullet}**@{--}*+!UL{\e_f''}
,(4,-4)*+!CL{S^m(\e_f'')}
,(2,-1)*+{\e_0^m\cdot\e_f''}
,(-7,-1)*{\bullet}*+!UR{\e_0^m\cdot\e_{m+2,P}}
;(-1,-7)*{\bullet}**@{--}*+!UR{\e_{m+2,P}}
,(-4,-4)*+!UR{S^m(\e_{m+2,P})}
%
%,(-7,0)*{\bullet}*+!DR{-m-1}
%
%,(0,-7)*{\bullet}*+!UL{-m-1}
%
,(6,0)*{\bullet}*+!DC{m}
,(0,6)*{\bullet}*+!CL{m}
,(-7,0)*{\bullet}*+!DC{-m-1}
,(0,-7)*{\bullet}*+!CL{-m-1}
}="L"
,{"L"+L \ar "L"+R*+!L{p}}
,{"L"+D \ar "L"+U*+!D{q}}
\endxy
$$
\caption{Hodge numbers of $S^m(\e_f')$, $S^m(\e_f'')$ and $S^m(\e_{m+2,P})$.}
\label{fig:types}
\end{figure}

\subsection{The case of $\SL_2(\Z)$}

In this case, $X_\G = \M_{1,1}$, $C=\Mbar_{1,1}$ and $D$ consists of a single
point, which we shall denote by $P$. The modular parameter $q=e^{2\pi i\tau}$ is
a local holomorphic coordinate on the orbifold $\Mbar_{1,1}$ centered at $P$.

There are no modular forms for $\SL_2(\Z)$ of odd weight. Since there is a
single cusp, there is a 1-dimensional space of Eisenstein series for each
weight $2n\ge 4$. The normalized Eisenstein series of even weight $2n$ is
\begin{equation}
\label{eqn:eisenstein}
G_{2n}(\tau) = \frac{1}{2}\frac{(2n-1)!}{(2\pi i)^{2n}}
\sum_{\substack{\lambda \in \Z\oplus \Z\tau\cr\lambda \neq 0}}
\frac{1}{\lambda^{2n}}
=-\frac{B_{2n}}{4n} + \sum_{k=1}^\infty \sigma_{2n-1}(k)q^k.
\end{equation}
This has value $(2n-1)!\zeta(2n)/(2\pi i)^{2n}$ at the cusp $P$.\footnote{Here
$B_{2n}$ is the $2n$th Bernoulli number, $\zeta(s)$ is the Riemann zeta function
and $\sigma_k(n)$ is the sum of the $k$th powers of the divisors of $n$.} The
dual homology class $\e_{2n,P}$ will be denoted by $\e_{2n}$ and the form
$\psi_{2n,P}$ by $\psi_{2n}$.

Later, we will use the tangent vector $\vv := \partial/\partial q$ of $P$ as a
base point. We will also write $\B_w$ instead of $\B_w(\SL_2(\Z))$.

\section{VMHS associated to Modular Forms and their Period Maps}
\label{sec:vmhs}

This section considers three related topics: relative higher albanese maps
(which are related to period mappings of VMHS), Manin's iterated Shimura
integrals, and the existence of extensions of VMHS coming from Hecke eigenforms.
The construction of these extensions is a special case of the general technique 
for constructing admissible VMHS sketched in the proof of 
Theorem~\ref{thm:equivalence}. Such extensions correspond to normal functions,
so every Hecke eigenform produces a normal function.

\subsection{Relative Albanese maps}

This construction generalizes the non-abelian albanese manifolds of
\cite{hain-zucker} from the unipotent case to the relative case. Although this
discussion applies more generally, here we restrict to the case of modular
curves.

Fix a base point $\tau_o$ of $\h$. Let $\G$ be a finite index subgroup of
$\SL_2(\Z)$. Denote the image of $\tau_o$ in $X_\G$ by $x_o$. The choice of
$\tau_o$ determines an isomorphism $\G\cong \pi_1(X_\G,x_o)$. Let $\cG_o$ be the
complex form of the relative completion of $\pi_1(X_\G,x_o)$. Let $\U_o$ be its
prounipotent radical, and let $\g_o$, $\u_o$ be their Lie algebras. Since the
bracket of $\g_o$ respects the Hodge filtration, $F^0\g_o$ is a subalgebra of
$\g_o$. Denote the corresponding subgroup of $\cG_o$ by $F^0\cG_o$.

Let
$$
\Omega \in F^0W_{-1}K^1(\Xbar_\G,D;\bu)
$$
be as above. Trivialize the pullback of $\bu$ to $\h$ using the sections
$\{\e_0^j\cdot \e_f',\e_0^j\cdot \e_f'',\e_0^j\cdot \e_{2n}\}$:
$$
\u\times \h \to \h.
$$
This trivializes both the Hodge and weight filtrations. It also fixes an
isomorphism $\cG_x \cong \SL(H_x)\ltimes \U_x$ for all $x\in X_\G$.  Denote the
pullback of $\Omega$ to $\h$ by
$$
\Omegatilde \in (E^1(\h)\comptensor \u)^\G.
$$

For convenience we now take the base point $\tau_o$ to be $i$. Since
$\Omegatilde$ is integrable, the function $F : \h \to \U_o\times \SL(H_o) \cong
\cG_o$ defined by
$$
\Ftilde(\tau) :=
\bigg(1 + \int_{\tau_o}^\tau \Omegatilde
+ \int_{\tau_o}^\tau \Omegatilde\Omegatilde
+ \int_{\tau_o}^\tau \Omegatilde\Omegatilde\Omegatilde + \cdots,
\textstyle{
\begin{pmatrix}
v & u/v \cr 0 & v^{-1}
\end{pmatrix}
}\bigg)
$$
where $\tau = u + i v^2$ ($u$, $v$ real), is well defined and smooth. It induces
a function $F:\h \to \cG_o/F^0\cG_o$ which is equivariant with respect to the
natural left $\G$-actions on each.

\begin{proposition}
The function $F : \h \to \cG_o/F^0\cG_o$ is holomorphic. (That is, it is
an inverse limit of holomorphic functions.)
\end{proposition}

\begin{proof}
Set $H = H_o$. The map $\SL(H_\R) \to \h$ that takes $g$ to $g\tau_o = gi$
induces an isomorphism $\SL(H_\R)/\SO(2) \to \h$. The inclusion $\SL(H_\R)
\hookrightarrow \SL(H_\C)$ induces an inclusion
$$
\h = \SL(H_\R)/\SO(2) \hookrightarrow \SL(H_\C)/F^0 \cong \P^1.
$$
Since the right hand matrix in $\Ftilde$ is in $\SL(H_\R)$ and takes $i$ to
$\tau$, the composition of $F$ with the projection $\cG_o/F^0\cG_o \to
\SL(H_{\C})/F^0$ is the inclusion $\h\hookrightarrow \SL(H_{\C})/F^0$.
Consequently, to show that $F$ is holomorphic, we need only check that its first
factor is holomorphic. Write $\Omegatilde = \Omegatilde' + \Omegatilde''$, where
$\Omegatilde'$ has type $(1,0)$ and $\Omegatilde''$ has type $(0,1)$. Since
$\Omega \in F^0K^1(\Xbar_\G,D;\bu)$,
$$
\Omegatilde'' \in E^{0,1}(\h)\comptensor F^0\u.
$$
The fundamental theorem of calculus implies that $F$ satisfies the differential
equation $d\Ftilde = \Ftilde\Omegatilde$ from which it follows that $\delbar
\Ftilde = \Ftilde\Omega''$. This implies the vanishing of $\delbar F$.
\end{proof}

The projection $\cG_o \to \SL(H_{o,\C})$ preserves the Hodge filtration and
induces a holomorphic, $\G$-invariant projection
$\cG_o/F^0\cG_o \to \SL(H_o)/F^0 \cong \P^1$.
Then one has the diagram
$$
\xymatrix{
(\cG_o/F^0\cG_o)|_\h\ar[d]\ar[r] & \cG_o/F^0\cG_o \ar[d]\cr
\h \ar@{^{(}->}[r]\ar[ur]^F & \SL(H_{x,\C})/F^0 \ar[r]^(.6)\simeq & \P^1.
%%) %% added to balance parentheses
}
$$

Recall that our choices have fixed an isomorphism $\cG_o \cong \SL(H_o)\ltimes
\U_o$. The natural homomorphism corresponds to a non-abelian 1-cocycle $\Theta_o
: \G \to \U_o(\Q)$.

\begin{lemma}
The function $F: \h \to \cG_o/F^0$ satisfies $F(\gamma\tau) =
\Theta_o(\gamma)F(\gamma\tau)$ for all $\gamma \in \G$.
\end{lemma}

\begin{proof}
The identification $\cG_o \cong \SL(H_o)\ltimes \U_o$ induces an identification
of $\cG_o/F^0$ with $\U_o\times \P^1$. If $c_\gamma$ is a path from $\tau_o$ to
$\gamma\tau_o$ and $c$ is a path from $\tau_o$ to $\tau$, then
$c_\gamma\ast(\gamma\cdot c_\gamma)$ is a path from $\tau_o$ to $\gamma\tau$. So
$$
F(\gamma\tau) = \big(\Theta_o(\gamma)T(c)^{-1},\gamma\tau\big)
= \Theta_o(\gamma)F(\tau),
$$
where $T^{-1}$ denotes $1 + \int\Omegatilde + \int\Omegatilde\Omegatilde +
\cdots$.
\end{proof}

Let $\U_{o,\Z}$ be the subgroup of $\U_o$ that is generated by
$\{\Theta_o(\gamma) : \gamma \in \G\}$. Let $\cG_{o,\Z}$ be the subgroup of
$\cG_o \cong \SL(H_o)\cong \U_o$ that corresponds to $\SL(H_{o,\Z})\ltimes
\U_{o,\Z}$. The previous result implies that $F(\gamma\tau)$ and $\gamma
F(\tau)$ lies in the same left $\U_{o,\Z}$ orbit.

The {\em universal relative albanese manifold} $\bA_\G$ of $X_\G$ is defined by
$$
\bA_\G = \U_{o,\Z}\bs\big((\cG_o/F^0\cG_o)|_\h\big).
$$
There is a natural quotient mapping to $X_\G$.  Taking the quotient of the
left-hand map of the previous diagram by $\G$ gives the {\em universal
non-abelian Albanese map}
$$
\xymatrix@C=3pc{
{\bA_\G} \ar[r] & X_\G\ar@/_0.8pc/[l]_F.
}
$$
The fiber over $x_o\in X_\G$ is isomorphic to $\U_{o,\Z}\bs\U_o/F^0$.

\begin{remark}
For each finite dimensional quotient $G_\alpha$ of $\cG_o$ in the category of
groups with a MHS, one can define
$$
\bA_\alpha = G_{\alpha,\Z}\bs\big((G_\alpha/F^0G_\alpha)|_\h\big).
$$
Since $\cG_o$ is the inverse limit of the $G_\alpha$, $\bA_\G = \varprojlim
\bA_\alpha$. The reduction of the section $F$ is a holomorphic section
$F_\alpha$ of this bundle. In particular, when $\G$ is a congruence subgroup, by
taking $G_\alpha$ to be the quotient
$$
0 \to \Mdual_f \otimes S^nH_o \to G_n \to \SL(H_o) \to 1
$$
of $\cG_o$, where $M_f$ is the smallest $\Q$-sub HS of $H^1(X_\G,S^m\H)$ that
contains $\w_f(\e)$ and $\Mdual_f$ its dual, we see that for each $f\in
\B_{n+2}(\G)$, there is a bundle over $X_\G$ with fiber over $x\in X_\G$ the
intermediate jacobian
$$
\Ext^1_\MHS(\Z,\Mdual_{f,\Z}\otimes S^n H_{x,\Z})
$$
with a holomorphic section induced by $F$. We will see below that this is a
normal function which is the period mapping of an extension of $\Q$ by
$\Mdual_f\otimes S^n\H$. When $\G$ is a congruence subgroup, each Eisenstein
series determines a normal function that corresponds to an extension of $\Z$ by
$S^n\H(n+1)$.
\end{remark}
\medskip

The coefficients of $F$ are holomorphic functions on $\h$ which can be realized
as periods of admissible variations of MHS. These include iterated integrals of
(holomorphic) modular forms, but there are many more. Below is an example of
such a holomorphic iterated integral that is not of the type considered by
Manin. It is, in some sense, a generalization of the Riemann theta function.
(Cf.\ \cite[Ex.~4.4]{hain:prospects}.)

\begin{example}
Suppose that $f\in M^o_{2n+2}$ and $g\in M^o_{2m+2}$ are cusp forms, not
necessarily of the same weight. Since $\SL_2(\Z)$ has cohomological dimension 1,
the form $\wbar_f(\e_f'')\wedge \w_g(\e_g')$ is exact. Lemma~\ref{lem:exact}
implies that there is
$$
\xi \in F^0 W_{-1}K^1\big(\Xbar_\G,D;\bS^{2n}(\e_f'')\otimes
\bS^m(\e_g')\big)
$$
such that 
$$
\delbar \xi + \wbar_f(\e_f'')\wedge \w_g(\e_g') = 0.
$$
Then
$$
\tau \mapsto \int_{\tau_o}^\tau \wbar_f(\e_f'')\, \w_g(\e_g') + \xi
$$
is a well defined function from $\h$ to $S^{2n}H_o \otimes S^{2m}H_o$. An
elementary argument (cf.\ \cite[Prop.~4.3]{hain:prospects}) implies that it is
holomorphic.

Perhaps the most interesting version of this example is where $f=g$ and one
composes it with an invariant bilinear form $S^{2n}H_o\otimes S^{2n}H_o\to \C$.
Such iterated integrals occur as periods of biextensions.
\end{example}

\subsection{Iterated Shimura integrals}
\label{sec:manin}

In \cite{manin1,manin2} Manin considered iterated integrals of holomorphic
modular forms and non-abelian generalizations of modular symbols. Here we
briefly discuss the relationship of his work to the Hodge theory of modular
groups. Recall that $\B_w = \B_w(\SL_2(\Z))$.

Manin considers iterated integrals with values in the algebras
$$
A =
\C\ll \e_0^j\cdot \e_{2n},\ \e_0^j\cdot \e_f'
: f\in\B_{2n},\ 0\le j \le 2n-2,\  n\ge 2 \rr
$$
and 
$$
B = \C\ll \e_0^j\cdot \e_{2n} :  0\le j \le 2n-2,\ n\ge 2\rr.
$$
These iterated integrals are of the form
$$
1 + \int_{\tau_o}^\tau \Omega + \int_{\tau_o}^\tau \Omega\Omega
+ \int_{\tau_o}^\tau \Omega \Omega \Omega + \cdots
$$
where
$$
\Omega = \Omega_A :=
\sum_{n\ge 2}\bigg(\psi_{2n}(\e_{2n}) + \sum_{f\in \B_{2n}} \w_f(\e_f')
\bigg)
$$
in the first case and
\begin{equation}
\label{eqn:second}
\Omega = \Omega_B := \sum_{n\ge 2}\psi_{2n}(\e_{2n})
\end{equation}
in the second case. Both of these forms are $\G$-invariant.

Let $\u_A$ (resp.\ $\u_B$) be the set of primitive elements of $A$ (resp.\
$\u_B$). Then $\Omega_A$ (resp.\ $\Omega_B$) takes values in $\u_A$ (resp.\
$\u_B$). Set
$$
\u = \L(\e_0^j\cdot \e_{2n},\ \e_0^j\cdot \e_f',\ \e^0\cdot\e_f''
: f\in\B_{2n},\ 0\le j \le 2n-2,\  n\ge 2)^\wedge.
$$
It follows from Figure~\ref{fig:types} that $F^0\u$ is generated by $\{\e_0^j
\cdot \e_f'' : f\in \B_{2n}, n\ge 2\}$, so that $\u_A$ is the quotient of $\u$
by the ideal $(F^0\u)$ {\em generated by} $F^0\u$, and $\u_B$ is the quotient of
$\u$ by the ideal generated by all $\e_0^j\cdot \e_f'$ and $\e_0\cdot \e_f''$.
His iterated integral is the reduction of the one in the previous section (the
first argument of $\Ftilde$) mod these ideals. This implies that there are many
interesting holomorphic iterated integrals which do not occur as iterated
Shimura integrals.

One can ask whether the $\SL_2(\Z)$ connection on the local system
$$
\u_B \times \h \to \h
$$
defined by $\Omega_B$ descends to an admissible VMHS over $\M_{1,1}$. It will
follow from Theorem~\ref{thm:quad} that it does not as we explain in
Remark~\ref{rem:manin}.

\subsection{Extensions of variations of MHS associated to Eisenstein series}
\label{sec:eisenstein}

Here we suppose, for simplicity, that $\G = \SL_2(\Z)$. In this section, we
sketch an explicit construction of an extension
$$
0 \to S^{2n}\H(2n+1) \to \bE \to \Q \to 0
$$
for each Eisenstein series $G_{2n+2}$.\footnote{This construction works equally
well when $\G$ is a congruence subgroup. The construction in the more general
case is sketched at the end of the next section.}

Let $H=\C\a \oplus \C\bw$. Define the Hodge filtration on $H$ by $F^0 H = H$ and
$F^1 H = \C\bw$. This induces a Hodge filtration on $S^{2n}H$. Trivialize the
bundle $\cH_\h \to \h$ with the sections $\a$ and $\bw$:
$$
\cH_\h \cong H \times \h.
$$
Trivialize $S^{2n}\cH_\h$ using monomials in $\a$ and $\bw$. Then $F^p
S^{2n}\cH_\h$ is trivialized by the sections $\{\a^{2n-j}\bw^j : j\ge p\}$.

Set $V = \C\e \oplus S^{2n} H(2n+1)$ and $\cV_\h = V\times \h$. Define Hodge and
weight filtrations on $V$ by giving $\e$ type $(0,0)$ and $\a^{2n-j}\bw^j$ type
$(j-2n-1,-j-1)$. Let $\SL_2(\Z)$ act in this bundle by acting trivially on $\e$,
and on $\a$ and $\bw$ by the factor of automorphy given in
Corollary~\ref{cor:automorphy}. The Hodge and weight filtrations are invariant
under this action, so that they descend to Hodge and weight filtrations on the
(orbifold) quotient bundle
$$
\cV := \SL_2(\Z)\bs \cV_\h \to \M_{1,1}.
$$
This bundle is trivial over the punctured $q$-disk $\D^\ast$. Extend it to a
bundle $\Vbar$ over $\Mbar_{1,1}$ by defining its sections over the $q$-disk
$\D$ to be $V\otimes \cO_{\D}$. The Hodge and weight bundles clearly extend to
sub-bundles of $\Vbar$.

Define a connection on $\cV_\h$ by $d + \Omega$, where
$$
\Omega = 
\begin{pmatrix}
0 & 0 \cr \psi_{2n+2}(\bw^{2n}) & \a\frac{\partial}{\partial \bw}\frac{dq}{q}
\end{pmatrix}
\in
\begin{pmatrix}
\C & 0 \cr S^{2n}H & \End S^{2n}H
\end{pmatrix}\, \frac{dq}{q}.
$$
It is holomorphic, flat and $\G$-invariant. It therefore descends to a flat
connection $\nabla$ on $\cV$ which has a regular singular point at the cusp when
viewed as a connection on $\Vbar$. This implies that the extended bundle is
Deligne's canonical extension of $(\cV,\nabla)$ to $\Mbar_{1,1}$. Since
$$
\Omega \in (F^{-1}W_{-1} \End V)\otimes \cO(\D)\frac{dq}{q},
$$
the weight filtration is flat and the connection satisfies Griffiths
transversality. Since $\psi_{2n+2}(\bw^{2n})$ has rational periods, it follows
that the local system $\V$ associated to $(\cV,\nabla)$ has a natural $\Q$-form.
The associated weight graded local system is
$$
\Gr^W_\dot \V = \Q(0) \oplus S^{2n}\H(2n+1).
$$
The existence of a relative weight filtration at $q=0$ follows from the argument
in the proof of Theorem~\ref{thm:admissible}  It follows that $\V$ is an
admissible variation of MHS over $\M_{1,1}$. The results of
Section~\ref{sec:ext_vmhs} imply that every extension of $\Q$ by $S^{2n}\H(m)$
over $\M_{1,1}$ is a multiple of this extension when $m=2n+1$ and trivial
otherwise.

\subsection{Extensions of variations of MHS associated to cusp forms}
\label{sec:cusp}

The construction of the extension corresponding to an eigen cusp form is
similar, but a little more elaborate. Suppose that $\G$ is a congruence subgroup
of $\SL_2(\Z)$ and that $m\ge 0$. The first step is to construct an extension
\begin{equation}
\label{eqn:mod_extn}
0 \to H^1(X_\G,S^m\H_\Z)^\ast \otimes S^m \H_\Z \to \V \to \Z \to 0
\end{equation}
in the category of $\Z$-MHS over $X_\G$.

Denote the completion of $X_\G$ by $\Xbar_\G$. Let $\Vbar$ be the $C^\infty$
vector bundle over $\Xbar_\G$ associated to the canonical extension
$$
\cO_{\Xbar_\G} \oplus H^1(X_\G,S^m\H)^\ast \otimes S^m \Hbar
$$
of the admissible variation $\Q\oplus H^1(X_\G,S^m\H)^\ast \otimes S^m \H$ over
$X_\G$. This has natural Hodge and weight sub-bundles. Denote the restriction of
$\Vbar$ to $X_\G$ by $\cV$ and the direct sum connection on it by $\nabla_0$.

Define a $C^\infty$ connection on $\cV$ by $\nabla = \nabla_0 + \Omega_{1,m}$,
where $\Omega_{1,m}$ is the form defined in Equation~(\ref{eqn:omega_1}). This
connection is flat, and thus defines a new holomorphic structure on the bundle
$\cV$. Arguments almost identical to those in Section~\ref{sec:holo} show that
$(\Vbar,\nabla)$ is Deligne's canonical extension of $(\cV,\nabla)$, that the
Hodge bundles are holomorphic sub-bundles of $\Vbar$ with respect to this new
complex structure, and that the connection $\nabla$ satisfies Griffiths
transversality. The existence of a relative weight filtration at each cusp is
established as in the proof of Theorem~\ref{thm:admissible}. The fact that
$\Omega_{1,m}$ represents the identity $H^1(X_\G,S^m\H)\to H^1(X_\G,S^m\H)$
implies that the local system $\V$ underlying the flat bundle $(\cV,\nabla)$ has
a natural $\Z$-form. It follows that there is an admissible $\Z$-VMHS $\V$ over
$X_\G$ whose corresponding $C^\infty$ vector bundle is $\cV$ and whose weight
graded quotients are $\Q(0)$ and $H^1(X_\G,S^m\H_\Z)^\ast \otimes S^m \H$. (Cf.\
Lemma~\ref{lem:wt_bdles}.)

Having constructed the extension (\ref{eqn:mod_extn}), we can now construct the
extension corresponding to a Hecke eigen cusp form $f\in \B_{m+2}(\G)$. The
smallest sub $\Q$-HS $M_f$ of $H^1(X_\G,S^m\H_\Q)$ whose complexification
contains $\w_f(\e)$ is pure of weight $m+1$. So $\Mdual_f \otimes S^m \H$ is
pure of weight $-1$. The corresponding extension
$$
0 \to \Mdual_f \otimes S^m \H_\Z \to \bE_f \to \Q \to 0
$$
is obtained by pushing out the extension (\ref{eqn:mod_extn}) along the dual of
the inclusion
$$
M_f \hookrightarrow H^1(X_\G,S^m\H_\Z).
$$
This extension has a natural $\Z$-form, which we denote by $E_{f,\Z}$.

The extension $\bE_{f,\Z}$ corresponds to a holomorphic section of the
associated bundle of intermediate jacobians, which has fiber
$$
J(H^1(X_\G,S^m \H)^\ast \otimes S^m H_x)
$$
over $x\in X_\G$, where for a $\Z$-MHS $V$ with negative weights
$$
J(V) := V_\C/\big(V_\Z+F^0V_\C) \cong \Ext^1_\MHS(\Z,V).
$$
The section is obtained by integrating the invariant 1-form $\w_f(\e_f') +
\w_f(\e_f'')$. More sections can be obtained by applying elements of
$\Aut M_{f,\Z}$.

A similar construction can be used to construct the extension of a normalized
Eisenstein series $f$. When $\G = \SL_2(\Z)$ this reduces to the construction in
the previous section. In this case, the smallest $\Q$-Hodge sub structure $M_f$
of $H^1(X_\G,S^m\H)$ that contains $\psi_f(\e)$ is $M_f =\Q(-m-1)$. Pushing out
the extension (\ref{eqn:mod_extn}) along the inclusion $M_f \to H^1(X_\G,S^m\H)$
gives the extension
$$
0 \to S^m\H(m+1) \to \bE_f \to \Q \to 0.
$$
corresponding to $f$.

\section{The Relative Completion of $\pi_1(\M_{1,\vec{1}},x)$}

By a sleight of hand, can deduce the MHS on the unipotent radical of the
relative completion of the fundamental group of $\M_{1,\vec{1}}$ from the MHS on
the unipotent radical of the relative completion of $\SL_2(\Z)$. The MHS on this
completion is of interest as it acts on the unipotent completion of the
fundamental group of a once punctured elliptic curve.

First recall some classical facts. (Detailed proofs can be found, for example,
in \cite{hain:china}.) The moduli space $\M_{1,\vec{1}}$ of elliptic curves with
a non-zero tangent vector at the identity is the complement of the discriminant
locus $u^3 - 27v^2=0$ in $\C^2$. For us, it is more useful to write it as the
quotient of $\C^\ast \times \h$ by the action
$$
\gamma : (\xi,\tau) \mapsto \big((c\tau+d)^{-1}\xi, \gamma\tau\big),
$$
where $\gamma = \begin{pmatrix}a & b \cr c & d \end{pmatrix}$. This action is
fixed point free, so that $\M_{1,\vec{1}}$ is an analytic variety. The
projection $\C\times \h \to \h$ induces a projection $\pi : \M_{1,\vec{1}} \to
\M_{1,1}$ that is the $\C^\ast$ bundle associated to the orbifold line bundle
$\cL \to \M_{1,1}$ with factor of automorphy $c\tau + d$. Modular forms of
$\SL_2(\Z)$ of weight $m$ are sections of $\cL^{\otimes m}$. (Cf.\
\cite[\S4]{hain:china}.) The cusp form $\D$ of $\SL_2(\Z)$ of weight 12
trivializes $\cL^{\otimes 12}$.

The $\SL_2(\Z)$ action lifts to an action of a
central extension
\begin{equation}
\label{eqn:extn}
0 \to \Z \to \Ghat \to \SL_2(\Z) \to 1
\end{equation}
on $\C\times \h$.\footnote{This action can be understood as follows: The
quotient of $\C\times\h$ by the central $\Z$ in $\Ghat$ is $\C^\ast\times \h$.
The quotient mapping is the exponential mapping on the first factor; $\SL_2(\Z)$
acts on $\C^\ast\times\h$ with factor of automorphy $(c\tau+d)^2$.} The group
$\Ghat$ is the mapping class group of a genus $1$ surface with one boundary
component.\footnote{The group $\Ghat$ is isomorphic to the $3$-string braid
group $B_3$ and also to the inverse image of $\SL_2(\Z)$ in the universal
covering group of $\SL_2(\R)$. (Cf.\ \cite[\S8]{hain:china}.)} This extension
corresponds to the orbifold $\C^\ast$-bundle $\M_{1,\vec{1}} \to \M_{1,1}$.

Denote the completion of $\Ghat$ with respect to the homomorphism
$$
\Ghat \to \SL_2(\Z) \hookrightarrow \SL_2(\Q)
$$
by $\cGhat$ and its prounipotent radical by $\Uhat$. Denote the completion of
$\SL_2(\Z)$ with respect to its inclusion into $\SL_2(\Q)$ by $\cG$ and its
prounipotent radical by $\U$. Denote the Lie algebras of $\U$ and $\Uhat$ by
$\u$ and $\uhat$, respectively. The projection $\Ghat \to \SL_2(\Z)$ induces a
homomorphism $\cGhat \to \cG$ that commutes with the projections to $\SL_2$.

\begin{proposition}
For each choice of a base point $x\in \M_{1,1}$ and each lift $\xhat$ of
$x$ to $\M_{1,\vec{1}}$, there is a natural isomorphism
$$
\cGhat_{\xhat} \cong \cG_x \times \Ga(1),
$$
where $\Ga(1)$ denotes the copy of $\Ga$ with the MHS $\Q(1)$. This induces an
isomorphism of MHS
$$
\ghat_{\xhat} \cong \g_x \oplus \Q(1).
$$
where $\ghat_{\xhat}$ is given the natural MHS constructed in
\cite{hain:malcev}.
\end{proposition}

\begin{proof}
Since the weight 12 cusp form $\D$ trivializes $\cL^{\otimes 12}$ and since
$\M_{1,\vec{1}}$ is $\cL^\ast$, there is a 12-fold covering
$$
\M_{1,\vec{1}} \to \M_{1,1}\times \C^\ast
$$
that commutes with the projections to $\M_{1,1}$. It induces an inclusion
$$
\xymatrix{
0 \ar[r] & \Z \ar[r]\ar[d]^{\times 12} & \Ghat \ar[r]\ar[d]^\phi &
\SL_2(\Z)\ar@{=}[d] \ar[r] & 1 \cr
0 \ar[r] & \Z \ar[r] & \SL_2(\Z)\times \Z \ar[r] & \SL_2(\Z) \ar[r] & 1
}
$$
of extensions. The completion of $\SL_2(\Z)\times\Z$ with respect to the obvious
homomorphism to $\SL_2(\Q)$ is $\cG\times\Ga$. This and the right exactness
(Prop.~\ref{prop:rt_exact}) of relative completion imply that the commutative
diagram
$$
\xymatrix{
 & \Ga \ar[r]\ar[d]^{\times 12} & \cGhat \ar[r]\ar[d]^{\phi_\ast} &
\cG \ar@{=}[d] \ar[r] & 1 \cr
0 \ar[r] & \Ga \ar[r] & \cG\times \Ga \ar[r] & \cG \ar[r] & 1
}
$$
has exact rows. It follows that $\phi_\ast : \cGhat \to \cG\times\Ga$ is an
isomorphism. The Hodge theoretic statements follow from the functoriality of the
MHS on relative completion.
\end{proof}

There is therefore an isomorphism $\uhat \cong \u \oplus \C\e_2$, in the
category of pronilpotent Lie algebras with an $\SL_2$ action. The new generator
$\e_2$ spans a copy of the trivial representation of $\SL_2$ and commutes with
the remaining generators
$$
\bigcup_{n>0}
\{\e_0^j\cdot \e_f',\ \e_0^j\cdot \e_f'',\ \e_0^j \cdot \e_{2n}:
0\le j \le 2n-2,\ f\in \B_{2n}\}.
$$
The Hodge type of $\e_2$ is $(-1,-1)$, which is consistent with the Hodge
types of the other $\e_{2m+2,P}$ given in Figure~\ref{fig:hodge_types}.

\begin{remark}
One can lift the power series connection $\Omega$ whose monodromy representation
$\SL_2(\Z) \to \SL_2\ltimes \U$ is the relative completion of $\SL_2(\Z)$ to a
power series connection $\Omegahat$ whose monodromy homomorphism $\Ghat \to
\SL_2\ltimes \Uhat$ is the relative completion of $\Ghat$.

The normalized Eisenstein series $G_2(\tau)$ is also defined by the series
(\ref{eqn:eisenstein}), suitably summed. Although $G_2$ is not a modular form,
it satisfies (cf.\ \cite[pp.~95-96]{serre})
$$
G_2(\gamma\tau) = (c\tau+d)^2 G_2(\tau) + ic(c\tau+d)/4\pi.
$$
This implies that
$$
\psi_2 := 2\pi i\,
G_2(\tau)\, d\tau - \frac{1}{2}\frac{d\xi}{\xi} \in E^1(\C^\ast\times \h)
$$
is $\SL_2(\Z)$-invariant, and thus a closed 1-form on
$\M_{1,\vec{1}}$.\footnote{It is useful to note that
$\psi_2=-\frac{1}{24}\frac{dD}{D}$, where, where $D=u^3-27v^2$ denotes the
discriminant function on $\M_{1,\vec{1}}$. This is because there is a unique
logarithmic 1-form on $\M_{1,\vec{1}}$ with given residue along the divisor of
nodal cubics. Cf.\ \cite[Eqn.~19.1]{hain:kzb}, where $D$ is denoted $\D$.}

If $\Omega \in K^1(\Mbar_{1,1},P;\bu)$ is a power series connection (as
constructed above), then
$$
\Omegahat = \Omega + \psi_2 \e_2
$$
is an integrable $\SL_2(\Z)$-invariant power series connection with values in
$\uhat := \u \oplus \C\e_2$. For each choice of a lift $\xhat \in \C\times\h$ of
a base point $x\in \M_{1,\vec{1}}$, the monodromy representation
$$
\Ghat \to \SL_2(\C)\ltimes \Uhat
$$
induces isomorphisms $\cGhat_x \cong \SL_2\ltimes \Uhat \cong (\SL_2\ltimes
\U)\times \Ga$.
\end{remark}

\section{The Monodromy Representation}
\label{sec:monodromy}

Let $E$ be an elliptic curve with identity $0$. Set $E'=E-\{0\}$ and let $\vv\in
T_0 E$ be a non-zero tangent vector. Denote the Lie algebra of the unipotent
completion of $\pi_1(E',\vv)$ by $\p(E,\vv)$. Recall from
Section~\ref{sec:unipt_fund} that this is a completed free Lie algebra with
abelianization $H_1(E)$.

Denote by $\bp$ the local system over $\M_{1,\vec{1}}$ whose fiber over
$[E,\vv]$ by $\p(E,\vv)$. Fix a base point $x_o=[E_o,\vv_o]$ of
$\M_{1,\vec{1}}$. Set $H_o = H_1(E_o)$ and $\p_o=\p(E_o,\vv_o)$. Denote the
completion of $\pi_1(\M_{1,\vec{1}},x_o)$ relative to the standard homomorphism
to $\SL(H_o)$ by $\cGhat_o$ and its prounipotent radical by $\Uhat_o$. Denote
their Lie algebras by $\ghat_o$ and $\uhat_o$, respectively.

The monodromy action $\pi_1(\M_{1,\vec{1}},x_o) \to \Aut \p_o$ respects the
lower central series of $\p_o$ and acts on each graded quotient through an
action of $\SL(H_o)$. The universal mapping property of relative completion
implies that the monodromy representation above induces a homomorphism
$$
\cGhat_o \to \Aut\p_o.
$$
This induces a homomorphism $\ghat_o \to \Der\p_o$ that we shall call
the {\em infinitesimal monodromy action}.

\begin{proposition}
The infinitesimal monodromy action $\ghat_o \to \Der\p_o$ is a morphism of MHS.
\end{proposition}

\begin{proof}[Sketch of Proof]
The universal punctured elliptic curve $\E' \to \M_{1,\vec{1}}$ has fiber $E'$
over $[E,\vv]\in \M_{1,\vec{1}}$. The tangent vector $\vv_o$ of $E$ at $0$ can
be regarded as a tangential base point of $\E'$. The diagram
$$
\xymatrix{
1 \ar[r] & \pi_1(E_o',\vv_o) \ar[r]\ar[d] & \pi_1(\E',\vv_o) \ar[r]\ar[d] &
\pi_1(\M_{1,\vec{1}},x_o) \ar[r]\ar[d] & 1
\cr
& 1 \ar[r] & \SL(H_o) \ar@{=}[r] & \SL(H_o) \ar[r] & 1
}
$$
gives rise to an exact sequence
$$
1 \to \cP_o \to \cG_{\E,o} \to \cGhat_o \to 1. $$ of completions that is
compatible with mixed Hodge structures. Here $\cG_{\E,o}$ denotes the completion
of $\pi_1(\E',\vv_o)$ with respect to the natural homomorphism to $\SL(H_o)$ and
$\cP_o$ the unipotent completion of $\pi_1(E_o',\vv_o)$. One has exactness on
the left as $\cP_o$ has trivial center. The conjugation action of $\cG_{\E,o}$
on $\cP_o$ induces a homomorphism $\g_{\E,o} \to \Der\p_o$ of their Lie algebras
(cf.\ \ref{ex:monod}) that is a morphism
of MHS.

The tangent vectors $\vv$ induce a section of $\pi_1(\E',\vv_o) \to
\pi_1(\M_{1,\vec{1}},x_o)$. It induces a section of $\cGhat_o \to \cG_{\E,o}$
that is compatible with mixed Hodge structures. The natural action of $\ghat_o$
on $\p_o$ is the composite $\ghat_o \to \g_{\E,o} \to \Der\p_o$ and is therefore
a morphism of MHS.
\end{proof}

Since $L^m\p_o = W_{-m}\p_o$, there is a canonical isomorphism (cf.\
(\ref{eqn:free}))
$$
\Gr^W_\dot \p_o \cong \L(H_o).
$$
of graded Lie algebras in the category of $\SL(H_o)$ modules. The map on each
graded quotient is an isomorphism of mixed Hodge structures.

The element $\sigma$ of $\pi_1(E_o',\vv_o)$ obtained by rotating the tangent
vector once around the identity is trivial in homology and thus lives in the
commutator subgroup. The image of its logarithm in
$$
\Gr^W_{-2}\p_o \cong \Lambda^2 H_o
$$
is $[\a,\b]$, where $\a,\b$ is any symplectic basis of $H_1(E)$. It spans a copy
of the trivial representation. Let
$$
\Der^0 \L(H_o) = \{\delta \in \Der\L(H_o) : \delta([\a,\b]) = 0\}.
$$
Since the natural action of $\pi_1(\M_{1,\vec{1}},x_o)$ on $\pi_1(E_o',\vv_o)$
fixes $\sigma$, we have:

\begin{corollary}
The image of the infinitesimal monodromy representation
\begin{equation}
\label{eqn:gr_monod}
\Gr^W_\dot \uhat_o \to \Der \L(H_o)
\end{equation}
lies in $\Der^0 \L(H_o)$. \qed
\end{corollary}

The Lie algebra $\Gr^W_\dot \uhat_o$ is freely generated by the image of any
$\SL(H_o)$-invariant Hodge section of $\Gr^W_\dot \uhat_o \to \Gr^W_\dot
H_1(\uhat_o)$. Since this projection is an isomorphism in weight $-1$, each
cuspidal generator $\e_f'$ and $\e_f''$ has a canonical lift to
$\Gr^W_{-1}\uhat_o$. Fix a lift $\etilde_{2n}$ of each Eisenstein generator
$\e_{2n}$ to $\Gr^W_\dot\uhat_o$.

\begin{theorem}
The image of the graded monodromy representation (\ref{eqn:gr_monod}) is
generated as an $\SL(H_o)$-module by the images of the $\etilde_{2n}$, $n\ge 1$.
\end{theorem}

\begin{proof}
First observe that $\Gr^W_{-1}\Der^0 \L(H_o) = 0$. This is because
$$
\Gr^W_{-1} \L(H_o) = H_o \text{ and } \Gr^W_{-2}\L(H_o) = \Q[\a,\b].
$$
The element $u \in H_o$ corresponds to the derivation $\ad_u$. Since
$\ad_u([\a,\b]) = [u,[\a,\b]]\neq 0$ for all non-zero $u\in H_o$,
$\Gr^W_{-1}\Der^0\L(H)=0$.

Since each $\e_f'$ and $\e_f''$ in $\Gr^W_{-1}\uhat_o$ has weight $-1$, this
vanishing implies that $\e_0^j\cdot \e_f'$ and $\e_0^j\cdot \e_f''$ are in the
kernel of the graded monodromy representation. It follows that the image is
generated by the images of the remaining generators --- the Eisenstein
generators $\e_0^j \cdot \etilde_{2n}$.
\end{proof}

The next task is to identify the images of the $\etilde_{2n}$ in
$\Der^0\L(H_o)$. For each $n\ge 0$ a basis $\v_1,\v_2$ of $H$ define derivations
$\epsilon_{2n}(\v_1,\v_2)$ by
\begin{equation}
%\label{eqn:deltas}
\epsilon_{2n}(\v_1,\v_2) :=
\begin{cases}
-\v_2\frac{\partial}{\partial \v_1} & n = 0; \cr
\ad_{\v_1}^{2n-1}(\v_2)-
\sum_{\substack{j+k=2n-1\cr j>k > 0}}
(-1)^j[\ad_{\v_1}^j(\v_2),\ad_{\v_1}^k(\v_2)]
\frac{\partial}{\partial \v_2} & n > 0.
\end{cases}
\end{equation}
Here we are identifying $\L(H)$ with its image in $\Der\L(H)$ under the
inclusion $\ad : \L(H) \hookrightarrow \Der \L(H)$.

The following result implies that the image of $\etilde_{2n}$ in $\Der^0\L(H)$
depends only on $\e_{2n}$ and not on the choice of the lift $\etilde_{2n}$.

\begin{proposition}[Hain-Matsumoto]
For each $n\ge 1$ there is a unique copy of $S^{2n}H(2n+1)$ in
$\Gr^W_{-2n-2}\Der^0\L(H)$. It has highest weight vector the derivation
$\epsilon_{2n}(\v_1,\v_2)$, where $\v_1,\v_2 \in H$ are non-zero vectors of
$\sl_2$-weight $1$ and $-1$, respectively. \qed
\end{proposition}

It follows that the image of $\etilde_{2n}$ in $\Der\L(H_o)$ is a multiple
(possibly zero) of $\epsilon_{2n}(\b,\a)$. We compute this multiple using the
universal elliptic KZB-connection \cite{cee,levin-racinet,hain:kzb}, which
provides an explicit formula for the connection on the bundle $\bp$ over
$\M_{1,\vec{1}}$.

\begin{theorem}
\label{thm:images}
For all choices of the lift $\etilde_{2n}$, the image of $\etilde_{2n}^B = 2\pi
i\,\etilde_{2n}$ under the graded monodromy representation
(\ref{eqn:gr_monod}) is $2\epsilon_{2n}(\b,\a)/(2n-2)!$ when $n>0$ and
$\epsilon_0$ when $n=0$.
\end{theorem}

\begin{proof}
Trivialize the pullback of $\cH$ to $\C^\ast\times \h$ by the sections
$$
T := \tau\a-\b = \exp(\tau\e_0)(-\b) \text{ and } A := (2\pi i)^{-1}\a.
$$
In \cite[\S13--14]{hain:kzb} it is shown that the pullback of $\bp$ to
$\C^\ast\times \h$ may be identified with the trivial bundle
$$
\L(T,A)^\wedge \times \C^\ast \times \h \to \C^\ast \times \h
$$
with the connection $\nabla = d + \w'$, where
$$
\w' = - 2\pi i\Big(d\tau\otimes \epsilon_0(T,A) + \sum_{m\ge 1}
\frac{2}{(2n-2)!}G_{2n}(\tau)d\tau \otimes \epsilon_{2n}(T,A)\Big).
$$

To prove the result, we need to rewrite this in terms of the frame $-\b,\a$ of
$\cH$. First note that $\epsilon_{2n}(c_1\v_1,c_2\v_2) =
c_1^{2n-1}c_2\epsilon_{2n}(\v_1,\v_2)$ and that if $g\in \SL(H)$, then
$\epsilon_{2n}(g\v_1,g\v_2) = g\cdot\epsilon_{2n}(\v_1,\v_2)$, where $g\in
\SL(H)$ acts on a derivation $\delta$ by $g\cdot \delta := g\delta g^{-1}$.
Since $\e_0\cdot \a = 0$, these imply that
$$
2\pi i\, \epsilon_{2n}(T,A) =
\epsilon_{2n}(T,\a) =
\epsilon_{2n}\big(\exp(\tau\e_0)(-\b),\exp(\tau\e_0)\a\big)
= -\exp(\tau\e_0)\cdot \epsilon_{2n}(\b,\a).
$$
It follows that $2\pi i\, G_{2n}(\tau) d\tau\otimes \epsilon_{2n}(T,A) = -
\psi_{2n}\big(\epsilon_{2n}(\b,\a)\big)/2\pi i$.

Since the natural connection $\nabla_0$ on $\cH$ is given by
$$
\nabla_0 = d - 2\pi i\, \epsilon_0(T,A)\otimes d\tau,
$$
the pullback connection may be written
\begin{align*}
\nabla &= \nabla_0 - 2\pi i\sum_{m\ge 1}
\frac{2}{(2n-2)!}G_{2n}(\tau)d\tau \otimes \epsilon_{2n}(T,A) \cr
& = \nabla_0 + \sum_{m\ge 1}
\frac{2}{(2n-2)!}\psi_{2n}\big(\epsilon_{2m}(\b,\a)/2\pi i\big).
\end{align*}
It follows that, regardless of the choice of the lifts of the $\etilde_{2n}$, 
the de~Rham generator $\etilde_{2n}$ goes to $2\epsilon_{2n}(\b,\a)/2\pi
i(2n-2)!$ under the graded monodromy representation. Since $\e_{2n}$ spans a
copy of $\Q(1)$, the Betti generator is $\e_{2n}^B = \e_{2n}^\DR/2\pi i$.
\end{proof}

%% what is confusing (and bad) here is that e_0 is e_0^B, while e_{2n} =
%% e_{2n}^\DR when $n>0$. Fix in later papers.

\begin{remark}
Since $\ghat = \g \oplus \Q(1)$, there are natural representations
$$
\g \to \Der\p \text{ and } \Gr^W_\dot \g \to \Der^0\L(H).
$$
There are also the outer actions
$$
\g \to \OutDer\p \text{ and } \Gr^W_\dot \g \to \OutDer\L(H).
$$
The representations $\Gr^W_\dot \g \to \Der\L(H)$ and $\Gr^W_\dot \g\to
\OutDer\L(H)$ have the same kernel as $\e_2 \notin \g$ and as
$$
\InnDer \L(H) \cap \Der^0\L(H) = \Q\epsilon_2(\b,\a).
$$
Exactness of $\Gr^W_\dot$ implies that $\g \to \Der\p$ and $\g\to \OutDer\p$
have the same kernel. Since it is generally easier to work with derivations than
with outer derivations, we will work with $\g\to \Der\p$.
\end{remark}

\section{The Eisenstein Quotient of a Completed Modular Group}
\label{sec:eisen_quot}

The results of the previous section imply that for all $x\in \M_{1,1}$, each
weight graded quotient of the image of $\g_x$ in $\Der\p_x$ is a sum of Tate
twists $S^m H_x(r)$ of symmetric powers of $H_x$. Any such Lie algebra quotient
of $\g_x$ has the property that its weight associated graded is generated by the
images of the Eisenstein generators $\e_{2n}$.

Suppose that $\G$ is a finite index subgroup of $\SL_2(\Z)$. Denote the
completion of $\pi_1(X_\G,x)$ with respect to the inclusion $\G \to \SL_2(H_x)$
by $\cG_x$ and its pronilpotent radical by $\U_x$. Denote their Lie algebras by
$\g_x$ and $\u_x$, respectively.

\begin{proposition}
\label{prop:eis}
For each $x\in X_\G$ there is a unique maximal quotient $\g^\eis_x$ of $\g_x$ in
the category of Lie algebras with a mixed Hodge structure with the property that
each weight graded quotient of $\g_x^\eis$ is a sum of Tate twists of symmetric
powers of $H_x$. Moreover, the Lie algebra isomorphism $\g_x \to \g_y$
corresponding to a path from $x$ to $y$ in $X_\G$ induces a Lie algebra (but not
a Hodge) isomorphism $\g_x^\eis \to \g_y^\eis$. The corresponding local system
$\bg^\eis := (\g_x^\eis)_{x\in X_\G}$ underlies an admissible VMHS.
\end{proposition}

The quotient $\g_x^\eis$ will be called the {\em Eisenstein quotient} of $\g_x$.
The corresponding quotient of $\cG_x$ will be denoted by $\cG_x^\eis$.

\begin{remark}
Note that if we instead use a tangential base point $\vv$, then $H_\vv$ is an
extension of $\Q$ by $\Q(1)$, so that $\g_\vv^\eis$ is a mixed Hodge-Tate
structure. (That is, all of its $M_\dot$ weight graded quotients are of type
$(p,p)$.) In this case, $\g^\eis_\vv$ is the ``maximal Tate quotient'' of
$\g_\vv$ in the category pro-Lie algebras with MHS.
\end{remark}

\begin{corollary}
When $\G=\SL_2(\Z)$, the monodromy homomorphism $\g_x \to \Der\p_x$ factors
through $\g_x^\eis$:
$$
\xymatrix{
\g_x \ar[r]\ar@/^1pc/[rr] & \g^\eis_x \ar[r] & \Der\p_x
}
$$
\qed
\end{corollary}

Note that if $\G$ has finite index in $\SL_2(\Z)$, then the Eisenstein quotient
of its relative completion surjects onto the Eisenstein quotient of the relative
completion of $\SL_2(\Z)$.

\begin{proof}[Proof of Proposition~\ref{prop:eis}]
If $\h_1$ and $\h_2$ are quotients of $\g_x^\eis$ whose weight graded quotients
are sums of Tate twists of symmetric powers of $H_x$, then the image $\h$ of
$\g_x^\eis \to \h_1\oplus \h_2$ is a quotient of $\g_x^\eis$ whose weight graded
quotients are sums of twists of symmetric powers of $H_x$. It also surjects onto
$\h_1$ and $\h_2$. This implies that the ``Eisenstein quotients'' of $\g_x$ form
an inverse system from which it follows that the Eisenstein quotient is unique.
Note that since $\sl(H_x) \cong S^2 H_x$, $\g_x$ surjects onto $\sl(H_x)$.

We begin the proof of the second part with an observation. Suppose that $V$ is a
MHS and that $K$ is a subspace of $V$ that is defined over $\Q$. Give it the
induced weight filtration. Then there is a natural isomorphism
$$
\Gr^W_m (V/K) \cong (\Gr^W_m V)/(\Gr^W_m K).
$$
This isomorphism respects the Hodge filtration on each induced from the Hodge
filtration $F^\dot\cap (W_m V)$ of $W_m V$. These are defined as the images of
the maps
\begin{multline*}
F^p W_m V \to (W_mV)/(W_mK) \to \Gr^W_m (V/K) \cr \text{ and }
F^p W_m V \to \Gr^W_m V \to (\Gr^W_m V)/(\Gr^W_m K).
\end{multline*}
The observation is that if this Hodge filtration defines a Hodge structure on
each $\Gr^W_m (V/K)$, then $K$ is a sub MHS of $V$ and $V/K$ is a quotient MHS
of $V$.

Now apply this with $V = \g_y$ and $K$ the ideal of $\g_y$ that corresponds to
the kernel of $\g_x \to \g_x^\eis$ under the isomorphism $\g_x \cong \g_y$ given
by parallel transport. This implies that $\g_y/K$ is a quotient of $\g_y$ in the
category of MHS whose weight graded quotients are sums of twists of symmetric
powers of $H_y$. It is  therefore an Eisenstein quotient of $\g_y$.

The last statement follows from the fact that if $\W$ is a quotient of the local
system underlying an admissible variation of MHS $\V$ with the property that
each fiber of $W_x$ has a MHS that is the quotient of the MHS on $V_x$, then
$\W$ is an admissible variation of MHS.
\end{proof}

The significance of $\g^\eis$ lies in the following result, which follows
directly from Theorem~\ref{thm:equivalence}.

\begin{corollary}
For all base points (finite, tangential), the category $\HRep(\g_x^\eis)$ of
Hodge representations (Def.~\ref{def:hrep}) of $\g_x^\eis$ is equivalent to the
category of admissible VMHS over $X_\G$ whose weight graded quotients are sums
of Tate twists $S^n\H(r)$ of symmetric powers of $\H$. \qed
\end{corollary}

\begin{remark}
We will call such variations over a modular curve {\em Eisenstein variations of
MHS}. Like Tate VMHS in the unipotent case (cf.\ \cite{hain-zucker2}),
Eisenstein variations over a modular curve can be written down reasonably
explicitly. This follows from the constructions of Section~\ref{sec:avmhs}. The
explanation we give below is somewhat technical. Suppose that $\A$ is an
Eisenstein variation over $X_\G$. Set
$$
A_{m,n}= H^0\big(X_\G,\Hom(S^n\H,\Gr^W_m \A)\big).
$$
This is a Tate Hodge structure of weight $m-n$. There is a natural isomorphism
$$
\Gr^W_m \A \cong \bigoplus_n A_{m,n}\otimes S^n\H
$$
of MHS. For each $r$ satisfying $0\le r \le \min(n,\ell)$, fix a highest weight
vector $\bh_{n,\ell}^{(r)}$ of $\sl_2$-weight $n+\ell-2r$ in $\Hom(S^n H,S^\ell
H)$, so that
$$
\Hom(S^n\H,S^\ell \H) \cong
\bigoplus_{r=0}^{\min(n,\ell)} \bS^{n+\ell-2r}(\bh_{n,\ell}^{(r)}).
$$
Implicit here is that $\bh_{n,\ell}^{(r)}$ has Hodge weight $\ell-n$. Set $\cA =
\A\otimes \cO_{X_\G}$. Denote the natural connection on it by $\nabla$. Set
$$
\cA_m = (\Gr^W_m \A)\otimes \cO_{X_\G} \cong \bigoplus_n A_{m,n}\otimes S^n\cH.
$$
The standard connection on $\cH$ induces a connection on each of these that we
denote by $\nabla_0$. The construction of Section~\ref{sec:avmhs} implies that
for each cusp $P\in D$ of $X_\G$, there are linear maps
\begin{equation}
\label{eqn:condn}
\varphi_{k,\ell,P}^{m,n} \in
F^{\ell-r+1}W_{2\ell-2r+2}\Hom_\C(A_{m,n},A_{k,\ell})
\end{equation}
such that
\begin{equation}
\label{eqn:A-isom}
(\cA,\nabla) \cong \big(\bigoplus_m \cA_m,\nabla_0 + \Omega\big)
\end{equation}
where
$$
\Omega =
\sum_{P\in D}\sum_{n,\ell\ge 0}\sum_{r=0}^{\min(n,\ell)}\sum_{m,k}
\varphi^{m,n}_{k,\ell,P} \otimes \psi_{n+\ell-2r+2,P}(\bh_{n,\ell}^{(r)}).
$$
Implicit in this statement is that the canonical extension $\overline \cA$ of
$(\cA,\nabla)$ is isomorphic to
$$
\bigoplus_{m,n} A_{m,n}\otimes (S^n\Hbar,\nabla_0).
$$
The isomorphism (\ref{eqn:A-isom}) is bifiltered with respect to the Hodge and
weight filtrations. The condition (\ref{eqn:condn}) implies that
$
\Omega\in
F^0W_{-1} K^1\big(\Xbar_\G,D;\End(\Gr^W_\dot\A)\big).
$

Caution: not every such 1-form $\Omega$ defines the structure of an admissible
variation of MHS over $X_\G$. The issue is that one needs the monodromy of
$(\cA,\nabla_0 + \Omega)$ to be defined over $\Q$. Determining which $\Omega$
give rise to Eisenstein variations is closely related to the problem of
determining the relations in $\u^\eis$.

\end{remark}

\section{Modular Symbols and Pollack's Quadratic Relations}
\label{sec:pollack_relns}

Motivic arguments (cf.\ \cite{hain-matsumoto:mem}) suggested that $\u_x^\eis$ may
not be freely generated by the $\e_0^j\cdot \e_{2m}$ and predict that the
relations that hold between the $\e_0^j\cdot \e_m$ arise from cusp forms. In
other words, cusp forms go from being generators of $\u_x$ to relations in
$\u_x^\eis$. The goal of the rest of this paper is to sketch a Hodge theoretic
explanation for these relations. For this we will need to recall some basic
facts about modular symbols, which record the periods of cusp forms and
determine the Hodge structure on $H^1_\cusp(\M_{1,1}, S^{2n}\H)$. In this and
subsequent sections, $\B_w = \B_w(\SL_2(\Z))$, the set of normalized Hecke eigen
cusp forms of $\SL_2(\Z)$.

\subsection{Modular Symbols}

Modular symbols are homogeneous polynomials attached to cusp forms of
$\SL_2(\Z)$. They play two roles: they give a concrete representation of the
cohomology class associated to a cusp form; secondly, modular symbols of degree
$m$ record the periods of the MHS on $H^1(\M_{1,1},S^m\H) \cong
H^1(\SL_2(\Z),S^m H)$. A standard reference is \cite[Ch.~IV]{lang}.

Recall that $\SL_2(\Z)$ has presentation
$$
\SL_2(\Z) = \langle S, U : S^2 = U^3,\ S^4 = U^6 = I\rangle.
$$
where
$$
S = \begin{pmatrix} 0 & -1 \cr 1 & 0 \end{pmatrix},\quad
T = \begin{pmatrix} 1 & 1 \cr 0 & 1 \end{pmatrix}, \quad
U := ST = \begin{pmatrix} 0 & -1 \cr 1 & 1 \end{pmatrix}
$$
The action of $\SL_2(\Z)$ on $\h$ factors through
$$
\PSL_2(\Z) := \SL_2(\Z)/(\pm I) = \langle S, U : S^2 = U^3 = I\rangle.
$$

\subsubsection{Group cohomology}

Suppose that $\G$ is a group and that $V$ is a left $\G$-module. Then one has
the complex
$$
\xymatrix{
0 \ar[r] & C^0(\G,V) \ar[r]^\delta & C^1(\G,V) \ar[r]^\delta &
C^2(\G,V) \ar[r]^(.6)\delta & \cdots
}
$$
of standard cochains, where $C^j(\G,V) = \{\text{functions } \phi : \G^j \to
V\}$. The differential takes $v\in V = C^0(\G,V)$ to the function $\delta v :
\gamma \mapsto (\gamma-1)v$ and $\delta : C^0(\G,V) \to C^1(\G,V)$ takes $\phi :
\G \to V$ to the function
$$
(\delta \phi)(\gamma_1,\gamma_2) \mapsto
\phi(\gamma_1) - \phi(\gamma_1\gamma_2) + \gamma_1\cdot \phi(\gamma_2).
$$
So $\phi$ is a 1-cocycle if and only if $\phi(\gamma_1\gamma_2) =
\phi(\gamma_1)+\gamma_1\cdot \phi(\gamma_2)$. The cohomology $H^\dot(\G,V)$ of
$\G$ with coefficients in $V$ is defined to be the homology of this complex.

Now suppose that $V$ is a real or complex vector space and that $\G$ acts on a
simply connected manifold $X$. Fix a base point $x_o\in X$. As in
Proposition~\ref{prop:monodromy}, to each $\gamma\in G$ we can associate the
unique homotopy class $c_\gamma$ of paths in $X$ from $x_o$ to $\gamma\cdot
x_o$. If $\w\in E^1(X)\otimes V$ is $\G$ invariant, then the function
$$
\phi: \gamma \mapsto \int_{c_\gamma} \w
$$
is a 1-cocycle. Changing the base point from $x_o$ to $x'$ changes $\phi$ by the
coboundary of $\int_{x_o}^{x'}\w \in V$. (Cf.\ Remark~\ref{rem:cobound}.) This
construction induces a map
$$
H^1(E^1(X\times V)^\G) \to H^1(\G,V),
$$
which is an isomorphism when $\G$ acts properly discontinuously and virtually
freely on $X$. This is the case when $\G$ is a modular group and $X$ is the
upper half plane.

Suppose that $V$ is divisible as an abelian group. When $-I$ acts trivially on
$V$, $V$ is the pullback of a $\PSL_2(\Z)$-module and
$$
H^1(\PSL_2(\Z),V) \to H^1(\SL_2(\Z),V)
$$
is an isomorphism.

\subsubsection{Cuspidal cohomology}

Suppose that $F$ is a field of characteristic zero. Set $H_F = F\a \oplus F\b$.
Define $C^\dot_\cusp(\SL_2(\Z),S^{2n}H_F)$ to be the kernel of the restriction
mapping
$$
C^\dot(\SL_2(\Z),S^{2n}H_F) \to \widetilde{C}^\dot(\langle T \rangle, S^{2n}H_F)
$$
where the right hand complex is the quotient of $C^\dot(\langle T \rangle,
S^{2n}H_F)$ by $\a^{2n}$ in degree 0. Set
$$
H^\dot_\cusp(\SL_2(\Z),S^{2n}H_F) := H^\dot(C^\dot_\cusp(\SL_2(\Z),S^{2n}H_F)).
$$
The corresponding long exact sequence gives the exact sequence
\begin{multline*}
0 \to H^1_\cusp(\SL_2(\Z),S^{2n}H_F) \to H^1(\SL_2(\Z),S^{2n}H_F) \cr
\to H^1(\langle T \rangle,S^{2n}H_F) \to H^2_\cusp(\SL_2(\Z),S^{2n}H_F) \to 0.
\end{multline*}
This is an instance of the exact sequences (\ref{eqn:exact_seqce}) and
(\ref{eqn:exact_seqce:mod}) where $C'=\M_{1,1}$.

The cuspidal cohomology group $H_\cusp^1(\SL_2(\Z),S^{2n}H_F)$ has a nice
description. Recall that $\SL_2(\Z)$ acts on $H$ via the formula
(\ref{eqn:formula}) with $\v_1=-\b$ and $\v_2 = \a$.

\begin{proposition}
Suppose that $F$ is a field of characteristic zero. For all $n\ge 1$, there is
an isomorphism
$$
H^1_\cusp(\SL_2(\Z),S^{2n}H_F)
$$
with the vector space of $\br(\a,\b) \in S^{2n}H_F$ that satisfy
\begin{equation}
\label{eqn:cocycle}
(I + S)\br(\a,\b) = 0 \text{ and } (I + U + U^2)\br(\a,\b) = 0
\end{equation}
modulo $\a^{2n}-\b^{2n}$.
\end{proposition}

\begin{proof}
Suppose that $\phi : \SL_2(\Z) \to S^{2n}H_F$ is a cuspidal 1-cocycle. Since
$-I$ acts trivially on $S^{2n}H$, the cocycle condition  and the equation
$(-I)^2 = 1$ imply that $\phi(-I)= 0$.  Since $S$ and $U$ generate $\SL_2(\Z)$,
$\phi$ is determined by $\phi(S)$. Since $U=ST$,
$$
\phi(U) = \phi(ST) = \phi(S) + S\phi(T) = \phi(S).
$$
Thus $\phi$ is determined by $\phi(S)$. Denote this element of $S^{2n}H_F$ by
$\br_\phi(\a,\b)$. Since $S^2=U^3=I$, the cocycle condition implies that
$\br_f(\a,\b)$ satisfies the equation (\ref{eqn:cocycle}). Conversely, if
$\br(\a,\b)\in S^{2n}H_F$ satisfies these equations, it determines a
well-defined cuspidal cocycle $\phi$ by $\phi(S)=\phi(U)=\br(\a,\b)$.

The last statement follows as the only cuspidal coboundaries are scalar
multiplies of $\delta\a^{2n}$. This has value $\br(\a,\b) = \b^{2n}-\a^{2n}$ on
$S$.
\end{proof}

\begin{remark}
Since $(I+S)\sum_j c_j \a^j \b^{2n-j}$ vanishes if and only if
$c_{2n-j}=(-1)^{j+1}c_j$ for all $j$, the terms of a cocycle $r(\a,\b)$ of top
degree in $\a$ and $\b$ is a multiple of $\a^{2n}-\b^{2n}$. Since this
corresponds to the coboundary of $\a^{2n}$, we can identify
$H^1_\cusp(\SL_2(\Z),S^{2n}H_F)$ with those $\br(\a,\b)$ that satisfy the
cocycle conditions (\ref{eqn:cocycle}) and have no terms of degree $2n$ in $\a$
or $\b$.
\end{remark}

\subsubsection{Modular symbols}
\label{sec:mod_symbs}

If $f$ is a cusp form of weight $2n+2$, the $S^{2n}H$-valued 1-form
$\w_f(\b^{2n})$ is $\SL_2(\Z)$-invariant. Since $f$ is a cusp form, it is
holomorphic on the $q$-disk. We can therefore take the base point $x_o$ above to
be the cusp $q=0$. Since $T$ fixes $q=0$, the function
$$
\gamma \mapsto
\int_{x_o}^{\gamma x_o} \w_f(\bw^{2n}) = 
(2\pi i)^{2n}\int_{x_o}^{\gamma x_o} \w_f(\b^{2n}) \in S^{2n}H_\C
$$
is a well defined cuspidal 1-cocycle. The {\em modular symbol} of $f$ is its
value\footnote{We use this normalization because, if $f\in \B_{2n+2}$, then
$f(q)\bw^{2n} dq/q \in H_\DR^1(\M_{1,1/\Q},S^{2n}\cH)$.
Cf.\ \cite[\S 21]{hain:kzb}.}
$$
\br_f(\a,\b) := \int_0^1 f(q)\bw^{2n}\frac{dq}{q} =
-(2\pi i)^{2n+1}\int_0^\infty f(iy)(\b-iy\a)^{2n}d(iy) \in S^{2n}H_\C
$$
on $S$. It satisfies the cocycle condition (\ref{eqn:cocycle}) and represents
the class
$$
(2\pi i)^{2n}\w_f(\b^{2n}) \in H^1_\cusp(\M_{1,1},S^{2n}\H).
$$
It determines $f$.

\subsection{Hodge theory}
\label{sec:hodge}

The Hodge structure
$$
H^1_\cusp(\M_{1,1},S^{2n}\H) = H^{2n+1,0} \oplus H^{0,2n+1}.
$$
has a description in terms of modular symbols. The underlying $\Q$ vector space
is the set of $\br(\a,\b)\in S^{2n}H_\Q$ that satisfy (\ref{eqn:cocycle}),
modulo $\a^{2n}-\b^{2n}$; the Hodge filtration is given by
$$
H^{2n+1,0} = F^{2n+1}H^1_\cusp(\M_{1,1},S^{2n}\H)
= \{\br_f(\a,\b) : f\in M^o_{2n+2}\}/\C(\a^{2n}-\b^{2n}).
$$

There is more structure. Each element $r(\a,\b)$ of $S^{2n}H$ can be written
in the form
$$
\br(\a,\b) = \br^+(\a,\b) + \br^-(\a,\b),
$$
where $\br^+(\a,\b)$ is the sum of the terms involving only even powers of $\a$
and $\b$ and $\br^-(\a,\b)$ is the sum of the terms involving only odd powers.

If $f$ has real Fourier coefficients (e.g., $f\in \B_{2n+2}$), then
\begin{equation}
\label{eqn:pm}
\br_f(\a,\b) = \br_f^+(\a,\b) + i\br_f^-(\a,\b) 
\end{equation}
where $\br_f^\pm(\a,\b)$ are real. Since the cocycle corresponding to $\bar{f}$
is $\br_f^+(\a,\b) - i\br_f^-(\a,\b)$, $\br_f^+(\a,\b),\br_f^-(\a,\b) \in
S^{2n}H_\R$ also satisfy the cocycle condition. Since the classes
$\w_f(\b^{2n})$, $f\in \B_{2n+2}$ span the cuspidal cohomology, we deduce:

\begin{proposition}
If $\br(\a,\b) \in S^{2n}H_F$ satisfies the cocycle condition
(\ref{eqn:cocycle}), then so do $\br^+(\a,\b)$ and $\br^-(\a,\b)$. \qed
\end{proposition}

This gives a decomposition $V_F = V_F^+ \oplus V_F^-$ of $V_F :=
H^1_\cusp(\M_{1,1},S^{2n}\H_F)$. Since
$$
H^{2n+1,0}_\cusp(\M_{1,1},S^{2n}\H) \cap H^1(\M_{1,1},S^{2n}\H_\R) = 0
$$
both parts $\br_f^\pm(\a,\b)$ of the modular symbol of a cusp form with  real
Fourier coefficients are non-zero. In particular, for each $f\in \B_{2n+2}$ we
can write
$$
V_{f,\R} = V_{f,\R}^+ \oplus V_{f,\R}^-
:= \R\,\br_f^+(\a,\b) \oplus \R\,\br_f^-(\a,\b).
$$
Note that $V_f^+$ and $V_f^-$ are real sub Hodge structures of $V_f$. 

\subsubsection{The action of real Frobenius}
\label{sec:frob}

Complex conjugation (aka ``real Frobenius'') $\Fr_\infty \in \Gal(\C/\R)$ acts
on $\M_{1,1}$ and on the local system $\H_\R$ as we shall explain below. It
therefore acts on $V_\R = H^1_\cusp(\M_{1,1},S^{2n}\H_\R)$. In this section we
show that its eigenspaces are $V_\R^\pm$.

The stack $\M_{1,1/\C}$ has a natural real (even $\Q$) form, viz, $\Gm\bbs
(\A^2_\R - D)$, where $D$ is the discriminant locus $u^3-27v^2=0$ and where
$t\cdot(u,v) = (t^4u,t^6v)$. The involution $\Fr_\infty: (u,v)\mapsto
(\bar{u},\bar{v})$ of $\M_{1,1}$ is covered by the involution of $\tau \mapsto
-\taubar$ of $\h$.

The universal curve $\E$ over it also has a natural real form as it is defined
over $\R$. The projection $\E \to \M_{1,1}$ is invariant under complex
conjugation. This implies that $\Fr_\infty$ acts on $\H_\R$. This action is
determined by the action of $\Fr_\infty$ on $H$, the fiber over the tangent
vector $\partial/\partial q$, which is real and therefore fixed by $\Fr_\infty$.
This induced map is easily seen to be the involution $\sigma : H \to H$ defined
by\footnote{Here $H$ is identified with $H^1(E_{\partial/\partial q})$ which is
isomorphic to $H_1(E_{\partial/\partial q})(-1)$. So the actions of $\Fr_\infty$
on $H_1(E_{\partial/\partial q})$ and $H^1(E_{\partial/\partial q})$ differ by
$-1$.}
\begin{equation}
\label{eqn:conjugation}
\b \mapsto -\b,\ \a \mapsto \a.
\end{equation}
The monodromy representation $\SL_2(\Z) \to \Aut H$ of $\Fr^\ast_\infty \H$ is
the standard representation conjugated by $\sigma$. There is therefore a natural
action
$$
\Fr_\infty : H^1_\cusp(\M_{1,1},S^{2n}\H_\R)
\to H^1_\cusp(\M_{1,1},S^{2n}\H_\R).
$$
Let $\Frbar_\infty : H^1_\cusp(\M_{1,1},S^{2n}\H_\C) \to
H^1_\cusp(\M_{1,1},S^{2n}\H_\C)$ be its composition with complex conjugation
$H^1_\cusp(\M_{1,1},S^{2n}\H_\C) \to H^1_\cusp(\M_{1,1},S^{2n}\H_\C)$. This is
the de~Rham involution.

\begin{lemma}
\label{lem:frob}
Real Frobenius $\Fr_\infty$ acts on $V_\R = H^1_\cusp(\M_{1,1},S^{2n}\H_\R)$ by
multiplication by $+1$ on $V_\R^+$ and $-1$ on $V_\R^-$.
\end{lemma}

\begin{proof}
The result follows from equations (\ref{eqn:pm}) and (\ref{eqn:conjugation}) and
the fact that the image of  $F^{2n+1}H^1_\DR(\M_{1,1/\R},S^{2n}\cH)$ in
$H^1_\cusp(\M_{1,1},S^{2n}\H_\C)$ is invariant under $\Frbar_\infty$.
\end{proof}

\subsubsection{Extensions of MHS associated to cusp forms}

Since $V_{f,\C} = V_{f,\R}\oplus F^rV_f$ when $r<2n+2$ and $F^0 V_f(r) = 0$ when
$r\ge 2n+2$, we have
\begin{equation}
\label{eqn:ext_hodge}
\Ext_\MHS^1\big(\R,V_{f}(r)\big) \cong
\begin{cases}
0 &  r < 2n+2, \cr
iV_{f,\R} & r \ge 2n+2\text{ even,}\cr
V_{f,\R} & r \ge 2n+2 \text{ odd.}
\end{cases}
\end{equation}
We can now compute the extensions invariant under the de~Rham involution.

\begin{proposition}
\label{prop:ext_hodge}
If $f \in\B_{2n+2}$, then
$$
\Ext_\MHS^1\big(\R,V_{f}(r)\big)^{\overline{\Fr}_\infty} \cong
\begin{cases}
0 &  r < 2n+2, \cr
iV_{f,\R}^- & r \ge 2n+2\text{ even,}\cr
V_{f,\R}^+ & r \ge 2n+2 \text{ odd.}
\end{cases}
$$
In particular, it is one dimensional for all $r\ge 2g+2$.
\end{proposition}

\begin{proof}
Since $\Fr_\infty : V_{f,\C} \to V_{f,\C}$ is $\C$-linear and since twisting by
$\R(r)$ multiplies this action by $(-1)^r$, Lemma~\ref{lem:frob} implies that
$\Frbar_\infty$ acts on $iV_\R^+(r)$ by multiplication by $(-1)^{r+1}$ and on
$iV_\R^-$ by $(-1)^{r}$. The result follows from the
$\Frbar_\infty$-equivariant isomorphism $\Ext^1_\MHS(\R,V_f(r)) \cong i
V_{f,\R}(r)$ which holds for all $r\ge 2n+2$.
\end{proof}

So each normalized Hecke eigen cusp form $f\in \B_{2n+2}$ determines an element
of $\Ext_\MHS^1\big(\R,V_{f}(r)\big)^{\overline{\Fr}_\infty}$ for each $r\ge
2n+2$. Namely, the extension corresponding to $i\br_f^+(\a,\b) \in iV_{f,\R}^+$
when $r$ is even, and $\br_f^-(\a,\b)\in V_{f,\R}^-$ when $r$ is odd.

\subsection{Pollack's quadratic relations}
\label{sec:pollack}

Motivic considerations \cite{hain-matsumoto:mem} suggest that $\u_x^\eis$ is not
free and that a minimal set of relations that hold between the
$\e_0^j\cdot\e_{2n}$ are parametrized by cusp forms. In fact, each cusp form
should determine relations between the $\e_0^j\cdot\e_{2n}$ of every degree $\ge
2$. (For this purpose, $\e_0$ is considered to have degree 1.) One way to guess
such relations is to find relations that hold between their images $\e_0^j\cdot
\epsilon_{2n}(\b,\a)$ in $\Der\Gr^W\p \cong \Der\L(H)$.\footnote{Since
$\Gr^W_\dot$ is exact, one neither gains nor loses relations in the associated
graded.} For convenience, set $\epsilon_{2n} = \epsilon_{2n}(\b,\a)$.

In his undergraduate thesis \cite{pollack}, Pollack found a complete set of
quadratic relations that hold between the $\epsilon_{2n}$, and found relations
of all degrees $>2$ that hold between the $\e_0^j\cdot \epsilon_{2n}$ modulo a
certain filtration of $\Der\L(H)$. Here we state his quadratic relations.

\begin{theorem}[{Pollack \cite[Thm.~2]{pollack}}]
\label{thm:pollack}
The relation
$$
\sum_{
\substack{j+k=n\cr j,k>0}}
c_{j} [\epsilon_{2j+2},\epsilon_{2k+2}] = 0 
$$
holds in $\Der\L(H)$ if and only if there is a cusp form $f$ of weight $2n+2$
whose modular symbol is
$$
r_f^+(\a,\b) = \sum_{\substack{j+k=n\cr j,k\ge 0}} c_j \a^{2j}\b^{2n-2j}.
$$
\end{theorem}

In view of Theorem~\ref{thm:images}, this suggests each cusp form of weight
$2n+2$ with $r_f^+(\a,\b) = \sum c_j \a^{2j}\b^{2n-2j}$ might determine a
relation
\begin{equation}
\label{eqn:ext_map}
\sum_{\substack{j+k=n\cr j,k>0}} c_{j}(2j)!(2k)!\, [\e_{2j+2},\e_{2k+2}] = 0
\end{equation}
in $\Gr^W_\dot\u_x^\eis$ and that these relations are connected with
$\Ext^1_\MHS(\R,V_f(2r))$ for appropriate $r>0$. In the remaining sections, we
show that Pollack's relations do indeed lift to relations in $\Gr^W_\dot
\u^\eis_x$ and explain the connection to extensions of MHS related to cusp
forms. In preparation for proving this, we restate Pollack's result in
cohomological terms.

We let the base point be $\vv = \partial/\partial q$, although any base point
will do. Denote the image of the monodromy homomorphism $\cG \to \Aut\p$ by
$\cD$ and the Lie algebra of its pronilpotent radical by $\n$. Since the
monodromy homomorphism is a morphism of MHS, $\cO(\cD)$ and $\n$ have natural
MHSs. This implies that $H^\dot(\cD,S^m H(r))$ has a natural MHS for each
$m\ge 0$ and $r\in \Z$ and that the natural isomorphism
$$
H^\dot(\cD,S^mH(r)) \cong \big[H^\dot(\n)\otimes S^m H(r)\big]^{\SL(H)}
$$
is an isomorphism of MHS. (Cf.\ \cite{hain:db-coho}.)

Since the monodromy representation factors through $\cG^\eis$, there is an
MHS isomorphism
$$
H_1(\n) \cong \bigoplus_{n>0} S^{2n}(\epsilon_{2n+2})
\cong \bigoplus_{n>0}S^{2n}H(2n+1).
$$
This implies that
$$
H^1\big(\cD,S^{2n}H(2n+1)\big) = \Hom_{\SL(H)}(H_1(\n),S^n H)(2n+1) \cong \Q(0).
$$
Regard $S^{2m}H = S^{2m}(\b^{2m})$. Let $\epsilondual_{2m}$ be the element of
$$
H^1(\cD,S^{2m}H) \cong \Hom_{\SL(H)}(H_1(\u),S^{2m}H)
$$
that takes the class of $\epsilon_{2m+2}$ to $\b^{2m}$.

The standard duality (Prop.~\ref{prop:cup}) between quadratic relations in $\n$
and the cup product $H^1(\n)\otimes H^1(\n)\to H^2(\n)$ gives the following dual
version of Pollack's quadratic relations, Theorem~\ref{thm:pollack}.

\begin{proposition}
\label{prop:dual_pollack}
There is a surjection
$$
H^2(\cD,S^{2n}H) \to H^1_\cusp(\M_{1,1},S^{2n}\H_\Q)^{\overline{\Fr}_\infty}
\otimes \Q(-2n-2)
$$
which is a morphism of MHS when
$H^1_\cusp(\M_{1,1},S^{2n}\H)^{\overline{\Fr}_\infty}$ is regarded as a HS of
type $(0,0)$. After tensoring with $\R$, it gives a surjection
\begin{equation}
\label{eqn:projn}
H^2(\cD,S^{2n}H_\R) \to \bigoplus_{f\in \B_{2n+2}}\R(-2n-2).
\end{equation}
The generator of $\R(-2n-2)$ corresponding to $f\in \B_{2n+2}$ will be denoted
by $z_f$. When $j+k=n$, the composition of the cup product
$$
H^1\big(\cD,S^{2j}H(2j+1)\big)\otimes H^1\big(\cD,S^{2k}H(2k+1)\big)
\to H^2\big(\cD,S^{2n}H(2n+2)\big)
$$
with the projection (\ref{eqn:projn}) is given by
\begin{equation}
\label{eqn:poll_cup}
\epsilondual_{2j} \otimes \epsilondual_{2k} \to
\sum_{f\in \B_{2n+2}} c_{j,k}(f)\, z_f 
\end{equation}
where $\br_f^+(\a,\b) = \sum_{j+k=n} c_{j,k}(f)\,\a^{2j}\b^{2k}$.
\end{proposition}

The reason that $V_f^-$ appears in $H^2$ and the coefficients of elements of
$V_f^+$ occur in the relations is explained by noting that the Petersson inner
product induces an isomorphism $V_f^-(V_f^+)^\ast$, while relations give
subspaces of $H_2$.

\begin{remark}
The occurrence of cusp forms here (without their associated Hodge structure)
should be related to, and may help explain, the appearance of modular and cusp
forms in the work of Conant, Kassabov and Vogtmann \cite{ckv1,ckv2,conant} on
the $\Sp(H)$-representation theory of the derivation algebra of a once punctured
Riemann surface $S$ of genus $g\gg 0$.
\end{remark}

\section{Deligne Cohomology and Extensions of VMHS}
\label{sec:db-coho}

The relations that hold in $\u^\eis_x$ are controlled by relations in the Yoneda
ext groups of the category $\MHS(\M_{1,1},\H)$. In this section, we sketch the
relationship between Deligne cohomology of the relative completion of the
fundamental group of an affine curve $C'=C-D$ and Yoneda ext groups in the
categories $\MHS(C',\H)$.  This generalizes the results of \cite{carlson-hain}
that hold in the unipotent case. We will work in the category of $\Q$-MHS,
although the discussion is equally valid in the category of $\R$-MHS. This
section is a summary of results from \cite{hain:db-coho} where full details can
be found.

\subsection{Deligne--Beilinson cohomology of a curve}
Let $\V$ be a PVHS over the affine (orbi) curve $C'=C-D$. The Deligne--Beilinson
cohomology $H^\dot_\cD(C',\V)$ is the cohomology of the complex
\begin{multline*}
\cone\big(F^0W_0\Dec_W K^\dot_\C(C,D;\V) \oplus W_0\Dec_W K^\dot_\Q(C,D;\V)
\cr
\to W_0\Dec_W K^\dot_\C(C,D;\V)\big)[-1].
\end{multline*}
Here $\Dec_W$ is  Deligne's {\em filtration decal\'ee} functor (with respect to
$W_\dot$), which is defined in \cite[\S1.3]{deligne:hodge2}. The DB-cohomology
fits in an exact sequence
\begin{equation}
\label{eqn:db-ses}
0 \to \Ext^1_\MHS\big(\Q,H^{j-1}(C',\V)\big) \to H_\cD^j(C',\V)
\to \Hom_\MHS\big(\Q,H^j(C',\V)\big) \to 0.
\end{equation}
Deligne--Beilinson cohomology of a higher dimensional variety $X=\Xbar-D$ with
coefficients in a PVHS $\V$ can be defined using Saito's mixed Hodge complex
that generalizes Zucker's.

The next result follows directly from the Manin-Drinfeld Theorem
(Thm.~\ref{thm:manin-drinfeld}) using the exact sequence (\ref{eqn:db-ses}).

\begin{proposition}
\label{prop:db-coho}
The DB-cohomology $H^j_\cD\big(\M_{1,1},S^m\H(r)\big)$ vanishes when $m$ is odd
and when $j>2$. When $j=1$ it vanishes except when $m=2n$ and $r=2n+1$, in which
case it is isomorphic to $\Q$. When $j=2$ 
\begin{multline*}
H^2_\cD\big(\M_{1,1},S^{2n}\H(r)\big)
= \Ext^1_\MHS\big(\Q,H^1(\M_{1,1},S^{2n}\H)(r)\big) \cr
\cong \Ext^1_\MHS\big(\Q,\Q(r-2n-1)\big)\oplus
\bigoplus_f \Ext^1_\MHS\big(\Q,M_f(r)\big)
\end{multline*}
where $f$ ranges over the equivalence classes of $f\in \B_{2n+2}$. \qed
\end{proposition}

As explained in Section~\ref{sec:frob}, the real Frobenius
$\Fr_\infty$ acts on $\M_{1,1}$ and on the local system $\H$. Since
$\Frbar_\infty$ preserves the Hodge filtration, it acts on the complex used to
compute $H^\dot_\cD\big(\M_{1,1},S^{2n}\H_\R(r)\big)$ and thus on the Deligne
cohomology itself.

\begin{corollary}
\label{cor:h2D}
For all $n > 0$ and $r\in \Z$, there are natural
$\overline{\Fr}_\infty$-equivariant isomorphisms
$$
H^2_\cD\big(\M_{1,1},S^{2n}\H_\R(r)\big)
\cong \Ext_\MHS^1\big(\R,\R(r-2n-1)\big) \oplus
\bigoplus_{f\in \B_{2n+2}}\Ext^1_\MHS\big(\R,V_f(r)\big).
$$
These vanish when $r<2n+2$. The first term on the right hand side corresponds to
the Eisenstein series $G_{2n+2}$. \qed
\end{corollary}

\subsection{Deligne--Beilinson cohomology of affine groups}

Here we recall the definition and basic properties of the Deligne--Beilinson
cohomology of an affine group from \cite{hain:db-coho}. Suppose that $G$ is an
affine $\Q$ group that is an extension of a reductive $\Q$ group $R$ by a
prounipotent group $U$. Suppose that the coordinate ring $\cO(G)$ and its sub
algebra $\cO(R)$ are Hopf algebras in the category ind-$\MHS$. Then the
coordinate ring $\cO(U)$ of $U$ is also a Hopf algebra in ind-$\MHS$. This
implies that its Lie algebra $\u$ is a pronilpotent Lie algebra in pro-$\MHS$.
Suppose that $V$ is a Hodge representation (Defn.~\ref{def:hrep}) of $G$. The
Deligne--Beilinson cohomology of $G$ is defined by
$$
H_\cD^\dot(G,V) := \Ext^\dot_{\HRep(G)}(\Q,V).
$$
This has a natural product: if $V_1,V_2$ are in $\HRep(G)$, there are natural
multiplication maps
$$
H_\cD^\dot(G,V_1)\otimes H_\cD^\dot(G,V_2) \to H_\cD^\dot(G,V_1\otimes V_2).
$$

The following result is proved in  \cite{hain:db-coho}.

\begin{proposition}
\label{prop:ses}
For all $j\ge 0$, there is a short exact sequence
$$
0 \to \Ext^1_\MHS\big(\Q,H^{j-1}(G,V)\big) \to H^j_\cD(G,V)
\to \Hom_\MHS\big(\Q,H^j(G,V)\big) \to 0. \qed
$$
\end{proposition}

\subsection{Extensions of VMHS over curves}

Suppose that $\H$ is a PVHS over $C'$. Fix a base point $x\in C'$ and let $R_x$
be the Zariski closure of the monodromy action $\pi_1(C',x) \to \Aut H_x$.
Denote the completion of $\pi_1(X,x)$ with respect to $\pi_1(C',x) \to R_x$ by
$\cG_x$. Theorem~\ref{thm:equivalence} implies that the category $\MHS(C',\H)$
is equivalent to $\HRep(\cG_x)$, so that there is a natural isomorphism
$$
\Ext_{\MHS(C',\H)}^\dot(\Q,\V) \cong H^\dot_\cD(\cG_x,V_x).
$$
Even more is true:

\begin{theorem}[\cite{hain:db-coho}]
\label{thm:del_ext}
If $\V$ is an object of $\MHS(C',\H)$, then there are natural isomorphisms
$$
\Ext_{\MHS(C',\H)}^\dot(\Q,\V) \cong H^\dot_\cD(\cG_x,V_x)
\overset{\simeq}{\longrightarrow} H^\dot_\cD(C',\V)
$$
where $V_x$ denotes the fiber of $\V$ over the base point $x$. These
isomorphisms are compatible with the natural products.\footnote{There is a more
general result which applies when $C'$ is replaced by a smooth variety $X$ of
arbitrary dimension. In that case, there is a natural homomorphism
$H^\dot_\cD(\cG_x,V_x) \to H^\dot_\cD(X,\V)$ that is an isomorphism in degrees
$\le 1$ and injective in degree 2. This is proved in \cite{hain:db-coho}.}
\qed
\end{theorem}

Denote the category of admissible variations of MHS over $C'$ by $\MHS(C')$.
Since the Deligne cohomology $H_\cD^\dot(C',\V)$ does not depend on the choice
of the basic variation $\H$ with $\V \in \MHS(C',\H)$, we have the following
useful result.

\begin{corollary}[\cite{hain:db-coho}]
\label{cor:ext-mhs}
For all admissible variations of MHS over $C'$, there is a natural isomorphism
$$
\Ext^\dot_{\MHS(C')}(\Q,\V) \cong H^\dot_\cD(C',\V)
$$
that is compatible with products.
\end{corollary}

\subsection{Extensions of variations of MHS over modular curves}
\label{sec:ext_vmhs}

Suppose that $\G$ is a congruence subgroup of $\SL_2(\Z)$. The following result
follows from Corollary~\ref{cor:ext-mhs} and the Manin-Drinfeld Theorem
(Thm.~\ref{thm:manin-drinfeld}).

\begin{proposition}
If $m>0$ and $A$ is a Hodge structure, then
$$
\Ext^1_{\MHS(X_\G)}(\Q,A\otimes S^m\H) \cong
\Hom_\MHS\big(\Q,A\otimes H^1(X_\G,S^m\H)\big).
$$
When $m=0$, $\Ext_{\MHS(X_\G)}^1(\Q,A) = \Ext_\MHS^1(\Q,A)$.
If $A$ is a simple $\Q$-HS, then
$
\Ext^1_{\MHS(X_\G)}(\Q,A\otimes S^m\H)
$
is non-zero if and only if either $A=\Q(m+1)$ or $A\cong M_f(m+1)$ for some
Hecke eigen cusp form $f\in \B_{m+2}(\G)$. \qed
\end{proposition}

This result can be interpreted as a computation of the group of normal functions
(tensored with $\Q$) over $X_\G$ associated to a PVHS of the form $\V=A\otimes
S^m\H$. These are holomorphic sections of the bundle of intermediate jacobians
associated to $\V$. The group of normal functions (essentially by definition) is
isomorphic to $\Ext^1_{\MHS(X_\G,\H)}(\Z,\V)$. The normal functions constructed
in Section~\ref{sec:cusp} generate all simple extensions and normal functions in
$\MHS(X_\G,\H)$.

\section{Cup Products and Relations in $\u^\eis$}
\label{sec:relations}

In this section we show that Pollack's quadratic relations lift to relations in
$\u^\eis$. In particular, we show that $\u^\eis$ is not free. Throughout, the
base point is $\vv = \partial/\partial q$, although most of the arguments are
valid with any base point. As before, $H$ denotes the fiber of $\Hbar$ over
$\vv$. In this setup, $\bw = 2\pi i\b$. We will omit the base point from the
notation.

Recall from Section~\ref{sec:pollack} that $\cD$ denotes the image of the
monodromy homomorphism $\cG \to \Aut \p$. As before, we take $S^{2n}H =
S^{2n}(\b^{2n})$. Let $\edual_{2n}$ be the element of
$$
H^1(\cG,S^{2n}H) \cong \Hom_{\SL(H)}\big(H_1(\u,S^{2n}H)\big)
$$
that takes the class of $\e_{2n+2}$ to $\b^{2n}$. Recall that the real Frobenius
$\overline{\Fr}_\infty$ acts on $H^\dot_\cD\big(\M_{1,1},S^{2n}\H_\R(r)\big)$.

\begin{lemma}
\label{lem:1_dimnl}
For all $m>0$, the homomorphisms $\cG \to \cG^\eis \to \cD$ induce isomorphisms
\begin{multline*}
\xymatrix@C=16pt{
H^1_\cD\big(\cD,S^{2n}\H(2n+1)\big) \ar[r]^(.47)\simeq &
H^1_\cD\big(\cG^\eis,S^{2n}\H(2n+1)\big)
}
\cr
\xymatrix@C=16pt{ & \ar[r]^(.15)\simeq &
H^1_\cD\big(\cG,S^{2n}\H(2n+1)\big) \ar[r]^(.44)\simeq &
H^1_\cD\big(\M_{1,1},S^{2n}\H(2n+1)\big)^{\overline{\Fr}_\infty}.
}
\end{multline*}
Each of these groups is a 1-dimensional $\Q$ vector space. The first is spanned
by $\epsilondual_{2n+2}$, the last by $\edual_{2n+2}/2\pi i$. The isomorphism
identifies $\epsilondual_{2n+2}$ with
$$
\frac{2}{(2n)!}\frac{\edual_{2n+2}}{2\pi i}.
$$
\end{lemma}

\begin{proof}
For all $n>0$, each of the groups
$$
H^1(\cD,S^{2n}H),\ H^1(\cG^\eis,S^{2n}H),\ H^1(\cG^\eis,S^{2n}H)
$$
is isomorphic to $\Q(-2n-1)$. The left-hand group is generated by
$\epsilondual_{2n+2}$ and the right two groups by $\edual_{2n+2}$.
Theorem~\ref{thm:images} implies that the projections $\cG \to \cG^\eis \to \cD$
take $\epsilondual_{2n+2}$ to $2\edual_{2n+2}/2\pi i(2n)!$, so that the
homomorphisms induced by the projections are isomorphisms.

Proposition~\ref{prop:ses} implies that the projections $\cG \to \cG^\eis\to
\cD$ induce isomorphisms
$$
\xymatrix{
H^1_\cD\big(\cG^\eis,S^{2n}\H(2n+1)\big) \ar[r]^(0.47)\simeq &
H^1_\cD\big(\cG,S^{2n}\H(2n+1)\big) \cong \Q
}
$$
for all $n>0$. The corresponding class in Deligne cohomology is easily seen to
be $\overline{\Fr}_\infty$ invariant.
\end{proof}

\begin{lemma}
There is a natural inclusion
$$
iH^1_\cusp(\M_{1,1},S^{2m}\H_\R)^{\overline{\Fr}_\infty} \hookrightarrow
H^2_\cD\big(\cD,S^{2m}H_\R(2m+2)\big)
$$
\end{lemma}

\begin{proof}
Since $\u$ is free, $H^2(\cG,S^mH)\cong H^2(\u,S^m H)^{\SL(H)}$
vanishes for all $m>0$. The result follows from Proposition~\ref{prop:ses},
the computation (Thm.~\ref{thm:admissible}) of $H^1(\u)$ and the
isomorphism
\begin{multline*}
\Ext^1_\MHS\big(\R,
H^1_\cusp(\M_{1,1},S^{2m}\H(2m+2))\big)^{\overline{\Fr}_\infty}\cr
\cong iH^1_\cusp(\M_{1,1},S^{2m}\H_\R)^{\overline{\Fr}_\infty}
= \bigoplus_{f\in \B_{2m+2}} iV_f^-.
\end{multline*}
which is well defined up to an even power of $2\pi i$ that depends upon the
choice of the first isomorphism.
\end{proof}

By Corollary~\ref{cor:h2D}, there is an $\overline{\Fr}_\infty$ invariant
projection
$$
H^2_\cD\big(\M_{1,1},S^{2n}\H_\R(2n+2)\big) \to
\Ext^1_\MHS\big(\R,V_f(2n+2)\big).
$$
The following computation is the key to proving that Pollack's quadratic
relations are motivic.

\begin{theorem}[{Brown \cite[Thm.~11.1]{brown},
Terasoma \cite[Thm.~7.3]{terasoma}}]
\label{thm:cup}
If $j,k>0$ and $n=j+k$ and if there is a cup form of weight $2n+2$, then the
image of the cup product
\begin{multline*}
H^1_\cD\big(\M_{1,1},S^{2j}\H_\R(2j+1)\big) \otimes
H^1_\cD\big(\M_{1,1},S^{2k}\H_\R(2k+1)\big)
\cr
\to H^2_\cD\big(\M_{1,1},S^{2n}\H_\R(2n+2)\big)
\end{multline*}
is non-zero. More precisely, the composition of the cup product with the
projection
$$
H^2_\cD\big(\M_{1,1},S^{2n}\H_\R(2n+2)\big)
\to \Ext^1_\MHS\big(\R,V_f(2n+2)\big)^{\overline{\Fr}_\infty} \cong V_{f,\R}^-
$$
is non-trivial for all $f\in \B_{2n+2}$. \qed
\end{theorem}

As we shall show below, a direct consequence is that Pollack's quadratic
relations hold in $\u^\eis$. Brown's period computations \cite{brown} are more
detailed and imply that all of Pollack's relations lift from a quotient of
$\Der^0 \p$ to relations in $\u^\eis$ and are therefore motivic. Full details
will appear in \cite{hain-matsumoto:mem}.

\begin{theorem}
\label{thm:quad}
Pollack's quadratic relations (\ref{eqn:ext_map}) hold in $\Gr^W_\dot\u^\eis$.
In particular, the pronilpotent radical $\u^\eis$ of $\g^\eis$ is not free.
\end{theorem}

We use the notation $\G V := \Hom_\MHS(\Q,V)$ to denote the set of Hodge classes
of type $(0,0)$ of a MHS $V$.

\begin{proof}
Suppose that $n>0$. Since the Hodge structure
$H^1_\cusp\big(\M_{1,1},S^{2n}\H(2n+2)\big)$ does not occur in $H^1(\u^\eis)$,
Proposition~\ref{prop:ses} and the Manin-Drinfeld Theorem imply that the
homomorphism
\begin{align*}
H^2_\cD\big(\cG^\eis,S^{2n}\H(2n+2)\big) &\to 
H^2_\cD\big(\cG,S^{2n}\H(2n+2)\big) \cr
&\cong \Ext^1_\MHS\big(\Q,H^1(\M_{1,1},S^{2n}\H(2n+2))\big)
\cr
&\to \Ext^1_\MHS\big(\Q,H^1_\cusp(\M_{1,1},S^{2n}\H(2n+2))\big)
\end{align*}
induces a well-defined map
$$
r : \G H^2(\cG^\eis,S^{2n}\H(2n+2)) \to
\Ext^1_\MHS\big(\Q,H^1_\cusp(\M_{1,1},S^{2n}\H(2n+2))\big).
$$
Compose this with the projection to
$$
\Ext^1_\MHS\big(\R,H^1_\cusp(\M_{1,1},S^{2n}\H(2n+2))\big)
\cong \bigoplus_{f\in \B_{2n+2}} V_{f,\R}
$$
to obtain a projection
$$
p : \G H^2(\cG^\eis,S^{2n}\H(2n+2)) \to  \bigoplus_{f\in \B_{2n+2}} V_{f,\R}.
$$
Consider the diagram
{\tiny
$$
\xymatrix@C=16pt{
H^2_\cD\big(\cD,S^{2n}H(2n+2)\big) \ar@{->>}[d] \ar[r] &
H^2_\cD\big(\cG^\eis,S^{2n}H(2n+2)\big) \ar@{->>}[d] \ar[r] &
H^2_\cD\big(\cG,S^{2n}\H(2n+2)\big) \ar@{->>}[d]
\cr
\G H^2(\cD,S^{2n}H(2n+2)\big) \ar[r] \ar[dr]_{p_\cD} &
\G H^2(\cG^\eis,S^{2n}H(2n+2)\big) \ar[r]^(.4)r \ar[d]^{p} &
\Ext^1_\MHS\big(\Q,H^1_\cusp(\M_{1,1},S^{2n}\H(2n+2))\big) \ar[dl]^{p_\cG}
\cr
& \bigoplus_{f \in \B_{2n+2}} V_{f,\R}
}
$$
}where $p_\cD$ is the projection dual to the Pollack relations
(Prop.~\ref{prop:dual_pollack}), $p$ is the projection constructed above and
$p_\cG$ is the standard projection. Naturality implies that the top left
square commutes; the right hand square commutes by the construction of $r$;
the bottom right triangle commutes by the definition of $p$. The cup product
computation (Thm.~\ref{thm:cup}) implies that the sum of the two triangles
commute. It follows that the diagram commutes.

Suppose that $j,k>0$ satisfy $n=j+k$. Lemma~\ref{lem:1_dimnl} and
Proposition~\ref{prop:dual_pollack} now imply that the composite of
$$
H^1\big(\cG^\eis,S^{2j}H(2j+1)\big) \otimes
H^1\big(\cG^\eis,S^{2k},S^{2k}H(2j+1)\big)
\to \G H^2\big(\cG^\eis,S^{2n}H(2n+2)\big)
$$
with the projection $p$ is given by the formula (\ref{eqn:poll_cup}).
The result now follows from the duality between cup product and quadratic
relations.
\end{proof}

Much of the discussion in this section can be generalized to modular curves
$X_\G$ where $\G$ is a congruence subgroup of $\SL_2(\Z)$. In particular, the
prounipotent radical $\u_\G^\eis$ is not free for all congruence subgroups.

\begin{remark}
\label{rem:manin}
This result implies that, when $\G=\SL_2(\Z)$, Manin's quotient $\u_B$ of
$\u_\G$ (cf.\ \ref{sec:manin}) is not a quotient of $\u_\G$ in the category of
Lie algebras with a MHS. If it were, it would $\u^\eis_\G$. But since $\u_B$ is
free, and since $\u_\G^\eis$ is not, $\U_B \to \U_\G^\eis$ cannot be an
isomorphism.
\end{remark}

\end{document}